\documentclass[10pt]{article}

\usepackage{geometry}
\usepackage{amsmath}
\usepackage{mathdots}
\usepackage{amssymb}
\usepackage{mathrsfs}
\usepackage{MnSymbol}
\usepackage{graphicx}
\usepackage{float}
\usepackage{amsthm}
\usepackage{tikz-cd}
\usepackage{dynkin-diagrams}
\usepackage[pagebackref]{hyperref}
\usepackage{tikz}
\usepackage[utf8]{inputenc}
\usepackage{multirow}
\usepackage{tikz-cd}

\newcommand{\codim}{\textrm{codim}}
\newcommand{\im}{\textrm{im }}
\newcommand{\coker}{\textrm{coker }}

\newcommand{\pdiff}[2]{\frac{\partial #1}{\partial #2}}

\def\Gr{\operatorname{Gr}}

\def \RatCurves{\operatorname{RatCurves}}
\def \Ann{\operatorname{Ann}}

\def \Pol{\operatorname{Pol}}
\def \Stab{\operatorname{Stab}}

\def \locus{\operatorname{locus}}
\def \Sm{\operatorname{Sm}}
\def \Hom{\operatorname{Hom}}

\def \Ad{\operatorname{Ad}}
\def \ad{\operatorname{ad}}
\def \res{\operatorname{res}}
\def \rank{\operatorname{rank}}
\def \Lie{\operatorname{Lie}}
\def \Leg{\operatorname{Leg}}

\def \VMRT{\operatorname{VMRT}}
\def \can{\operatorname{can}}

\def \split{\operatorname{split}}
\def \surj{\operatorname{surj}}
\def \Hilb{\operatorname{Hilb}}

\def \O{\operatorname{O}}
\def \SO{\operatorname{SO}}
\def \SL{\operatorname{SL}}
\def \Sp{\operatorname{Sp}}
\def \GL{\operatorname{GL}}
\def \diag{\operatorname{diag}}
\def \SI{\operatorname{S}\mathbb{I}}
\def \CI{\operatorname{C}\mathbb{I}}

\newtheorem{theorem}{Theorem}[section]
\newtheorem{corollary}[theorem]{Corollary}
\newtheorem{lemma}[theorem]{Lemma}
\newtheorem{proposition}[theorem]{Proposition}
\newtheorem{claim}[theorem]{Claim}

\newtheorem*{theorem*}{Theorem}
\newtheorem*{maintheorem*}{Main Theorem}
\newtheorem*{corollary*}{Corollary}
\newtheorem*{lemma*}{Lemma}
\newtheorem*{proposition*}{Proposition}
\newtheorem*{claim*}{Claim}

\theoremstyle{definition}
\newtheorem{definition}[theorem]{Definition}

\newtheorem{assumption}[theorem]{Assumption}

\newtheorem{prob}{Problem}
\newtheorem*{prob*}{Problem}

\newtheorem*{question*}{Question}

\newtheorem{example}[theorem]{Example}
\newtheorem*{example*}{Example}

\newtheorem*{exer*}{Exercise}

\newtheorem*{acknowledgements*}{Acknowledgements}

\newtheorem{remark}[theorem]{Remark}
\newtheorem*{remark*}{Remark}
\newtheorem*{assump*}{Assumption}
\newtheorem*{notation*}{Notations}

\def \CC {\mathbb{C}}

\def \II {\mathbb{I}}
\def \JJ {\mathbb{J}}

\def \PP {\mathbb{P}}

\def \ZZ {\mathbb{Z}}

\def \Ocal {\mathcal{O}}

\def \gfr {\mathfrak{g}}
\def \hfr {\mathfrak{h}}

\def \sfr {\mathfrak{s}}

\def \ufr {\mathfrak{u}}

\def \Hscr {\mathscr{H}}

\def \Kscr {\mathscr{K}}

\def \hbar {\bar{h}}

\def \cbf {\mathbf{c}}
\def \dbf {\mathbf{d}}

\def \rbf {\mathbf{r}}

\title{Legendrian families of lines on nilpotent orbit closures}
\author{Minseong Kwon\thanks{M.~Kwon was supported by Morningside Center of Mathematics, Chinese Academy of Sciences, and by Beijing Postdoctoral Research Foundation.}}

\date{September 4, 2026}

\begin{document}
\maketitle

\begin{abstract}
    Let $\gfr$ be a complex semisimple Lie algebra, and $\overline{Z}$ a nilpotent orbit closure in the projectivization $\PP(\gfr)$.
    We investigate the space of tangent directions of lines on $\overline{Z}$ passing through a general point $z$, denoted by $F(\overline{Z},\,z)$.
    We first prove that every irreducible component of $F(\overline{Z},\,z)$ is an integral subvariety of the contact hyperplane in the projectivized tangent space.
    Next, we study Legendrian components of $F(\overline{Z},\,z)$ in two cases: the case where $\overline{Z}$ admits a Springer resolution; the case where $\overline{Z}$ is square-zero in a projectivized simple Lie algebra of classical type.
    As an application, two corollaries are presented: a characterization of Richardson orbits among square-zero orbits in terms of $F(\overline{Z},\,z)$; a description of $F(\overline{Z},\,z)$ for $\overline{Z}$ arising from stratified Mukai flops associated to irreducible Hermitian symmetric spaces.
\end{abstract}

\emph{Keywords}: nilpotent orbit, line, contact structure, Legendrian variety.

\emph{2020 Mathematics Subject Classification}:
17B08, 
14H45, 
58A30. 

\tableofcontents

\section{Introduction}

\begin{notation*}
We work in the category of complex algebraic varieties, except for Section~\ref{section:Preliminaries} where we consider the holomorphic category.
A variety means a reduced seperated scheme of finite type over $\CC$, not assumed to be irreducible.
For the projectivization, we use the geometric convention: $\PP(V) \coloneqq (V \setminus 0) / \CC^{\times}$.
\end{notation*}

Rational curves play a prominent role in the study of Fano varieties.
Among them, lines are of particular importance, and their tangent directions encode rich information on the geometry of a given Fano variety.
See, for example, \cite{Hwang01VMRT}, 

A notable class of Fano varieties comes from \emph{nilpotent orbit closures} in projectivized semisimple Lie algebras.
Here, for a semisimple Lie algebra $\gfr$, a \emph{nilpotent orbit} in $\PP(\gfr)$ means an adjoint orbit of nilpotent elements.
A nilpotent orbit $Z$ in $\PP(\gfr)$ admits a natural holomorphic contact structure $D$ \cite{Beauville1998FanoContact}, in the sense that $D$ is a hyperplane distribution in $T_{Z}$ such that the Lie bracket of vector fields defines a nondegenerate bundle morphism
\[
    D \wedge D \rightarrow T_{Z}/D.
\]
Thus the closure $\overline{Z} \subset \PP(\gfr)$ is a singular Fano contact variety up to normalization, and moreover, $\overline{Z}$ is covered by lines with only one exception (cf. Corollary~\ref{coro:covered by lines except for adj type C}).
Hence it is natural to study lines on nilpotent orbit closures $\overline{Z}$.

In this paper, we investigate the space of tangent directions of lines on $\overline{Z}$ through a given point $z \in Z$:
\[
    F(\overline{Z},\,z) \coloneqq \{[T_{l,\,z}] \in \PP(T_{Z,\,z}) : \text{$l$ is a line on $\overline{Z}$ passing through $z$}\}.
\]
Notice that $F(\overline{Z},\,z)$ is well defined since $Z$ is an open dense orbit in $\overline{Z}$.
First we discuss which properties $F(\overline{Z},\,z)$ shares with the spaces of lines on Fano contact manifolds, despite singularities of $\overline{Z}$.
Then we present various examples, showing new phenomena which do not occur in the smooth case.

\subsection{Motivation}

The geometry of lines on $\overline{Z}$ is best understood in the case where $Z$ is an \emph{adjoint variety}, in which case $Z$ is projective, i.e., $Z = \overline{Z}$ (cf. Definition~\ref{defn: adjoint var}).
The adjoint varieties are also known as homogeneous compact contact manifolds classified in \cite{Boothby61}.
In this case, $F(Z,\,z)=F(\overline{Z},\,z)$ is known as the \emph{subadjoint variety} in the literature, and its description is classically well known \cite[Table~I]{Buczynski2006LegendrianSubvarieties}.
The subadjoint variety has played an important role in the study of the adjoint varieties, see for example \cite{Hwang1997RigidityHomogeneous} for deformation rigidity and \cite{Ho00} for a characterization theorem.

A key feature of the subadjoint variety is that it is a Legendrian submanifold of an odd dimensional projective space.
More precisely, for the contact structure $D$ of $Z$, the fiber $D_{z}$ is equipped with a natural conformal symplectic structure, which induces a contact structure of $\PP(D_{z})$ (cf. Example~\ref{ex:P 2n+1}).
Then the subadjoint variety $F(Z,\,z)$ is contained in $\PP(D_{z})$ as a Legendrian submanifold, meaning that it is smooth, integral (i.e., everywhere tangent to the contact structure of $\PP(D_{z})$) and of maximal dimension among integral submanifolds (cf. Definition~\ref{defn:integral/Legendrian submflds}).
In fact, it is well known that the Legendrian property holds true in a more general setup of contact manifolds: we refer to \cite{Ke01}, \cite[\S2]{Hwang01VMRT} and \cite{Hw23Partial} for details.

Meanwhile, when $\overline{Z}$ is a nilpotent orbit closure other than the adjoint varieties, $\overline{Z}$ is never smooth, and thus the theory of contact manifolds does not directly apply.
Therefore it is not obvious whether $F(\overline{Z},\,z)$ is Legendrian, or even contained in the contact hyperplane $\PP(D_{z})$, and so the following questions naturally arise.
In what follows, we say that a subvariety of a smooth contact variety is \emph{integral} (resp. \emph{Legendrian}) if its smooth locus is integral (resp. Legendrian).
\begin{question*}
    Let $\overline{Z}$ be a nilpotent orbit closure in the projectivized semisimple Lie algebra $\PP(\gfr)$, and let $z$ be a point in the open dense orbit $Z$.
    \begin{enumerate}
        \item\label{question1} Is $F(\overline{Z},\,z)$ an integral subvariety of the contact hyperplane $\PP(D_{z})$?
        \item\label{question2} If so, is $F(\overline{Z},\,z)$, or any of its irreducible components Legendrian?
    \end{enumerate}
\end{question*}
From our results, it will follow that the answer to Question~\ref{question1} is affirmative, while the answer to Question~\ref{question2} is not.

\subsection{Main results}

Before stating our main results, let us note that $F(\overline{Z},\,z)$ is nonempty, that is, $\overline{Z}$ is covered by lines, unless $\overline{Z}$ is the adjoint variety of type $C$ (cf. Corollary~\ref{coro:covered by lines except for adj type C}).

Our first main theorem gives a positive answer to Question~\ref{question1}.

\begin{theorem}[Theorem~\ref{thm:VMRT is integral}] \label{main thm:integral}
    Given a nilpotent orbit $Z$ and a point $z \in Z$, $F(\overline{Z},\,z)$ is contained in the contact hyperplane $\PP(D_{z})$ as an integral subvariety, whenever $F(\overline{Z},\,z)$ is nonempty.
\end{theorem}

As mentioned before, the main obstacle to Theorem~\ref{main thm:integral} is singularities of $\overline{Z}$, making the standard deformation theory on contact manifolds unavailable.
To overcome this problem, we use a nice resolution of singularities of $\overline{Z}$, called the \emph{Jacobson--Morozov resolution} (cf. Section~\ref{ssection:JM resoln}).
The key is that, although the Jacobson--Morozov resolution does not admit a contact structure (and such a {\lq contact\rq} resolution does not exist in general \cite{fu07contact}), the contact form of $Z$ naturally extends to a 1-form valued in a line bundle (cf. Proposition~\ref{prop:Theta KKS}).
This enables us to use deformation-theoretic arguments from the smooth setup, namely results of \cite{Hw23Partial}.
In the same spirit, as an analogue of \cite{Ke01}, we also prove another integrality theorem: the locus of lines through $z$ is an integral subvariety of $Z$.
See Theorem~\ref{thm:locus of lines is integral} for the precise statement.

Theorem~\ref{main thm:integral} leads us to the notion of \emph{Legendrian components} of $F(\overline{Z},\,z)$, i.e., irreducible components of $F(\overline{Z},\,z)$ that are Legendrian in the contact hyperplane.
We denote by $\Leg(\overline{Z},\,z)$ the set of all Legendrian components of $F(\overline{Z},\,z)$, and then Question~\ref{question2} asks if $\Leg(\overline{Z},\,z)$ is nonempty.
We show that $\Leg(\overline{Z},\,z)$ is sometimes empty, by computing $\Leg(\overline{Z},\,z)$ for \emph{square-zero} nilpotent orbits $Z$.
Here, we say that $Z$ is \emph{square-zero} if $\gfr$ is a simple Lie algebra of classical type and every point of $Z$ is represented by a square-zero matrix, or in other words, it is of Jordan type $[2^{r},\,1^{N-2r}]$ (cf. Section~\ref{section:square zero}).
Our result shows that there are square-zero counterexamples to Question~\ref{question2} when $\gfr$ is of Dynkin type $B$ or $C$:

\begin{theorem}[Corollaries~\ref{coro:Leg of square zero in slN} and \ref{coro:Leg of square zero in soN spN}] \label{main thm:square zero}
    Let $\gfr$ be a simple Lie algebra of classical type, $Z$ a square-zero nilpotent orbit in $\PP(\gfr)$, and $z \in Z$ a point.
    \begin{enumerate}
        \item If $\gfr = \mathfrak{sl}_{N}$ and $Z$ is of Jordan type $[2^{r},\,1^{N-2r}]$ ($N \ge 2$, $1 \le r \le N/2$), then $F(\overline{Z},\,z)$ is nonempty exactly when $N \ge 3$, and in this case,
        \[
            |\Leg(\overline{Z},\,z)| = \left\{ \begin{array}{cc}
                1 + \lfloor r/2 \rfloor & \text{if $N=2r$,} \\
                1+r & \text{otherwise}.
            \end{array} \right.
        \]
        \item If $\gfr = \mathfrak{so}_{N}$ and $Z$ is of Jordan type $[2^{r},\,1^{N-2r}]$ ($N \ge 5$, $2 \le r \le N/2$, $r$ even), then $F(\overline{Z},\,z)$ is nonempty exactly when $N\ge 6$, and in this case,
        \[
            |\Leg(\overline{Z},\,z)| = \left\{ \begin{array}{cc}
                1 + \lfloor r/4 \rfloor & \text{if $N=2r$,} \\
                0 & \text{if $N=2r+1$ and $r/2$ is odd,} \\
                1 + (r/2) & \text{if $N=2r+2$,} \\
                1 & \text{otherwise}.
            \end{array} \right.
        \]
        \item If $\gfr = \mathfrak{sp}_{N}$ and $Z$ is of Jordan type $[2^{r},\,1^{N-2r}]$ ($N \ge 4$, $N$ even, $1 \le r \le N/2$), then $F(\overline{Z},\,z)$ is nonempty exactly when $r \ge 2$, and in this case,
        \[
            |\Leg(\overline{Z},\,z)| = \left\{ \begin{array}{cc}
                1 + \lfloor r/2 \rfloor & \text{if $N=2r$ and $r$ is odd,} \\
                2 + (r/2) & \text{if $N=2r$ and $r$ is even,} \\
                2 & \text{if $N\ge 2r+2$ and $r$ is even,} \\
                0 & \text{otherwise}.
            \end{array} \right.
        \]
    \end{enumerate}
\end{theorem}

In fact, our result is stronger than what is stated in Theorem~\ref{main thm:square zero}: we also determine Legendrian components of $F(\overline{Z},\,z)$ up to the isotropy action of $\Stab_{G}(z)$ and the size of the set $\Leg(\overline{Z},\,z)/\Stab_{G}(z)$ of orbits where $G$ is the adjoint group of $\gfr$.
See Corollaries~\ref{coro:Leg of square zero in slN} and \ref{coro:Leg of square zero in soN spN} for details.

Since the class of square-zero nilpotent orbits includes various kinds of examples, our result provides a new insight beyond Questions~\ref{question1} and \ref{question2}.
A remarkable example is the case where $Z$ is \emph{Richardson}, that is, $\overline{Z}$ is the moment image of the projectivized cotangent bundle over a rational homogeneous space (cf. Definition~\ref{defn:Richardson}).
A characterization of Richardson orbits in terms of $F(\overline{Z},\,z)$ is known: a nilpotent orbit $Z$ is Richardson if and only if $F(\overline{Z},\,z)$ contains a linear Legendrian component \cite[Theorem~3.2]{PY19Nilpotent}.
From Theorem~\ref{main thm:square zero}, together with the well-known classification of Richardson orbits (see Example~\ref{ex:sq-zero Richardson}), we deduce another characterization of Richardson orbits among square-zero orbits:

\begin{theorem}[Corollary of Corollaries~\ref{coro:Leg of square zero in slN} and \ref{coro:Leg of square zero in soN spN}]\label{main thm: sq zero Richardson}
    Let $\gfr$ be a simple Lie algebra of classical type and $Z$ a square-zero nilpotent orbit in $\PP(\gfr)$.
    Assume that $F(\overline{Z},\,z)$ is nonempty.
    Then $Z$ is Richardson if and only if either
    \[
        |\Leg(\overline{Z},\,z)/\Stab_{G}(z)| \ge 2,
    \]
    or $\overline{Z} \subset \PP(\gfr)$ is the image of the moment map of $\PP(T^{\vee}_{\PP^{N-1}})$ defined by the $\Sp_{N}$-action on $\PP^{N-1}$ ($N\ge 6$ even), that is,
    \[
        \PP(\overline{O}_{[2^{2},\,1^{N-4}]}) \subset \PP(\mathfrak{sp}_{N}).
    \]
    In the last case, $|\Leg(\overline{Z},\,z)/\Stab_{G}(z)| = 1$ and $F(\overline{Z},\,z)$ contains two linear Legendrian components, which are identified under the $\Stab_{G}(z)$-action.
\end{theorem}

Among square-zero Richardson orbits, of particular interest are the Richardson orbits arising from \emph{stratified Mukai flops} of classical type.
Loosely speaking, \emph{stratified Mukai flops} are building blocks of birational maps between the projectivized cotangent bundles $\PP(T^{\vee}_{G/P_{i}})$ over two rational homogeneous spaces $G/P_{i}$ ($i=1,\,2$) such that the moment maps are birational onto the same nilpotent orbit closure $\overline{Z}$:
\[
    \begin{tikzcd}
        \PP(T^{\vee}_{G/P_{1}}) \arrow[rd] \arrow[d] \arrow[rr, dotted]&  & \PP(T^{\vee}_{G/P_{2}}) \arrow[ld] \arrow[d]\\
        G/P_{1} & \overline{Z} & G/P_{2}.
    \end{tikzcd}
\]
More precisely, stratified Mukai flops are also birational maps of the above form, and according to \cite[Theorem~6.1]{Na06} and \cite[Corollary~5.9]{Fu07extremal}, such a birational map can be connected by stratified Mukai flops.
There are four types of stratified Mukai flops: type $A$, $D$, $E_{6,\,I}$ and $E_{6,\,II}$ (cf. Table~\ref{tab:Stratified Mukai flops}, Section~\ref{ssection:Stratified Mukai flops}).
Except for the last type, $G/P_{i}$ are irreducible Hermitian symmetric spaces such that $G/P_{1} \simeq G/P_{2}$ (not equivariantly).
Moreover, in type $A$ and $D$, the associated nilpotent orbit closure $\overline{Z}$ is square-zero, and so Theorem~\ref{main thm:square zero} applies.
We extend this result to the $E_{6,\,I}$-case:
\begin{theorem}[Corollary of Corollaries~\ref{coro: Mukai flop E6I}, \ref{coro:Leg of square zero in slN} and \ref{coro:Leg of square zero in soN spN}] \label{main thm:Mukai flop}
    For the stratified Mukai flop of type $A$, $D$ or $E_{6,\,I}$, let $X \simeq G/P_{1} \simeq G/P_{2}$ be the associated irreducible Hermitian symmetric space, and $\overline{Z}$ the associated Richardson orbit closure.
    Then we have
    \[
        |\Leg(\overline{Z},\,z)| = 1 + \rank(X)
    \]
    where $\rank(X)$ is the rank of the symmetric space $X$ (cf. Remark~\ref{rmk:rank of IHSS}).
\end{theorem}

The $E_{6,\,I}$-case of Theorem~\ref{main thm:Mukai flop} follows from, rather than explicit computation, our study on $F(\overline{Z},\,z)$ in a more general setup of Richardson orbit closures.
The idea is to study lines on $\overline{Z}$ using rational curves on the projectivized cotangent bundles over rational homogeneous spaces.
For this purpose, we investigate rational curves on the projectivized cotangent bundles over quasi-projective manifolds, in which case the standard deformation theory applies (cf. Section~\ref{ssection:projectivized cotangent bundles}).
In this setup, we discuss when a rational curve on the base manifold can be lifted to a \emph{contact line} (cf. Definition~\ref{defn:contact line}) on the projectivized cotangent bundle and then compare their parameter spaces, which has its own interest (cf. Theorem~\ref{thm:VMRT formulation}).
For instance, a result which is crucial in the proof of Theorem~\ref{main thm:Mukai flop} is the following:

\begin{theorem}[Corollary~\ref{coro:irreducible comp over VMRT of lines}]
    Let $G/P$ be a rational homogeneous space, and $\pi : \PP(T^{\vee}_{G/P})\rightarrow G/P$ the projection.
    Assume that the moment map $\PP(T^{\vee}_{G/P}) \rightarrow \PP(\gfr)$ is birational onto its image $\overline{Z}$.
    Choose a point $z \in Z$, and for a line $l$ on $\overline{Z}$ passing through $z$, denote by $l^{\sim}$ the strict transform of $l$.
    Let $\Kscr$ be a family of lines on $G/P$ with respect to a polarization such that the anticanonical degree $\deg_{\Kscr} (K_{G/P}^{\vee})$ is at least $4$.
    Then lines $l$ on $\overline{Z}$ through $z$ such that $\pi(l^{\sim})$ are members of $\Kscr$ form a single Legendrian component of $F(\overline{Z},\,z)$.
\end{theorem}

Observe that $\deg_{\Kscr} (K_{G/P}^{\vee})$ is always at least 2.
The case where $\deg_{\Kscr} (K_{G/P}^{\vee})$ is $2$ or $3$ is discussed in Theorem~\ref{thm:VMRT formulation for G/P} and Remark~\ref{rmk:VMRT of G/P}.

\subsection{Further questions}

The study of lines on nilpotent orbit closures is still in the early stage, and based on our results, many natural follow-up questions can be raised.
For future reference, we record some of them.

\begin{prob}\label{prob:number of Leg}
    Can Theorem~\ref{main thm: sq zero Richardson} be generalized to a suitable larger class of nilpotent orbits?
\end{prob}
In other words, Problem~\ref{prob:number of Leg} asks to what extent Richardson orbits can be characterized in terms of $|\Leg(\overline{Z},\,z)|$ or $|\Leg(\overline{Z},\,z)/\Stab_{G}(z)|$.
Notice that if one considers all nilpotent orbits, then such a characterization may not exist: for example, if $\overline{Z}$ is the locus of nilpotent elements in $\PP(\mathfrak{sl}_{2} \oplus \cdots \oplus \mathfrak{sl}_{2})$ ($n+1$ times, $n \ge 1$), then it is the Richardson orbit closure associated to the homogeneous space $\PP^{1} \times \cdots \times \PP^{1} = (\SL_{2} \times \cdots  \times \SL_{2})/B$ ($n+1$ times), and $F(\overline{Z},\,z)$ is just a (Legendrian) linear space of dimension $n-1$.

\begin{prob}
    When $\Leg(\overline{Z},\,z)$ is (non)empty?
\end{prob}

This is a refined version of Question~\ref{question2}.
Several examples with $F(\overline{Z},\,z) \not=\emptyset$ but $\Leg(\overline{Z},\,z) = \emptyset$ are given in Theorem~\ref{main thm:square zero}.
Two sufficient conditions for $\Leg(\overline{Z},\,z)$ being nonempty are well known: one is the Richardson condition, and the other is the existence of a line contained in $Z$, not only in $\overline{Z}$ (cf. Theorem~\ref{thm:covering family is free}).
Note that none of them is a necessary condition: non-Richardson $Z$ with $\Leg(\overline{Z},\,z)\not=\emptyset$ can be found in Theorem~\ref{main thm:square zero}, and in the example given below Problem~\ref{prob:number of Leg}, no line on $\overline{Z}$ is contained in $Z$.

\begin{prob} \label{prob: characterization}
    Which Legendrian subvarieties of the odd dimensional projective spaces can be Legendrian components of $F(\overline{Z},\,z)$?
\end{prob}

Recall that the subadjoint varieties can be characterized as homogeneous Legendrian subvarieties of the odd dimensional projective spaces \cite{LandsbergManivel2007LegendrianVarieties}, \cite{Buczynski2006LegendrianSubvarieties}, \cite{Kw25}.
Problem~\ref{prob: characterization} seeks a similar characterization for Legendrian components of $F(\overline{Z},\,z)$.
The last problem also can be viewed as a variation of the problem solved in \cite{Hw23Partial}, where varieties consisting of tangent directions of minimal rational curves on complex manifolds are investigated.
The point of Problem~\ref{prob: characterization} is that we focus on Legendrian varieties defined by nilpotent orbit closures.

\subsection{Organization}

In Section~\ref{section:Preliminaries}, we recall well-known facts in contact geometry.
Namely in Section~\ref{ssection:contact lines}, we summarize properties of contact lines on contact manifolds.

In Section~\ref{section:Integrality theorems}, we prove Theorem~\ref{main thm:integral}.
For this purpose, we recall the construction of the Jacobson--Morozov resolution in Section~\ref{ssection:JM resoln}.
The proofs of the integrality theorems \ref{thm:VMRT is integral} and \ref{thm:locus of lines is integral} are presented in Section~\ref{ssection:Behavior of lines}.

Section~\ref{section:Springer resoln} is devoted to the case of Richardson orbit closures.
Here we mainly focus on those with the Springer resolutions, i.e., the case where the moment map from the projectivized cotangent bundle is birational (cf. Assumption~\ref{assumption:Springer resolution}).
Most results follow from the study of contact lines on the projectivized cotangent bundles over quasi-projective manifolds, which is developped in Section~\ref{ssection:projectivized cotangent bundles}.
After specializing the results to Richardson orbit closures in Section~\ref{ssection:Richardson}, we prove the $E_{6,\,I}$-case of Theorem~\ref{main thm:Mukai flop} in Section~\ref{ssection:Stratified Mukai flops}.

The computation for the square-zero cases (Theorem~\ref{main thm:square zero}) is quite long and postponed to the last Section~\ref{section:square zero}.
A sketch of the proof can be found in Section~\ref{ssection:idea of proof}.
A summary of the results is given in Corollary~\ref{coro:Leg of square zero in slN}, Section~\ref{ssection::slN} for the case of $\mathfrak{sl}_{N}$, and in Corollary~\ref{coro:Leg of square zero in soN spN}, Section~\ref{ssection::soN spN} for the others.

\subsection*{Acknowledgements}

The author would like to thank Baohua Fu for helpful discussions and suggestions during this project.
The author is grateful to Jun-Muk Hwang for sharing his idea regarding the contents of Section~\ref{ssection:projectivized cotangent bundles}, especially Theorem~\ref{thm:VMRT formulation}.
This work was supported by Morningside Center of Mathematics, Chinese Academy of Sciences, and by Beijing Postdoctoral Research Foundation.

\section{Preliminaries} \label{section:Preliminaries}
\subsection{Contact structures}
In this section, we recall basic notions in holomorphic contact geometry.
To this end, we work in the holomorphic category.
Namely, vector bundles on complex manifolds and morphisms between them are assumed to be holomorphic.

\begin{definition} \label{defn:contact}
    Let $Z$ be a complex manifold, and $T_{Z}$ the holomorphic tangent bundle.
    Let $D$ be a subbundle of $T_{Z}$ of corank one.
    \begin{enumerate}
        \item For a point $z \in Z$, the \emph{Levi tensor} of $D$ at $z$ is a $\CC$-linear morphism
        \[
            \text{Levi}_{D,\,z} : \bigwedge^{2}D_{z} \rightarrow T_{Z,\,z} / D_{z}
        \]
        defined as follows:
        for $v,\,w \in D_{z}$, choose local sections of $D$ extending $v$ and $w$, say $\tilde{v}$ and $\tilde{w}$, respectively.
        Then
        \[
            \text{Levi}_{D,\,z}(v \wedge w) \coloneqq [\tilde{v},\, \tilde{w}](z) \mod D_{z}
        \]
        where $[\tilde{v},\, \tilde{w}](z)$ is the value of the vector field $[\tilde{v},\, \tilde{w}]$ at the point $z$.
        This definition is independent of the choice of $\tilde{v}$ and $\tilde{w}$.

        \item We say that $D$ is a \emph{contact structure} of $Z$ if for any $z \in Z$, the Levi tensor $\text{Levi}_{D,\,z}$, viewed as a skew-symmetric 2-form on $D_{z}$, is a symplectic form of $D_{z}$.
        In this case, the quotient line bundle $T_{Z}/D$ is called the \emph{contact line bundle}, and the quotient morphism $T_{Z} \rightarrow T_{Z}/D$ is called the \emph{contact form}.
    \end{enumerate}
\end{definition}

\begin{definition}\label{defn:integral/Legendrian submflds}
    Let $Z$ be a complex manifold equipped with a contact structure $D$.
    Let $X$ be a submanifold of $Z$.
    We say that $X$ is an \emph{integral submanifold} of $Z$ if $X$ is everywhere tangent to $D$, that is, for any $x\in X$, we have $T_{X,\,x} \subset D_{x}$.
    In other words, $T_{X,\,x}$ is an isotropic subspace of $D_{x}$ with respect to the Levi tensor $\text{Levi}_{D,\,x}$, and in particular, $2\dim X + 1 \le \dim Z$.
    If $X$ is an integral submanifold and $2 \dim X + 1= \dim Z$, then we say that $X$ is a \emph{Legendrian submanifold} of $Z$.
\end{definition}

Roughly speaking, integral submanifolds are solutions to a contact system, and Legendrian submanifolds are maximal dimensional ones.

\begin{definition} \label{defn:integral/Legendrian subvarieties}
    Let $Z$ be a smooth variety equipped with a contact structure $D$.
    We say that a subvariety $X$ of $Z$ is an \emph{integral subvariety} (resp. a \emph{Legendrian subvariety}) if the smooth locus $\Sm(X)$ of $X$ is an integral submanifold (resp. a Legendrian submanifold) of $Z$.
\end{definition}

    \begin{example}\label{ex:PT*M}
        Let $M$ be a complex manifold, and consider $\PP(T^{\vee}_{M})$, the total space of the projectivized cotangent bundle.
        Let $\pi : \PP(T^{\vee}_{M}) \rightarrow M$ be the natural projection.
        Then $\PP(T^{\vee}_{M})$ comes equipped with a canonical contact structure $D_{\can}$.
        Indeed, for $z \in \PP(T^{\vee}_{M})$, the annihilator $\Ann(z) \subset T_{M,\,\pi(z)}$ is a hyperplane, and
        \[
            D_{\can} \coloneqq \bigcup_{z \in \PP(T^{\vee}_{M})} (d_{z}\pi)^{-1}(\Ann(z))
        \]
        is a contact structure.
        Notice that by the definition, $D_{\can}$ is defined by the contact form
        \[
            \theta_{\can} : T_{\PP(T^{\vee}_{M})} \rightarrow \Ocal_{\PP(T^{\vee}_{M})}(1), \quad v \in T_{\PP(T^{\vee}_{M}),\,z} \mapsto (\alpha \mapsto \langle \alpha,\, d\pi(v)\rangle) \in (\CC \alpha)^{\vee} = (\Ocal_{\PP(T^{\vee}_{M})}(1))_{z}
        \]
        where $z = [\alpha]$, $\alpha \in T^{\vee}_{M,\,\pi(z)}$.
        Then $\theta_{\can}$ and $D_{\can}$ fit into the following commutative diagram with exact rows:
        \begin{center}
            \begin{tikzcd}
                0 \arrow[r] & D_{\can} \arrow[r] & T_{\PP(T^{\vee}_{M})} \arrow[d, "d\pi"] \arrow[r, "\theta_{\can}"] & \Ocal_{\PP(T^{\vee}_{M})}(1) \arrow[r] \arrow[d, equal] & 0 \\
                0 \arrow[r] & \Omega_{\PP(T^{\vee}_{M})/M}(1) \arrow[r] & \pi^{*}T_{M} \arrow[r] & \Ocal_{\PP(T^{\vee}_{M})}(1) \arrow[r] & 0.
            \end{tikzcd}
        \end{center}
        Here, the second row is the relative Euler sequence on $\PP(T^{\vee}_{M})$.
        In particular, the associated contact line bundle is $\Ocal_{\PP(T^{\vee}_{M})}(1)$.
        In this paper, the contact structure of $\PP(T^{\vee}_{M})$ refers to $D_{\can}$ (though $\PP(T^{\vee}_{M})$ may admit more than one contact structure \cite[Proposition~2.14]{KPSW00Projective}).
\end{example}

\begin{example} \label{ex:P 2n+1}
    Let $(V,\,\omega)$ be a symplectic vector space.
    Then the odd dimensional projective space $\PP(V)$ admits a natural contact structure: for $v \in V$, $T_{\PP(V),\, [v]}$ is naturally identified with $(V/\CC v) \otimes (\CC v)^{\vee}$, and hence $D_{[v]} \coloneqq (\{w \in V : \omega(v,\,w) = 0\} / \CC v) \otimes (\CC v)^{\vee}$ is a hyperplane in $T_{\PP(V),\,[v]}$.
    Then $D \coloneqq \bigcup_{[v] \in \PP(V)} D_{[v]}$ is a contact structure of $\PP(V)$, and its contact line bundle is $\Ocal_{\PP(V)}(2)$.
    Conversely, any contact structure on $\PP(V)$ comes from a symplectic form on $V$ (cf. \cite[Example~2.1]{LeBrun95Fano}).
\end{example}

\subsection{Deformation of contact lines}\label{ssection:contact lines}

Next, we recall properties of rational curves of degree one with respect to the contact line bundle.

\begin{definition}\label{defn:contact line}
    Let $Z$ be a complex manifold, and $L$ a line bundle on $Z$.
    A rational curve $C$ on $Z$ is called an \emph{$L$-line} if $\deg_{C}L=1$, i.e., for the normalization map $\nu : \PP^{1} \rightarrow C$, we have $\nu^{*}L \simeq \Ocal_{\PP^{1}}(1)$.
    If $Z$ is a contact manifold and $L$ is the contact line bundle, $L$-lines are called \emph{contact lines}.
\end{definition}

It is well known that every contact line is tangent to the contact structure.
Slightly more generally, we have:
\begin{proposition} \label{prop:contact lines are tangent to D}
    Let $Z$ be a complex manifold and $L$ a line bundle on $Z$.
    Suppose that we have a bundle morphism $\theta : T_{Z} \rightarrow L$.
    If $C$ is a rational curve on $Z$ such that $\deg_{C}L \le 1$, then for any $z \in \Sm(C)$, we have $T_{C,\,z} \subset \ker \theta_{z}$.
\end{proposition}
\begin{proof}
    Let $\nu: \PP^{1} \rightarrow C$ be the normalization map.
    It is enough to show that the composition
    \[
        T_{\PP^{1}} \xrightarrow{d \nu} \nu^{*}T_{Z} \xrightarrow{\nu^{*}\theta} \nu^{*}L
    \]
    is the zero map.
    Indeed, since $T_{\PP^{1}} \simeq \Ocal_{\PP^{1}}(2)$ and $\nu^{*}L \simeq \Ocal_{\PP^{1}}(d)$ for $d = \deg_{C}L \le 1$, every morphism $T_{\PP^{1}} \rightarrow \nu^{*}L$ is the zero map.
\end{proof}

Another well-known fact is that contact lines are unbendable in the deformation theoretic sense.
To be precise, from now on, we work in the complex algebraic category.
Given a quasi-projective variety $Z$, denote by $\RatCurves(Z)$ the normalized Chow scheme of rational curves on $Z$ (\cite[Definition-Proposition~2.11, Ch.~II]{Kollar96}).

\begin{definition} \label{defn:VMRT}
    Let $Z$ be a smooth quasi-projective variety.
    \begin{enumerate}
        \item A nonconstant morphism $f:\PP^{1} \rightarrow Z$ is said to be \emph{free} if $f^{*}T_{Z}$ is globally generated.
        We say that a rational curve $C$ on $Z$ is \emph{free} if the normalization map $\PP^{1} \rightarrow C \subset M$ is a free morphism.
        \item A free morphism $f : \PP^{1} \rightarrow Z$ is said to be \emph{unbendable} if
        \begin{equation} \label{eqn:splitting type of unbendable}
            f^{*}T_{Z} \simeq \Ocal_{\PP^{1}}(2) \oplus \Ocal_{\PP^{1}}(1)^{\oplus a} \oplus \Ocal_{\PP^{1}}^{\oplus (\dim Z - a -1)}
        \end{equation}
        for some integer $0 \le a \le \dim Z - 1$.
        The integer $a$ is uniquely determined as $\deg_{\PP^{1}}(f^{*}K^{\vee}_{Z}) - 2$.
        We say that a rational curve $C$ is an \emph{unbendable curve of type $(1^{a},\,0^{\dim Z - a-1})$} if the normalization map $\nu : \PP^{1}\rightarrow C$ satisfies the equation (\ref{eqn:splitting type of unbendable}).
        \item An irreducible component $\Kscr$ of $\RatCurves(Z)$ is called a \emph{family of unbendable curves of type $(1^{a},\,0^{\dim Z - a -1})$} if general members of $\Kscr$ are unbendable and of type $(1^{a},\,0^{\dim Z - a -1})$ for some integer $0 \le a \le \dim Z - 1$.

        \item Let $\Kscr$ be a family of unbendable curves of type $(1^{a},\,0^{\dim Z - a -1})$.
        The \emph{VMRT} (standing for variety of minimal rational tangents) of $\Kscr$ at $z$ is defined to be the subset of $\PP(T_{Z,\,z})$ consisting of elements of the form $[d \nu(T_{\PP^{1},\,o})]$ where $\nu : \PP^{1} \rightarrow Z$ is the normalization map of a member of $\Kscr$ and $o \in \PP^{1}$ is a point such that $\nu(o) = z$ and $\nu$ is immersive at $o$.
        This set, denoted by $\VMRT^{\Kscr}_{z}$, is a subvariety of pure dimension $a$ in $\PP(T_{Z,\,z})$ (cf. \cite[Theorem~1.3 and Proposition~1.4]{Hwang01VMRT}).
    \end{enumerate}
\end{definition}

\begin{remark}
Given an unbendable curve $C$, the normalization $\nu : \PP^{1} \rightarrow C$ is immersive, and hence $[d\nu(T_{\PP^{1},\,o})] \in \PP(T_{Z,\,z})$ is well defined.
\end{remark}

\begin{theorem}[{\cite[Theorem~~3.11, Ch.~II]{Kollar96}, \cite[Lemma~3.5]{Ke01}}] \label{thm:covering family is free}
    Let $Z$ be a smooth quasi-projective variety, and $z$ a very general point in $Z$.
    Then every rational curve passing through $z$ is free.
    In particular, if $Z$ is equipped with a contact structure and $2n+1=\dim Z$, then every contact line passing through $z$ is an unbendable curve of type $(1^{n-1},\,0^{n+1})$.
\end{theorem}

\begin{proof}
    Freeness of rational curves through very general points on a smooth quasi-projective variety is the content of \cite[Theorem~~3.11, Ch.~II]{Kollar96}.
    It is well known that free contact lines are unbendable and of type $(1^{n-1},\,0^{n+1})$, see for example \cite[Lemma~3.5]{Ke01}.
\end{proof}

For an irreducible component $\Kscr$ of $\RatCurves(Z)$ and points $z \in Z$, $\Kscr_{z}$ is defined to be the subvariety parametrizing members of $\Kscr$ passing through $z$.
More precisely, if
\[
    \Kscr \xleftarrow{\rho} \text{Univ} \xrightarrow{\mu} Z
\]
is the family of rational curves parametrized by $\Kscr$, then $\Kscr_{z} = \rho(\mu^{-1}(z))$.
It is a closed subvariety of $\Kscr$.

\begin{definition}
    Let $Z$ be a smooth quasi-projective variety, equipped with a contact structure.
    An irreducible component $\Hscr$ of $\RatCurves(Z)$ is called a \emph{family of contact lines} if members of $\Hscr$ are contact lines, and for general $z \in Z$, $\Hscr_{z} \not=\emptyset$ (or equivalently, some members of $\Hscr$ are free; cf. Theorem~\ref{thm:covering family is free} and \cite[Proposition~2.7, Ch.~IV]{Kollar96}).
\end{definition}

By Theorem~\ref{thm:covering family is free}, any family $\Hscr$ of contact lines is a family of unbendable curves of type $(1^{n-1},\, 0^{n+1})$, and for very general $z$, $\VMRT^{\Hscr}_{z}$ is of pure dimension $n-1$.

\section{Integrality theorems}
\label{section:Integrality theorems}

From now on, we work in the complex algebraic category.
The goal of this section is to show Theorem~\ref{thm:VMRT is integral}: tangent directions of lines on a nilpotent orbit closure form an integral subvariety in the contact hyperplane (cf. Example~\ref{ex:P 2n+1}).
In the proof, the Jacobson--Morozov resolution plays an important role, and we recall its construction in Section~\ref{ssection:JM resoln}.

First we recall the definition of the contact structures on nilpotent orbits.
Let $\gfr$ be a complex semisimple Lie algebra, and $G$ the simply connected Lie group with $\Lie(G)=\gfr$.
Then $G$ acts naturally on $\gfr$ via the adjoint representation.
To be precise, the adjoint representation is denoted by $\Ad : G \times \gfr \rightarrow \gfr$, $(g,\, v) \mapsto \Ad_{g}(v)$.
Observe that the differential of $\Ad$ is nothing but the Lie algebra representation $\ad : \gfr \times \gfr \rightarrow \gfr$, $(u,\,v) \mapsto \ad_{u}(v) = [u,\,v]$ (Lie bracket).

\begin{definition}
    An element $v \in \gfr$ is said to be \emph{nilpotent} if $\ad_{v}$ is a nilpotent endomorphism of $\gfr$.
    A \emph{nilpoent orbit} in $\gfr$ means an adjoint orbit $\Ad_{G}(v)$ in $\gfr$ for a nilpotent element $v \in \gfr$.
    By a slight abuse of notation, we say that an adjoint orbit in the projectivization $\PP(\gfr)$ is a \emph{nilpotent orbit} in $\PP(\gfr)$ if it is of the form $\Ad_{G}([v])$ for a nonzero nilpotent element $v \in \gfr$.
\end{definition}

Recall the natural one-to-one correspondence between the set of nonzero nilpotent orbits $O$ in $\gfr$, and the set of nilpotent orbits $Z$ in $\PP(\gfr)$, given by $O \mapsto \PP(O)$.
Here $\PP(O)$ is the image of $O$ under the projectivization map $p : \gfr \setminus \{0\} \rightarrow \PP(\gfr)$.
Indeed, the assignment $O \mapsto \PP(O)$ has an inverse, by the Jacobson--Morozov theorem (\cite[Theorem~3.3.1]{CollingwoodMcGovern1993NilpotentOrbits}): every nonzero nilpotent element can be extended to an $\mathfrak{sl}_{2}$-triple, and in particular, a nilpotent orbit in $\gfr$ is invariant under the homothety.

\begin{definition} \label{defn: adjoint var}
    Let $Z \subset \PP(\gfr)$ be a nilpotent orbit.
    We say that $Z$ is a \emph{minimal nilpotent orbit} if $Z = \overline{Z}$ in $\PP(\gfr)$.
    Similarly, we say that a nilpotent orbit $O \subset \gfr$ is \emph{minimal} if it is nonzero and $O = \overline{O} \setminus \{0\}$ in $\gfr$, or equivalently, if $\PP(O)$ is a minimal nilpotent orbit in $\PP(\gfr)$.
    When $\gfr$ is simple, there exists a unique minimal nilpotent orbit in $\PP(\gfr)$, called the \emph{adjoint variety}.
\end{definition}

\begin{remark} \label{rmk:adjoint variety}
    If $\gfr$ is not simple, then any minimal nilpotent orbit in $\PP(\gfr)$ is the adjoint variety $Z_{\hfr} \subset \PP(\hfr)$ where $\hfr$ is a simple ideal of $\gfr$.
\end{remark}

Let $Z = \PP(O) \subset \PP(\gfr)$ be a nilpotent orbit.
It is well known that $O$ admits a symplectic form $\omega_{KKS}$ that is invariant under $\Ad_{G}$ and the homothety (the so-called Kostant--Kirillov--Souriau form; see \cite[\S2]{Beauville1998FanoContact}).
On the side of $Z$, this is reflected by the existence of contact structure.

\begin{definition}
    Let $\kappa$ be the Killing form of $\gfr$.
    Choose a point $[v] \in Z$, and let $\Stab_{G}([v])$ be the stabilizer of $[v]$ in $G$, and $\Lie(\Stab_{G}([v]))$ its Lie algebra.
    Consider the $G$-equivariant bundle morphism
    \[
        \theta_{KKS} : T_{Z} \simeq G \times^{\Stab_{G}([v])} (\gfr/\Lie(\Stab_{G}([v]))) \rightarrow G \times^{\Stab_{G}([v])} (\CC v)^{\vee} \simeq \Ocal_{\PP(\gfr)}(1)|_{Z}
    \]
    induced by
    \[
        \gfr/\Lie(\Stab_{G}([v])) \rightarrow (\CC v)^{\vee}, \quad w \mod \Lie(\Stab_{G}([v])) \mapsto \kappa(v,\,w).
    \]
    Then $\theta_{KKS}$ is a contact form (see \cite[Remark~2.3]{Beauville1998FanoContact}).
    Define $D_{KKS} \coloneqq \ker(\theta_{KKS})$.
    By the definition, $D_{KKS}$ is a contact structure of $Z$ and the contact line bundle is $\Ocal_{\PP(\gfr)}(1)|_{Z}$.
\end{definition}

\begin{remark} \label{rmk:symplectification}
    The contact form $\theta_{KKS}$ is related to $\omega_{KKS}$ via the so-called symplectification process (cf.~\cite{LeBrun95Fano}, \cite[Lemma~1.4]{Beauville1998FanoContact}).
    Indeed, the $\CC^{\times}$-bundle $p : O \rightarrow Z = \PP(O)$ can be viewed as the one associated to the tautological line bundle $\Ocal_{\PP(\gfr)}(-1)|_{Z}$.
    Thus $p^{*}\Ocal_{\PP(\gfr)}(1)$ is trivialized over $O$, and so $\theta_{KKS}$ can be defined via the relation $p^{*}\theta_{KKS} = \iota_{\Xi}\omega_{KKS}$ where $\Xi$ is the vector field on $O$ generating the fiberwise $\CC^{\times}$-action and $\iota_{\Xi}$ is the contraction.
\end{remark}

\begin{example} \label{ex:no lines on adjoint of type C}
    Suppose that $\gfr$ is a simple Lie algebra.
    If $\gfr =\mathfrak{sl}(V)$ for a vector space $V$, then the adjoint variety in $\PP(\gfr)$ is $\PP(T^{\vee}_{\PP(V)})$, and $D_{KKS}$ agrees with $D_{\can}$ (cf. Examples~\ref{ex:PT*M}).
    If $\gfr = \mathfrak{sp}(V)$ for a symplectic vector space $V$, then the adjoint variety in $\PP(\gfr)$ is the 2nd Veronese embedding of $\PP(V)$, and $D_{KKS}$ coincides with the contact structure in Example~\ref{ex:P 2n+1}.
    
\end{example}

\subsection{Jacobson--Morozov resolutions} \label{ssection:JM resoln}

Let $Z = \PP(O) \subset \PP(\gfr)$ be a nilpotent orbit.
We recall the construction of a well-known resolution of singularities of the closure $\overline{Z}$, which is another consequence of the Jacobson--Morozov theorem.
Our main reference is \cite[\S4]{Beauville1998FanoContact}.

Choose a point $[v] \in Z$.
By the Jacobson--Morozov theorem, there are $f,\,h \in \gfr$ such that
\[
    \ad_{h}v = 2v, \quad \ad_{h}f = -2f, \quad \ad_{v}f = h.
\]
In other words, $\{v,\,h,\,f\}$ is an $\mathfrak{sl}_{2}$-triple.
Then $\gfr$ is decomposed into eigenspaces of $\ad_{h}$, and by the $\mathfrak{sl}_{2}$-representation theory, all eigenvalues of $\ad_{h}$ are integers, i.e.,
\begin{equation} \label{eqn:eigensp decomp of g}
    \gfr = \bigoplus_{k \in \ZZ} \gfr_{k}, \quad \gfr_{k} \coloneqq \{w \in \gfr : \ad_{h}w = kw\}.
\end{equation}
For $k \in \ZZ$, define $\gfr_{\ge k} \coloneqq \bigoplus_{k' \ge k} \gfr_{k'}$.
Then $\gfr_{\ge 0}$ is a parabolic subalgebra of $\gfr$, and we denote by $G_{\ge 0}$ the corresponding parabolic subgroup of $G$.
Let $E$ be the homogeneous vector bundle $G \times^{G_{\ge 0}} \gfr_{\ge 2}$.
Then $E$ is the quotient of $G \times \gfr_{\ge 2}$ by the $G_{\ge 0}$-action given by $g_{1}.(g_{2},\,w) \coloneqq (g_{2}g_{1}^{-1},\, \Ad_{g_{1}}w)$.
Now we obtain a $G$-equivariant birational morphism
\begin{equation}\label{eqn:JM resoln of O}
    \tilde{r}_{JM}: E = G \times^{G_{\ge 0}} \gfr_{\ge 2} \rightarrow \overline{O}, \quad (g,\, w)^{\bullet} \mapsto \Ad_{g}w
\end{equation}
where $(g,\,w)^{\bullet}$ is the point represented by $(g,\,w) \in G \times \gfr_{\ge 2}$.
In fact, under the morphism $\tilde{r}_{JM}$ in (\ref{eqn:JM resoln of O}), the $G$-orbit of $(1,\, v)^{\bullet}$ in $G \times^{G_{\ge 0}} \gfr_{\ge 2}$ is isomorphically sent onto $O$.
Similarly, we have a $G$-equivariant birational morphism
\begin{equation}\label{eqn:JM resoln of Z}
    r_{JM} : \PP(E) = G \times^{G_{\ge 0}} \PP(\gfr_{\ge 2}) \rightarrow \overline{Z}, \quad (g,\, [w])^{\bullet} \mapsto [\Ad_{g}w],
\end{equation}
and the $G$-orbit of $(1,\, [v])^{\bullet}$ in $G \times^{G_{\ge 0}} \PP(\gfr_{\ge 2})$ is isomorphically sent onto $Z$.

\begin{definition}
    The $G$-equivariant birational morphism $r_{JM} : \PP(E) \rightarrow \overline{Z}$ in (\ref{eqn:JM resoln of Z}) is called the \emph{Jacobson--Morozov resolution} of $\overline{Z}$.
\end{definition}

By the construction, we have $(r_{JM})^{*} \Ocal_{\PP(\gfr)}(1) \simeq \Ocal_{\PP(E)}(1)$, the dual of the tautological line bundle on the projective bundle $\PP(E)$.

As observed by \cite[Lemma~2]{Pa91} and \cite[Lemma~4.3]{Beauville1998FanoContact}, an important feature of the Jacobson--Morozov resolution of $\overline{O}$ is that the pullback of $\omega_{KKS}$ extends to a global 2-form (but not symplectic in general; see \cite[Corollary~2.9]{Fu06}).
This has the following counterpart in contact geometry:
\begin{proposition} \label{prop:Theta KKS}
    Let $r_{JM} : \PP(E) \rightarrow \overline{Z}$ be the Jacobson--Morozov resolution.
    Then the pullback $(r_{JM})^{*} (\theta_{KKS})$ of the contact form $\theta_{KKS} : T_{Z} \rightarrow \Ocal_{\PP(\gfr)}(1)|_{Z}$ extends to a $G$-equivariant surjective bundle morphism $\Theta_{KKS}: T_{\PP(E)} \rightarrow \Ocal_{\PP(E)}(1)$.
\end{proposition}

\begin{proof}
    Let $\kappa$ be the Killing form of $\gfr$.
    Since $\gfr/\gfr_{\ge -1} \rightarrow (\gfr_{\ge 2})^{\vee}$, $w \mod \gfr_{\ge -1} \mapsto \kappa(-,\,w)$ is a $G_{\ge0}$-equivariant isomorphism, we have a $G$-equivariant isomorphism $E^{\vee} \simeq G \times^{G_{\ge 0}} (\gfr/\gfr_{\ge -1})$.
    Therefore we have a $G$-equivariant surjection $\varphi : T_{G/G_{\ge 0}} \rightarrow E^{\vee}$ defined by the natural projection $\gfr/\gfr_{\ge 0} \rightarrow \gfr/\gfr_{\ge -1}$.

    Denote by $q : \PP(E) \rightarrow G/G_{\ge 0}$ the natural projection.
    Define $\Theta_{KSS} : T_{\PP(E)} \rightarrow \Ocal_{\PP(E)}(1)$ as the following composition:
        \begin{equation} \label{eqn:Theta KKS}
            \Theta_{KKS}: T_{\PP(E)} \xrightarrow{dq} q^{*} T_{G/G_{\ge 0}} \xrightarrow{q^{*}\varphi} q^{*} E^{\vee} \xrightarrow{e} \Ocal_{\PP(E)}(1)
        \end{equation}
    where $e : q^{*} E^{\vee} \rightarrow \Ocal_{\PP(E)}(1)$ comes from the relative Euler sequence.
    Then $\Theta_{KKS}$ is a $G$-equivariant surjection.
    Furthermore, at $(1,\,[v])^{\bullet} \in \PP(E)$, (\ref{eqn:Theta KKS}) becomes
    \begin{equation*}
            \Theta_{KKS,\, (1,\,[v])^{\bullet}}: \gfr/\Lie(\Stab_{G}([v])) \rightarrow \gfr/\gfr_{\ge 0} \rightarrow \gfr/\gfr_{\ge -1} \xrightarrow{e} (\CC v)^{\vee}
        \end{equation*}
        where $e$ sends $w \mod \gfr_{\ge -1}$ to the linear functional $v \mapsto \kappa(v,\, w)$.
        Thus $\Theta_{KKS}$ agrees with $\theta_{KKS}$ at $(1,\,[v])^{\bullet}$, and hence over $G.(1,\,[v])^{\bullet}$ by the $G$-equivariance.
\end{proof}

Notice that $\Theta_{KKS}$ is in general not a contact form (cf. \cite{fu07contact}).

\begin{remark}
    Alternatively, Proposition~\ref{prop:Theta KKS} can be deduced from the aforementioned extendability of $\omega_{KKS}$ over $E$.
    The idea is as follows.
    Let $p:E^{\times} = E \setminus (\text{zero section}) \rightarrow \PP(E)$ be the principal $\CC^{\times}$-bundle: this is exactly the $\CC^{\times}$-bundle associated to the tautological line bundle $\Ocal_{\PP(E)}(-1)$, and thus the pullback $p^{*}\Ocal_{\PP(E)}(1)$ is trivialized over $E^{\times}$.
    Observe that the principal $\CC^{\times}$-bundle $O\rightarrow Z = \PP(O)$ is embedded in $p$ and the fiberwise $\CC^{\times}$-actions are compatible.
    Let $\Omega_{KKS}$ be a $G$-equivariant extension of $\omega_{KKS}$ on $E^{\times}$.
    Since $\omega_{KKS}$ is equivariant under the fiberwise $\CC^{\times}$-action (that is, $t^{*}\omega_{KKS} = t\omega_{KKS}$, $\forall t \in \CC^{\times}$), so is $\Omega_{KKS}$.
    Therefore if $\Xi$ is the vector field on $E^{\times}$ generating the fiberwise $\CC^{\times}$-action, then the contraction $\iota_{\Xi} \Omega_{KKS}$ is $\CC^{\times}$-invariant.
    Recalling that $p^{*}\Ocal_{\PP(E)}(1)$ is globally trivialized, one can define $\Theta_{KKS} \rightarrow \Ocal_{\PP(E)}(1)$ by setting $p^{*}\Theta_{KKS} = \iota_{\Xi}\Omega_{KKS}$.
    The symplectification process (cf. Remark~\ref{rmk:symplectification}) ensures that the same relation holds true for $\theta_{KKS}$ and $\omega_{KKS}$, and hence $\Theta_{KKS}$ is an extension of $\theta_{KKS}$.
\end{remark}

\subsection{Behavior of lines}\label{ssection:Behavior of lines}

As before, let $Z = \PP(O)$ be a nilpotent orbit in $\PP(\gfr)$, and $\overline{Z}$ its closure in $\PP(\gfr)$.
We also write $2n+1=\dim Z$.
Choose a point $z = [v] \in Z$, and consider the Jacobson--Morozov resolution defined by $v$ and an $\mathfrak{sl}_{2}$-triple $\{v,\,h,\,f\}$.
To ease the notation, we write $\tilde{z} = (1,\,[v])^{\bullet}$ in $\PP(E)$.
Recall that the $G$-orbit of $\tilde{z}$ is an open orbit in $\PP(E)$, isomorphic to $Z$.

\begin{lemma} \label{lemma: g2 = 1 iff minimal}
    In the eigenspace decomposition of $\gfr$ (\ref{eqn:eigensp decomp of g}), $\dim \gfr_{\ge 2} = 1$ if and only if $Z$ is a minimal nilpotent orbit.
\end{lemma}

\begin{proof}
    If $Z$ is a minimal nilpotent orbit, then $v$ is a highest root vector of a simple ideal $\hfr$ of $\gfr$.
    Since $[\hfr,\,\hfr'] = 0$ for any other simple ideal $\hfr'$ of $\gfr$, $h$ and $f$ are contained in $\hfr$, i.e., $\{v,\,h,\,f\}$ is an $\mathfrak{sl}_{2}$-triple in $\hfr$.
    It is well known that the $\ad_{h}$-decomposition (\ref{eqn:eigensp decomp of g}) for $\hfr$ is
    \[
        \hfr = \hfr_{-2} \oplus \hfr_{-1} \oplus \hfr_{0} \oplus \hfr_{1} \oplus \hfr_{2}
    \]
    and $\hfr_{2} = \CC v$.
    Since all other simple ideals $\hfr'$ of $\gfr$ are contained in $\gfr_{0}$, we see that $\gfr_{\ge 2} = \hfr_{2} = \CC v$.

    Conversely, assume that $\dim \gfr_{\ge 2} = 1$.
    Then $\PP(E) = G \times^{G_{\ge 0}} \PP(\gfr_{2})$ is isomorphic to $G/G_{\ge 0}$, and so it is $G$-homogeneous.
    Since the Jacobson--Morozov resolution is $G$-equivariant, $\overline{Z}$ is $G$-homogeneous, and hence $\overline{Z} = Z$.
\end{proof}

\begin{corollary} \label{coro:covered by lines except for adj type C}
    $\overline{Z}$ is covered by lines in $\PP(\gfr)$ unless it is the adjoint variety of a simple ideal of $\gfr$, isomorphic to the symplectic Lie algebra $\mathfrak{sp}_{N}$, $N \ge 2$ even.
\end{corollary}

\begin{proof}
    It is well known that an adjoint variety is covered by lines if and only if it is not of type $C$ (see for example \cite[Table~I]{Buczynski2006LegendrianSubvarieties}).
    Thus it suffices to show that non-minimal $Z$ is covered by lines.
    By Lemma~\ref{lemma: g2 = 1 iff minimal}, $\PP(\gfr_{\ge 2})$ is a positive dimensional linear subspace of $\overline{Z}$, and it contains $[v] \in Z$.
    In particular, there is a line on $\overline{Z}$ passing through $[v]$.
    Since $Z$ is an open $\Ad_{G}$-orbit in $\overline{Z}$, $\overline{Z}$ is covered by lines.
\end{proof}

\begin{lemma}\label{lemma:str transform of lines are unbendable}
    On $\PP(E)$, $\Ocal_{\PP(E)}(1)$-lines passing through $\tilde{z}$ are smooth, everywhere tangent to $\ker(\Theta_{KKS})$, and unbendable.
\end{lemma}
\begin{proof}
    Let $l^{\sim}$ be an $\Ocal_{\PP(E)}(1)$-line containing $\tilde{z}$.
    By the projection formula, it is the strict transform of a line on $\overline{Z}$, and hence a smooth rational curve.
    Moreover, it is everywhere tangent to $\ker(\Theta_{KKS})$ by Proposition~\ref{prop:contact lines are tangent to D}.

    Assume that $l^{\sim}$ is not unbendable.
    Since $l^{\sim}$ meets the open dense orbit $G.\tilde{z}$ in $\PP(E)$, $l^{\sim}$ is a free curve by Theorem~\ref{thm:covering family is free}.
    But since $l^{\sim}$ is not unbendable, the splitting type of $T_{\PP(E)}|_{l^{\sim}}(-2)$ contains at least two nonnegative factors, and thus $l^{\sim}$ can be deformed to another $\Ocal_{\PP(E)}(1)$-line, say $l'$, with two points $z_{1}\not=z_{2} \in l^{\sim}$ fixed (see for example \cite[Proposition~3.10, Ch.~II]{Kollar96}).
    We can assume that $z_{1}$ and $z_{2}$ are in the open dense orbit $G.\tilde{z}$, and hence $l'$ also intersects with $G.\tilde{z}$.
    Again by the projection formula, $r_{JM}(l^{\sim})$ and $r_{JM}(l')$ are different lines passing through different points $r_{JM}(z_{1})$ and $r_{JM}(z_{2})$, which is impossible.
    Therefore $l^{\sim}$ is unbendable.
\end{proof}

Recall the definition of $F(\overline{Z},\,z)$ given in the introduction:
\[
    F(\overline{Z},\,z) = \{[T_{l,\,z}] \in \PP(T_{Z,\,z}) : \text{$l$ is a line on $\overline{Z}$ containing $z$}\}.
\]
This is a closed subvariety of $\PP(T_{Z,\,z})$.
In fact, under the natural isomorphism $\PP(T_{\PP(\gfr),\,z}) \simeq \Gr(1,\,\gfr/\CC v) \subset \Gr(2,\,\gfr)$, $F(\overline{Z},\,z)$ is identified with the Fano variety of lines on $\overline{Z}$ passing through $z$.

\begin{theorem} \label{thm:VMRT is integral}
    Let $X$ be an irreducible component of $F(\overline{Z},\,z)$.
    Then $X$ is an integral subvariety of $\PP(D_{KKS,\, z})$, equipped with the contact structure arising from the Levi tensor (cf. Example~\ref{ex:P 2n+1}).
\end{theorem}
\begin{proof}
    By the projection formula, the differential of the Jacobson--Morozov resolution identifies $F(\overline{Z},\,z)$ in $\PP(T_{Z,\,z})$ with the set of tangent directions of $\Ocal_{\PP(E)}(1)$-curves on $\PP(E)$ passing through $\tilde{z}$.
    Then the statement follows from Lemma~\ref{lemma:str transform of lines are unbendable} and \cite[Theorem~1.4]{Hw23Partial}.
\end{proof}

As an immediate corollary, we have:

\begin{corollary} \label{coro:dim VMRT is at most n-1}
    We have $\dim F(\overline{Z},\,z) \le n-1$.
\end{corollary}

\begin{definition}
    An irreducible component $X$ of $F(\overline{Z},\,z)$ is called \emph{Legendrian} if $\dim X = n-1$, or equivalently, $X$ is a Legendrian subvariety of $\PP(D_{KKS,\,z})$.
    The set of Legendrian components of $F(\overline{Z},\,z)$ is denoted by $\text{Leg}(\overline{Z},\,z)$.
\end{definition}

Let us conclude by another integrality theorem, which is a variant of \cite[Proposition~4.1]{Ke01}.
In the following, for a family of rational curves on a variety $Y$, say
\[
    \Kscr \xleftarrow{\rho}\text{Univ}\xrightarrow{\mu} Y,
\]
we define
\[
    \locus(\Kscr) \coloneqq \overline{\im \mu}.
\]
Notice that if $\Kscr$ is projective, then $\im \mu$ is also projective, and so taking the closure is not necessary.
In particular, $\locus(F(\overline{Z},\,z))$ is the union of all lines on $\overline{Z}$ through $z$.

\begin{theorem} \label{thm:locus of lines is integral}
    Let $X$ be an irreducible component of $\locus(F(\overline{Z},\,z))$ in $\overline{Z}$.
    Then $X \cap Z$ is an integral subvariety of $Z$.
\end{theorem}
\begin{proof}
    It suffices to show that the locus of $\Ocal_{\PP(E)}(1)$-lines on $\PP(E)$ through $\tilde{z}$ is integral, which follows from the following Lemma~\ref{lemma:Kebekus lemma}
\end{proof}

\begin{lemma}[{\cite[Proposition~4.1]{Ke01}}] \label{lemma:Kebekus lemma}
    Let $Y$ be a smooth quasi-projective variety, $L$ a line bundle on $Y$, and $\theta: T_{Y} \rightarrow L$ a surjective bundle morphism.
    Assume that $U$ is a nonempty open subset of $Y$ such that the restriction $\theta|_{U}: U \rightarrow L|_{U}$ is a contact form.
    Then for a very general point $y \in U$ and an irreducible component $\Kscr$ of $\RatCurves(Y,\,y)$ whose members are $L$-lines, the intersection $\locus(\Kscr) \cap U$ is an integral subvariety of $U$ with respect to the contact structure $\ker \theta |_{U}$.
\end{lemma}
\begin{proof}
    This is essentially proven in the proof of \cite[Proposition~4.1]{Ke01}, and we only sketch the proof.
    First, since $y$ is very general in $U$ (and hence in $Y$), by Theorem~\ref{thm:covering family is free}, $\Kscr$ parametrizes free curves.
    Then we have an irreducible component $\Hscr$ of $\Hom_{bir}^{nor}(\PP^{1},\, Y,\,o\mapsto y)$, the normalized Hom-scheme of morphisms sending a chosen point $o \in \PP^{1}$ to $y$, equipped with the following commutative diagram:
    \begin{center}
        \begin{tikzcd}
            \PP^{1} \times \Hscr \arrow[rr, bend left=30, "F"] \arrow[r] \arrow[d, "pr_{2}"] & \text{Univ} \arrow[d, "\rho"] \arrow[r, "\mu"] & Y \\
            \Hscr \arrow[r] & \Kscr &
        \end{tikzcd}
    \end{center}
    where $F$ is the evalution morphism and $\Kscr \xleftarrow{\rho}\text{Univ}\xrightarrow{\mu} Y$ is the family of rational curves parametrized by $\Kscr$.
    See \cite[Theorem~2.16, Ch.~II]{Kollar96} for details.
    Then $\locus(\Kscr) = \overline{\im \mu} = \overline{\im F}$.
    By generic smoothness, $F$ is a smooth morphism at a general point $(p,\,[f]) \in \PP^{1} \times \Hscr$, and so it is enough to show that the image of the differential $d_{(p,\,[f])}F$ is contained in $\ker \theta$.
    By \cite[Proposition~3.4, Ch.~II]{Kollar96}, we have
    \[
        d_{(p,\,[f])}F = d_{p}f + \res_{p,\,f}
    \]
    where $d_{p}f$ is the differential of $f : \PP^{1} \rightarrow Y$ and $\res_{p,\,f}$ is the restriction map
    \[
        \res_{p,\,f}: H^{0}(\PP^{1},\, f^{*}T_{Y}(-o)) \rightarrow f^{*}T_{Y} \otimes \CC(p) = T_{Y,\, f(p)}.
    \]
    We know that the image of $d_{p}f$ is contained in $\ker \theta$ by Proposition~\ref{prop:contact lines are tangent to D}, and hence it is sufficient to show that the image of $\res_{p,\,f}$ is contained in $\ker \theta$.
    This means that given a deformation $f_{t}$ of the morphism $f$ with $f_{t}(o) = y$, the induced section $\left.\pdiff{f_{t}}{t}\right|_{t=0}$ of $f^{*}T_{Y}$ is indeed a section of $f^{*} \ker \theta$.
    That is, for the global section $\sigma$ of $f^{*}L$ given by $\left.\pdiff{f_{t}}{t}\right|_{t=0}$, we claim that $\sigma = 0$.
    Now the key proposition \cite[Proposition~3.1]{Ke01}, which is valid for deformation of integral disks in any contact manifold (not necessarily algebraic), implies that $\sigma$ vanishes at $o$ with vanishing order $\ge 2$ (here we recall that the neighborhood $U$ of $f(o) = y$ is a contact manifold).
    Since $f$ parametrizes an $L$-line, $\sigma$ is a section of $f^{*}L \simeq \Ocal_{\PP^{1}}(1)$, and thus we conclude that $\sigma = 0$.
\end{proof}

Since $\locus(F(\overline{Z},\,z))$ is the cone over $F(\overline{Z},\,z)$ with vertex $z$, Theorem~\ref{thm:locus of lines is integral} implies Corollary~\ref{coro:dim VMRT is at most n-1}.
It also provides another description of $\Leg(\overline{Z},\,z)$, as the set of irreducible components $X$ of $F(\overline{Z},\,z)$ such that $\locus(X) \cap Z$ is Legendrian.

\section{Richardson orbit closures with Springer resolutions}
\label{section:Springer resoln}

In this section, we focus on the case of \emph{Richardson orbit closures}, that is, the moment images of the projectivized cotangent bundles over rational homogeneous spaces (cf. Definition~\ref{defn:Richardson}).
To this end, we study contact lines on projectivized cotangent bundle over any smooth quasi-projective varieties in Section~\ref{ssection:projectivized cotangent bundles}.
Its application to Richardson orbit closures is discussed in Section~\ref{ssection:Richardson}.
In Section~\ref{ssection:Stratified Mukai flops}, we discuss the stratified Mukai flops and prove Theorem~\ref{main thm:Mukai flop} in the case of type $E_{6,\,I}$ (cf. Corollary~\ref{coro: Mukai flop E6I}).

\subsection{Contact lines on projectivized cotangent bundles}\label{ssection:projectivized cotangent bundles}

Let $M$ be a smooth quasi-projective variety.
We denote by $\pi : \PP(T^{\vee}_{M}) \rightarrow M$ the natural projection.
Recall that $\PP(T^{\vee}_{M})$ is equipped with the natural contact structure $D_{\can}$ (cf. Example~\ref{ex:PT*M}).
Write $\dim M = n+1$, so that $\dim \PP(T^{\vee}_{M}) = 2n+1$.

The following notation is frequently used in this section:

\begin{definition}
    Let $V$ be a vector bundle on $\PP^{1}$, and choose an isomorphism
    \[
        \varphi:V \xrightarrow{\simeq} \bigoplus_{k=1}^{\rank V}\Ocal_{\PP^{1}}(v_{k})
    \]
    for some integers $v_{k}$.
    For an integer $v_{0}$, we define $V^{\ge v_{0}}$ to be the subbundle of $V$ given by
    \[
        V^{\ge v_{0}} = \varphi^{-1}\left(\bigoplus_{k: v_{k} \ge v_{0}} \Ocal_{\PP^{1}}(v_{k}) \right).
    \]
    This definition is independent of the choice of the isomorphism $\varphi$.
\end{definition}

Now we discuss liftability of rational curves on $M$ to contact lines on $\PP(T^{\vee}_{M})$.

\begin{lemma}\label{lemma:contact lines are from split surj}
    Let $f :\PP^{1} \rightarrow M$ be a morphism.
    Let $q : f^{*}T_{M} \rightarrow \Ocal_{\PP^{1}}(1)$ be a surjection, and let $\tilde{f}:\PP^{1}\rightarrow \PP(T^{\vee}_{M})$ and $s : \PP^{1} \rightarrow \PP(f^{*}T^{\vee}_{M})$ be the morphisms defined by $q$.
    Denote by $\pi':\PP(f^{*}T^{\vee}_{M}) \rightarrow \PP^{1}$ the natural projection, and by $f' : \PP(f^{*}T^{\vee}_{M}) \rightarrow \PP(T^{\vee}_{M})$ the morphism induced by $f$:
    \begin{center}
        \begin{tikzcd}
            \PP(f^{*}T^{\vee}_{M}) \arrow[r, "f'"] \arrow[d,"\pi'"] & \PP(T^{\vee}_{M}) \arrow[d, "\pi"] \\
            \PP^{1} \arrow[u, bend left=30, "s"] \arrow[r, "f"] \arrow[ru, "\tilde{f}"] & M.
        \end{tikzcd}
    \end{center}
    Then the following hold:
    \begin{enumerate}
        \item We have an exact sequence
        \begin{equation} \label{eqn:section s vs free f}
        0 \rightarrow N^{\vee}_{s}(1) \rightarrow f^{*}T_{M} \xrightarrow{q} \Ocal_{\PP^{1}}(1) \rightarrow 0
        \end{equation}
        where $N_{s} \coloneqq \coker(ds)$ is the normal bundle of the section $s(\PP^{1})$.

        \item The morphism $\tilde{f}$ is birational onto its image.
        If furthermore $\tilde{f}$ is unbendable and $f$ is immersive, then the sequence (\ref{eqn:section s vs free f}) is split.
    \end{enumerate}
\end{lemma}
\begin{proof}
\begin{enumerate}
    \item Recall the relative Euler sequence and tangent sequence associated to $\pi':\PP(f^{*}T^{\vee}_{M})\rightarrow \PP^{1}$:
    \begin{equation*}
        0 \rightarrow \Omega_{\pi'}(1) \rightarrow (\pi')^{*}(f^{*}T_{M}) \rightarrow \Ocal_{\PP(f^{*}T^{\vee}_{M})}(1) \rightarrow 0,
    \end{equation*}
    \begin{equation*}
        0 \rightarrow T_{\pi'} \rightarrow T_{\PP(f^{*}T^{\vee}_{M})} \xrightarrow{d \pi'} (\pi')^{*} T_{\PP^{1}} \rightarrow 0
    \end{equation*}
    where $\Omega_{\pi'}$ (resp. $T_{\pi'}$) is the relative differential bundle (resp. vertical tangent bundle).
    Taking the pullback via $s$, we obtain
    \begin{equation} \label{eqn:rel Euler for pi'}
        0 \rightarrow (s^{*}\Omega_{\pi'})(1) \rightarrow f^{*}T_{M} \xrightarrow{q} \Ocal_{\PP^{1}}(1) \rightarrow 0,
    \end{equation}
    \begin{equation} \label{eqn:rel tang for pi'}
        0 \rightarrow s^{*}T_{\pi'} \rightarrow s^{*}T_{\PP(f^{*}T^{\vee}_{M})} \xrightarrow{s^{*}(d \pi')} T_{\PP^{1}} \rightarrow 0.
    \end{equation}
    Notice that the sequence (\ref{eqn:rel tang for pi'}) is split, since $s^{*}(d\pi')$ has an inverse $ds: T_{\PP^{1}} \rightarrow s^{*}T_{\PP(f^{*}T^{\vee}_{M})}$.
    Therefore the normal bundle $N_{s}$ is isomorphic to $s^{*}T_{\pi'}$, and hence we obtain the sequence (\ref{eqn:section s vs free f}) from the sequence (\ref{eqn:rel Euler for pi'}).

    \item First, we show that $\tilde{f}$ is birational.
    Observe that it is not constant since $\tilde{f}^{*}\Ocal_{\PP(T^{\vee}_{M})}(1) \simeq \Ocal_{\PP^{1}}(1)$.
    Let $\nu : \PP^{1} \rightarrow \tilde{f}(\PP^{1})$ be the normalization map.
    Then we have a factorization of $\tilde{f}$
    \[
        \tilde{f} : \PP^{1} \xrightarrow{g} \PP^{1} \xrightarrow{\nu} \tilde{f}(\PP^{1}).
    \]
    Here, $g$ is a finite morphism, say of degree $d \ge 1$.
    Again since
    \[
        \Ocal_{\PP^{1}}(1) \simeq \tilde{f}^{*} \Ocal_{\PP(T^{\vee}_{M})}(1) \simeq g^{*} (\nu^{*}\Ocal_{\PP(T^{\vee}_{M})}(1)),
    \]
    if we write $\nu^{*}\Ocal_{\PP(T^{\vee}_{M})}(1) \simeq \Ocal_{\PP^{1}}(k)$ for some $k \in \ZZ$, then we have $\Ocal_{\PP^{1}}(1) \simeq \Ocal_{\PP^{1}}(dk)$, which is possible only when $d=k=1$.
    Therefore $g$ is an isomorphism, and $\tilde{f}$ is birational onto its image.

    Now assume that $\tilde{f}$ is unbendable and $f$ is immersive.
    It remains to show that $q$ is a split surjection.
    Observe that since $\tilde{f}$ is free and $\tilde{f}^{*}(d\pi) : \tilde{f}^{*}T_{\PP(T^{\vee}_{M})} \rightarrow f^{*}T_{M}$ is surjective, $f^{*}T_{M}$ is globally generated.
    Thus we can write
    \[
        f^{*}T_{M} \simeq (f^{*}T_{M})^{\ge 2} \oplus \Ocal_{\PP^{1}}(1)^{\oplus a} \oplus \Ocal_{\PP^{1}}^{\oplus b}, \quad a,\,b \in \ZZ_{\ge 0}, \quad a+b = (n+1) - \rank (f^{*}T_{M})^{\ge 2}.
    \]
    Then $\ker(q)$ contains $(f^{*}T_{M})^{\ge 2}$.
    If $q$ is not split, then $\ker(q)$ further contains $\Ocal_{\PP^{1}}(1)^{\oplus a}$, since any nonzero endomorphism of $\Ocal_{\PP^{1}}(1)$ is an isomorphism.
    So $q$ factors through a surjection $q':\Ocal_{\PP^{1}}^{\oplus b} \rightarrow \Ocal_{\PP^{1}}(1)$ (in particular, $b \ge 2$), and considering the degrees, we have $\ker(q') = \Ocal_{\PP^{1}}^{\oplus(b-2)} \oplus \Ocal_{\PP^{1}}(-1)$.
    In summary, if $q$ is not split, then we have
    \[
        \ker(q) \simeq (f^{*}T_{M})^{\ge 1} \oplus \Ocal_{\PP^{1}}^{\oplus(b-2)} \oplus \Ocal_{\PP^{1}}(-1).
    \]
    and so
    \begin{equation} \label{eqn:N_s when q not split}
        N_{s} \simeq \Ocal_{\PP^{1}}(2) \oplus \Ocal_{\PP^{1}}(1)^{\oplus(b-2)} \oplus ((f^{*}T_{M})^{\ge 1}(-1))^{\vee}
    \end{equation}
    by the sequence (\ref{eqn:section s vs free f}).
    On the other hand, since $f$ is immersive, so is $f'$.
    Then the injection
    \[
        df' : T_{\PP(f^{*}T^{\vee}_{M})} \rightarrow (f')^{*} T_{\PP(T^{\vee}_{M})},
    \]
    induces an injection
    \[
        s^{*}(df') : N_{s} \rightarrow N_{\tilde{f}}
    \]
    where $N_{\tilde{f}} \coloneqq \coker(d \tilde{f})$.
    This is impossible since $N_{\tilde{f}}$ is split into the sum of $\Ocal_{\PP^{1}}(1)$'s and $\Ocal_{\PP^{1}}$'s but $N_{s}$ contains an $\Ocal_{\PP^{1}}(2)$-factor by (\ref{eqn:N_s when q not split}).
    Therefore $q$ is split.
\end{enumerate}
\end{proof}

\begin{definition}
    Let $V$ be a vector bundle on $\PP^{1}$.
    Choose a point $o \in \PP^{1}$ and a hyperplane $H \subset V_{o}$.
    Define
    \[
        \Hom_{\surj}(V,\,\Ocal_{\PP^{1}}(1)) \coloneqq \{q \in \Hom(V,\,\Ocal_{\PP^{1}}(1)) : \text{$q$ is surjective}\},
    \]
    \[
        \Hom_{\split}(V,\,\Ocal_{\PP^{1}}(1)) \coloneqq \{q \in \Hom_{\surj}(V,\,\Ocal_{\PP^{1}}(1)) : \text{$q$ is a split surjection}\},
    \]
    \[
        \Hom_{\surj}(V,\,\Ocal_{\PP^{1}}(1);\,o,\, H) \coloneqq \{q \in \Hom_{\surj}(V,\,\Ocal_{\PP^{1}}(1)) : \ker q_{o} = H\},
    \]
    and
    \[
        \Hom_{\split}(V,\,\Ocal_{\PP^{1}}(1);\,o,\, H) \coloneqq \Hom_{\split}(V,\,\Ocal_{\PP^{1}}(1)) \cap \Hom_{\surj}(V,\,\Ocal_{\PP^{1}}(1);\,o,\, H).
    \]
\end{definition}

\begin{remark}\label{rmk:Hsp}
\begin{enumerate}
    \item If $\Hom_{\surj}(V,\,\Ocal_{\PP^{1}}(1);\,o,\, H)$ is nonempty, then $V^{\ge 2}_{o} \subset H$. If $\Hom_{\split}(V,\,\Ocal_{\PP^{1}}(1);\,o,\, H)$ is nonempty, then $V^{\ge1}_{o} \not\subset H$.
    \item $\Hom_{\surj}(V,\,\Ocal_{\PP^{1}}(1))$, $\Hom_{\split}(V,\,\Ocal_{\PP^{1}}(1))$, $\Hom_{\surj}(V,\,\Ocal_{\PP^{1}}(1);\,o,\, H)$ and $\Hom_{\split}(V,\,\Ocal_{\PP^{1}}(1);\,o,\, H)$ are homothety-invariant subsets of the vector space $\Hom(V,\,\Ocal_{\PP^{1}}(1))$.
\end{enumerate}
    
\end{remark}

\begin{proposition} \label{prop:Hsp}
    Let $V$ be a vector bundle on $\PP^{1}$.
    Choose a point $o \in \PP^{1}$ and a hyperplane $H \subset V_{o}$.
    Assume that $V^{\ge 2}_{o} \subset H$ but $V^{\ge 1}_{o} \not\subset H$ (in particular, the splitting type of $V$ has an $\Ocal_{\PP^{1}}(1)$-factor).
    Then $\Hom_{\split}(V,\,\Ocal_{\PP^{1}}(1);\,o,\, H) = W \setminus W'$ where $W$ is a vector subspace of dimension $\deg((V/V^{\ge2})^{\vee}(1))+1$ in $\Hom(V,\,\Ocal_{\PP^{1}}(1))$, and $W'$ is a vector subspace of $W$ of codimension 1.
    In particular,
    \[
        \PP(\Hom_{\split}(V,\,\Ocal_{\PP^{1}}(1);\,o,\, H)) \simeq \CC^{\deg((V/V^{\ge2})^{\vee}(1))}.
    \]
\end{proposition}

\begin{proof}
    Consider the following split short exact sequence:
    \[
        0 \rightarrow \Hom(V/V^{\ge1},\, \Ocal_{\PP^{1}}(1)) \rightarrow \Hom(V,\, \Ocal_{\PP^{1}}(1)) \rightarrow \Hom(V^{\ge1},\, \Ocal_{\PP^{1}}(1)) \rightarrow 0.
    \]
    Since a morphism $q:V \rightarrow \Ocal_{\PP^{1}}(1)$ is a split surjection if and only if $q|_{V^{\ge 1}}\not=0$, we have
    \begin{equation} \label{eqn:Hsp}
        \Hom_{\split}(V,\,\Ocal_{\PP^{1}}(1)) = \Hom(V,\,\Ocal_{\PP^{1}}(1)) \setminus \Hom(V/V^{\ge1},\,\Ocal_{\PP^{1}}(1)).
    \end{equation}
    On the other hand, consider the morphism
    \[
        \Phi:\Hom(V,\,\Ocal_{\PP^{1}}(1)) \rightarrow \Hom(V_{o}/V^{\ge 2}_{o},\,(\Ocal_{\PP^{1}}(1))_{o}), \quad q \mapsto (v \mod V^{\ge 2}_{o} \mapsto q(v)).
    \]
    This is well defined since any morphism $V \rightarrow \Ocal_{\PP^{1}}(1)$ annihilates $V^{\ge 2}$, i.e., $\Hom(V,\,\Ocal_{\PP^{1}}(1)) \simeq \Hom(V/V^{\ge2},\,\Ocal_{\PP^{1}}(1))$.
    Indeed, $\Phi$ can be written as follows:
    \begin{center}
        \begin{tikzcd}
            \Hom(V,\,\Ocal_{\PP^{1}}(1)) \arrow[rr, "\Phi", bend left=20] \arrow[r,"\simeq"] & \Hom(V/V^{\ge2},\, \Ocal_{\PP^{1}}(1)) \arrow[r] \arrow[d, "\simeq"] & \Hom(V_{o}/V^{\ge 2}_{o},\,(\Ocal_{\PP^{1}}(1))_{o}) \arrow[d, "\simeq"] \\
            &H^{0}(\PP^{1},\, (V/V^{\ge2})^{\vee}(1)) \arrow[r, "ev_{o}"]& ((V/V^{\ge2})^{\vee}(1))_{o}
        \end{tikzcd}
    \end{center}
    where $ev_{o}$ is the evaluation map.
    Since $(V/V^{\ge2})^{\vee}(1)$ is a globally generated vector bundle, $ev_{o}$ is surjective, and so is $\Phi$.
    Since $H$ contains $V^{\ge2}_{o}$, $H/V^{\ge2}_{o}$ is a hyperplane in $V_{o}/V^{\ge2}_{o}$, and defines a one-dimensional subspace
    \[
        l_{H} \coloneqq \Ann(H/V^{\ge2}_{o}) \subset \Hom(V_{o}/V^{\ge 2}_{o},\,(\Ocal_{\PP^{1}}(1))_{o}).
    \]
    Then $\Phi^{-1}(l_{H})$ is a linear subspace of dimension
    \begin{align*}
        \dim \Phi^{-1}(l_{H}) &=\dim \Hom(V,\,\Ocal_{\PP^{1}}(1)) - \codim(l_{H},\, \Hom(V_{o}/V^{\ge 2}_{o},\,(\Ocal_{\PP^{1}}(1))_{o})) \\
        &= h^{0}(\PP^{1},\,V^{\vee}(1)) - \rank(V/V^{\ge 2})+1 \\
        &= h^{0}(\PP^{1},\,(V/V^{\ge2})^{\vee}(1)) - \rank((V/V^{\ge 2})^{\vee}(1))+1 \\
        &= \deg((V/V^{\ge2})^{\vee}(1))+1 \quad (\because \text{$(V/V^{\ge2})^{\vee}(1)$ is globally generated}).
    \end{align*}
    By the construction, we have $\Hom_{\split}(V,\,\Ocal_{\PP^{1}}(1);\,o,\,H) = \Hom_{\split}(V,\,\Ocal_{\PP^{1}}(1)) \cap \Phi^{-1}(l_{H} \setminus \{0\})$.
    We claim that $\Phi^{-1}(l_{H} \setminus \{0\}) \subset \Hom_{\split}(V,\,\Ocal_{\PP^{1}}(1))$.
    By (\ref{eqn:Hsp}), it means that for any morphism $q : V \rightarrow \Ocal_{\PP^{1}}(1)$ such that $\ker q_{o} = H$, we have $q \not\in \Hom(V/V^{\ge1},\,\Ocal_{\PP^{1}}(1))$.
    In fact, if $q \in \Hom(V/V^{\ge 1},\,\Ocal_{\PP^{1}}(1))$, then $q$ annihilates $V^{\ge1}$, and hence $V^{\ge1}_{o} \subset H$, which is a contradiction.
    Now we conclude the proof by
    \[
        \Hom_{\split}(V,\,\Ocal_{\PP^{1}}(1);\,o,\,H) = \Phi^{-1}(l_{H}\setminus\{0\}) = \Phi^{-1}(l_{H}) \setminus \Phi^{-1}(0).
    \]
\end{proof}

\begin{theorem}\label{thm:lifting criterion}
    Let $z \in \PP(T^{\vee}_{M})$ be a point such that rational curves on $\PP(T^{\vee}_{M})$ through $z$ and rational curves on $M$ through $\pi(z)$ are free.
    Let $f : \PP^{1} \rightarrow M$ be an immersive morphism, and $o \in \PP^{1}$ a point such that $f(o) = \pi(z)$.
    Then the following are equivalent:
    \begin{enumerate}
        \item\label{item1:thm:lifting criterion} There exists a morphism $\tilde{f} : \PP^{1} \rightarrow \PP(T^{\vee}_{M})$ such that $\tilde{f}^{*} \Ocal_{\PP(T^{\vee}_{M})}(1) \simeq \Ocal_{\PP^{1}}(1)$, $\pi \circ \tilde{f} = f$ and $\tilde{f}(o) = z$;
        \item\label{item2:thm:lifting criterion} We have $(f^{*}T_{M})^{\ge 2}_{o} \subset \Ann(z)$ but $(f^{*}T_{M})^{\ge 1}_{o} \not\subset \Ann(z)$.
    \end{enumerate}
    If one of the conditions holds, then any $\tilde{f}$ as in the condition (\ref{item1:thm:lifting criterion}) is the normalization map for $\tilde{f}(\PP^{1})$.
    Moreover, if $b \in \ZZ_{\ge 0}$ is the number of $\Ocal_{\PP^{1}}$-factors in the splitting type of $f^{*}T_{M}$, then there exists a one-to-one correspondence between the affine space $\CC^{b}$ and the set of $\tilde{f}$ as in the condition (\ref{item1:thm:lifting criterion}).
\end{theorem}
\begin{proof}
    First, we show that $\tilde{f}$ as in the condition (\ref{item1:thm:lifting criterion}) arises from a split surjection.
    Let $q$ be an element of $\Hom_{\surj}(f^{*}T_{M},\,\Ocal_{\PP^{1}}(1);\,o,\, \Ann(z))$ such that the associated morphism $\PP^{1} \rightarrow \PP(T^{\vee}_{M})$ over $f$ is $\tilde{f}$.
    Indeed, by Lemma~\ref{lemma:contact lines are from split surj}, $\tilde{f}$ is the normalization of its image $\tilde{f}(\PP^{1})$, and so by Theorem~\ref{thm:covering family is free}, $\tilde{f}$ is unbendable.
    Again by Lemma~\ref{lemma:contact lines are from split surj}, $q$ is split.
    Now we see that (\ref{item1:thm:lifting criterion}) $\Rightarrow$ (\ref{item2:thm:lifting criterion}) (cf. Remark~\ref{rmk:Hsp}).

    It remains to show that when we have (\ref{item2:thm:lifting criterion}), then the condition (\ref{item1:thm:lifting criterion}) holds true and $\tilde{f}$ as in (\ref{item1:thm:lifting criterion}) are parametrized by $\CC^{b}$.
    Observe that these are parametrized by $\PP(\Hom_{\split}(f^{*}T_{M},\,\Ocal_{\PP^{1}}(1);\,o,\,\Ann(z)))$, which is the affine space of dimension $\deg((f^{*}T_{M}/(f^{*}T_{M})^{\ge 2})^{\vee}(1))$ by Proposition~\ref{prop:Hsp}.
    In particular, (\ref{item2:thm:lifting criterion}) $\Rightarrow$ (\ref{item1:thm:lifting criterion}).
    Since $f$ is free by the assumption, we have
    \[
        \deg((f^{*}T_{M}/(f^{*}T_{M})^{\ge 2})^{\vee}(1)) = \#\{\text{$\Ocal$-factors in $f^{*}T_{M}$}\} = b,
    \]
    which completes the proof.
\end{proof}

With the aid of deformation theory, Theorem~\ref{thm:lifting criterion} can be formulated in terms of a family of rational curves on $M$ and that of contact lines on $\PP(T^{\vee}_{M})$.
This works particularly well in the case where $f(\PP^{1})$ is a line on $M$, and the precise statement is as follows:

\begin{theorem}\label{thm:VMRT formulation}
    Let $z \in \PP(T^{\vee}_{M})$ be a point, and put $m \coloneqq \pi(z) \in M$.
    Assume that rational curves on $\PP(T^{\vee}_{M})$ through $z$ and rational curves on $M$ through $m$ are free.
    Let $\Kscr$ be a family of unbendable curves on $M$ of type $(1^{a},\,0^{n-a})$ for some $0 \le a \le n$.
    Choose an embedding $M \subset \PP^{N}$, and assume that $\Kscr$ is a family of lines in $\PP^{N}$.
    Then $\VMRT^{\Kscr}_{m}$ is a smooth subvariety of pure dimension $a$ in $\PP(T_{M,\,m})$.
    Moreover, the following are equivalent:
    \begin{enumerate}
        \item\label{item1:thm:VMRT formulation} There exists a point $\alpha \in \VMRT^{\Kscr}_{m} \cap \PP(\Ann(z))$ such that the projective tanget space of $\VMRT^{\Kscr}_{m}$ at $\alpha$ is not contained in $\PP(\Ann(z))$. (In particular, $a \ge 1$.)
        \item\label{item2:thm:VMRT formulation} There exists a contact line $C$ on $\PP(T^{\vee}_{M})$ through $z$ such that $\pi$ sends $C$ isomorphically onto a member of $\Kscr$.
    \end{enumerate}
    If one of the conditions holds, then there exists a unique family $\Hscr$ of contact lines on $\PP(T^{\vee}_{M})$ such that any $C$ as in the condition (\ref{item2:thm:VMRT formulation}) is a member of $\Hscr$.
    Furthermore, $\pi$ sends every member of $\Hscr$ isomorphically onto a member of $\Kscr$, and for the well-defined morphism
    \[
        \Pi: \VMRT^{\Hscr}_{z} \rightarrow \VMRT^{\Kscr}_{m} \cap \PP(\Ann(z)), \quad [T_{C,\, z}] \mapsto [T_{\pi(C),\, m}],
    \]
    the image $\Pi(\VMRT^{\Hscr}_{z})$ consists of $\alpha$ as in the condition (\ref{item1:thm:VMRT formulation}), and for such an $\alpha$, we have a finite and surjective morphism $\CC^{n-a} \rightarrow \Pi^{-1}(\alpha)$.
\end{theorem}
\begin{proof}
    Since every member $l$ of $\Kscr$ through $m$ is a free line on $M$, $l$ is unbendable (indeed, we have $N_{l/M} \subset N_{l/\PP^{N}} \simeq \Ocal_{\PP^{1}}(1)^{\oplus (N-1)}$).
    In this situation, it is well known that $\VMRT^{\Kscr}_{m}$ is nowhere singular and of pure dimension $a$.
    It is also well known that the projective tangent space of $\VMRT^{\Kscr}_{[T_{l,\,m}]}$ at $[T_{l,\,m}]$ is given as $\PP (T_{M}|_{l})^{\ge 1}_{m}$.
    For details, we refer to \cite[Theorem 1.3, Propositions 1.4, 1.5 and 2.3]{Hwang01VMRT}.
    
    Now we show the equivalence of the conditions (\ref{item1:thm:VMRT formulation}) and (\ref{item2:thm:VMRT formulation}).
    If $\alpha$ is as in the condition (\ref{item1:thm:VMRT formulation}), then considering the line $l$ on $M$ satisfying $[T_{l,\,m}] = \alpha$, Theorem~\ref{thm:lifting criterion}, together with the discussion in the previous paragraph, implies that there exists a contact line $C$ through $z$ sent onto $l$ bijectively under $\pi$.
    Since $l$ is smooth, $\pi|_{C}:C\rightarrow l$ must be an isomorphism.
    Conversely, if $C$ is as in the condition (\ref{item2:thm:VMRT formulation}), then again Theorem~\ref{thm:lifting criterion} implies that $\alpha = [T_{\pi(C),\,m}]$ is as in the condition (\ref{item1:thm:VMRT formulation}).
    Notice that as a byproduct, we have proven the following:
    \begin{claim} \label{claim:thm:VMRT formulation}
        In $\PP(T_{M,\,m})$, the subset of $\alpha$ as in the condition (\ref{item1:thm:VMRT formulation}) and the subset of $[T_{\pi(C),\,m}]$ for $C$ as in the condition (\ref{item2:thm:VMRT formulation}) coincide.
    \end{claim}

    Next, assuming the conditions (\ref{item1:thm:VMRT formulation}) and (\ref{item2:thm:VMRT formulation}), we construct a family of contact lines on $\PP(T^{\vee}_{M})$ including $C$ as in the condition (\ref{item2:thm:VMRT formulation}).
    Define $A \coloneqq \Ocal_{\PP^{N}}(1)|_{M}$ so that members of $\Kscr$ are $A$-lines.
    Denote by $\Kscr^{\text{free}}$ the open subset of $\Kscr$ consisting of free members, and by
    \[
        \Kscr^{\text{free}} \xleftarrow{\rho} \text{Univ} \xrightarrow{\mu} M
    \]
    the family of lines parametrized by $\Kscr^{\text{free}}$.
    \begin{claim}\label{claim2:thm:VMRT formulation}
        The sheaf $\rho_{*}\mu^{*}(T^{\vee}_{M} \otimes A)$ is locally free on $\Kscr^{\text{free}}$.
        Moreover, for a point $x \in \Kscr^{\text{free}}$ and the corresponding line $l_{x} \coloneqq \mu(\rho^{-1}(x))$, the fiber of $\rho_{*}\mu^{*}(T^{\vee}_{M} \otimes A)$ over $x$ is naturally isomorphic to $\Hom(T_{M}|_{l_{x}},\, \Ocal_{l_{x}}(1))$.
    \end{claim}
    \begin{proof}[Proof of Claim~\ref{claim2:thm:VMRT formulation}]
        We apply the Grauert theorem \cite[Corollary~12.9, Ch.~III]{HartshorneBook}.
        First recall that $\rho$ is a (flat) $\PP^{1}$-bundle, and thus $\mu^{*}(T^{\vee}_{M} \otimes A)$ is flat over $\Kscr^{\text{free}}$.
        Since $\Kscr^{\text{free}}$ is the locus of free lines, it is smooth, and for any point $x \in \Kscr^{\text{free}}$, the corresponding line $l_{x}$ is unbendable and of type $(1^{a},\,0^{n-a})$.
        Moreover, since
        \begin{align*}
            \dim H^{0}(\rho^{-1}(x),\,\mu^{*}(T^{\vee}_{M} \otimes A) |_{\rho^{-1}(x)}) &= \dim H^{0}(l_{x},\,T^{\vee}_{M} \otimes A |_{l_{x}}) \\
            &= \dim H^{0}(\PP^{1},\, \Ocal_{\PP^{1}}(1)^{\oplus(n-a)} \oplus \Ocal_{\PP^{1}}^{\oplus a} \oplus \Ocal_{\PP^{1}}(-1)) \\
            &= 2n-a
        \end{align*}
        is constant, the Grauert theorem applies, and hence $\rho_{*}\mu^{*}(T^{\vee} \otimes A)$ is a locally free sheaf whose fiber over $x$ is given as
        \[
            H^{0}(\rho^{-1}(x),\, \mu^{*}(T^{\vee}_{M}\otimes A)|_{\rho^{-1}(x)}) \simeq H^{0}(l_{x},\, (T^{\vee}_{M} \otimes A)|_{l_{x}}) \simeq \Hom(T_{M}|_{l_{x}},\, \Ocal_{\PP^{1}}(1)).
        \]
    \end{proof}
    Thanks to Claim~\ref{claim2:thm:VMRT formulation}, we can define
    \[
        \Hscr' \coloneqq \bigcup_{x \in \Kscr^{\text{free}}} \PP(\Hom_{\surj}(T_{M}|_{l_{x}},\,\Ocal_{\PP^{1}}(1))) \subset \PP(\rho_{*}\mu^{*}(T^{\vee}_{M} \otimes A)).
    \]
    Then $\Hscr'$ is an open subset of the total space of the projective bundle $\PP(\rho_{*}\mu^{*}(T^{\vee}_{M} \otimes A))$, parametrizing contact lines that are sent onto members of $\Kscr^{\text{free}}$ isomorphically.
    Therefore we have a natural injective morphism
    \[
        \varphi:\Hscr' \rightarrow \Hilb_{2t+1}(\PP(T^{\vee}_{M}),\, \pi^{*}A \otimes \Ocal_{\pi}(1))
    \]
    into the Hilbert scheme of subschemes of $\PP(T^{\vee}_{M})$ whose Hilbert polynomials with respect to the ample line bundle $\pi^{*}A \otimes \Ocal_{\pi}(1)$ is equal to $2t+1$.
    Since the image $\varphi(\Hscr')$ contains all points represented by $C$ as in the condition (\ref{item2:thm:VMRT formulation}) and such points are smooth points in the Hilbert scheme (cf. \cite[Theorems 2.8 and 2.15, Ch.~I]{Kollar96}), $\varphi(\Hscr')$ is contained in a unique irreducible component $\Hscr''$ of the Hilbert scheme.
    Under the Hilbert--Chow morphism, $\Hscr''$ is sent onto an irreducible component of the Chow scheme of $\PP(T^{\vee}_{M})$, and its normalization, denoted by $\Hscr$, is a unique irreducible component of $\RatCurves(\PP(T^{\vee}_{M}))$ parametrizing all points represented by $C$ as in the condition (\ref{item2:thm:VMRT formulation}).

    Now we claim that for any member $C'$ of $\Hscr$, $\pi|_{C'} : C' \rightarrow \pi(C')$ is an isomorphism, and $\pi(C')$ is a member of $\Kscr$ (in particular, $C'$ is smooth).
    By the projection formula, for $C$ as in the condition (\ref{item2:thm:VMRT formulation}), we have
    \[
        1=\deg_{\pi(C)}(A) = \deg_{C}(\pi^{*}A)=\deg_{C'}(\pi^{*}A) = \left\{\begin{array}{ll}
            0 & \text{if $\pi(C')$ is a point;} \\
            \deg_{\pi(C')}(A) \cdot \deg(\pi|_{C'}) & \text{if $\pi(C')$ is a curve.}
        \end{array}\right. 
    \]
    This is possible only when $\pi(C')$ is a curve and
    \[
        \deg_{\pi(C')}(A) = \deg(\pi|_{C'}) = 1,
    \]
    i.e., $\pi$ sends $C'$ onto a line on $\PP^{N}$ isomorphically.
    Therefore if we denote by $\Hilb_{t+1}(M,\,A)$ the Hilbert schemes of lines on $M \subset \PP^{N}$, then we have a morphism
    \[
        \pi_{*}:\Hscr \rightarrow \Hilb_{t+1}(M,\,A), \quad C' \mapsto \pi(C').
    \]
    On the other hand, since the Chow variety of lines on $M$ is isomorphic to the semi-normalization of $\Hilb_{t+1}(M,\,A)$ (cf. \cite[Theorem~6.3 and Corollary~6.6.1, Ch.~I]{Kollar96}), we have a natural finite morphism
    \[
        \iota : \Kscr \rightarrow \Hilb_{t+1}(M,\,A)
    \]
    onto an irreducible component.
    Since $\pi_{*}(\Hscr)$ contains the point represented by $\pi(C)$, which is a smooth point of $\Hilb_{t+1}(M,\,A)$ (cf. \cite[Theorems 2.8 and 2.15, Ch.~I]{Kollar96}), we see that $\pi_{*}(\Hscr) \subset \iota(\Kscr)$.
    It follows that $\pi(C')$ is a member of $\Kscr$.

    Now the well-definedness of $\Pi$ in the statement follows.
    The description of the image of $\Pi$ follows from Claim~\ref{claim:thm:VMRT formulation}.
    It remains to show that for $\alpha$ as in the condition (\ref{item1:thm:VMRT formulation}), we have a finite and surjective morphism $\CC^{n-a}\rightarrow\Pi^{-1}(\alpha)$.
    If $l_{\alpha}$ is the line with $[T_{l_{\alpha},\,m}] = \alpha$, then
    \[
        \Pi^{-1}(\alpha) = \{[T_{C,\,z}] \in \PP(T_{\PP(T^{\vee}_{M}),\,z}) : \text{$C$ is a member of $\Hscr_{z}$ and $\pi(C) = l_{\alpha}$}\}.
    \]
    Since $\PP (\Hom_{\split}(T_{M}|_{l_{\alpha}},\,\Ocal_{\PP^{1}}(1);\,m,\, \Ann(z)))$ parametrizes members of $\Hscr_{z}$ sent onto $l_{\alpha}$, by taking their tangent directions at $z$, we obtain a surjective morphism
    \[
        \PP (\Hom_{\split}(T_{M}|_{l_{\alpha}},\,\Ocal_{\PP^{1}}(1);\,m,\, \Ann(z))) \rightarrow \Pi^{-1}(\alpha).
    \]
    This is a finite morphism since every member of $\Hscr_{z}$ is unbendable.
    Now the statement follows from Theorem~\ref{thm:lifting criterion}.
\end{proof}

\subsection{Richardson orbit closures}\label{ssection:Richardson}

Now we come back to nilpotent orbits, and apply results of Section~\ref{ssection:projectivized cotangent bundles} to the case of \emph{Richardson orbit closures} (defined below), in which case singularities can be resolved by projectivized cotangent bundles.
As before, let $\gfr$ be a semisimple Lie algebra, and $G$ the associated simply connected Lie group.
Let $\kappa$ be the Killing form of $\gfr$.
Again, let $Z = \PP(O) \subset \PP(\gfr)$ be a nilpotent orbit, $z = [v] \in Z$ a point, and $\overline{Z} \subset \PP(\gfr)$ the nilpotent orbit closure.
Write $2n+1 = \dim Z$.

\begin{definition} \label{defn:Richardson}
    Let $P$ be a parabolic subgroup of $G$, and $\ufr_{P} = \Lie(R^{u}(P))$ the Lie algebra of the unipotent radical, identified with $T^{\vee}_{G/P,\,eP}$ via the isomorphism $w \in \ufr_{P} \mapsto \kappa(-,\,w)\in T^{\vee}_{G/P,\,eP}$.
    Consider the diagram
    \begin{equation} \label{diagram:Springer morphism}
    \begin{tikzcd}
        \PP(T^{\vee}_{G/P}) \simeq G \times^{P} \PP(\ufr_{P}) \arrow[r, "s_{P}"] \arrow[d, "\pi_{P}"] & \PP(\gfr) \\
        G/P&
    \end{tikzcd}
    \end{equation}
    where $\pi_{P}$ is the natural projection and $s_{P}$ is the \emph{Springer morphism}, defined as follows:
    \begin{equation*}
    s_{P}((g,\,[w])^{\bullet}) \coloneqq [\Ad_{g}w], \quad \forall (g,\,[w])^{\bullet} \in G \times^{P} \PP(\ufr_{P}).
    \end{equation*}
    Since $\ufr_{P}$ contains an open $\Ad_{P}$-orbit \cite[Proposition~5]{Ri74conjugacy}, $\PP(T^{\vee}_{G/P})$ contains an open $G$-orbit, and the image of $s_{P}$ is a nilpotent orbit closure, say $\overline{Z_{P}} \subset \PP(\gfr)$.
    Nilpotent orbits of the form $Z_{P}$ are called \emph{Richardson orbits} in $\PP(\gfr)$.
    Given a Richardson orbit $Z \subset \PP(\gfr)$, a parabolic subgroup $P$ of $G$ satisfying $Z = Z_{P}$ is called a \emph{polarization} of $Z$, and the set of polarizations of $Z$ is denoted by $\Pol(Z)$.
\end{definition}

In the diagram (\ref{diagram:Springer morphism}), the Springer morphism $s_{P}$ is generically finite onto the image $\overline{Z_{P}}$ and so $\dim G/P = n+1$, but not birational in general.
Nonetheless, we have an analogue of Proposition~\ref{prop:Theta KKS}:

\begin{proposition} \label{prop:Springer map is contacto}
    Suppose that $Z$ is Richardson, and let $P \in \Pol(Z)$.
    Let $\widetilde{Z}$ be the open $G$-orbit in $\PP(T^{\vee}_{G/P})$.
    Then $(s_{P})^{*}(\Ocal_{\PP(\gfr)}(1)) = \Ocal_{\PP(T^{\vee}_{G/P})}(1)$, and under the restriction $s_{P}|_{\widetilde{Z}} : \widetilde{Z} \rightarrow Z$, the pullback of the contact form $\theta_{KKS} : T_{Z} \rightarrow \Ocal_{\PP(\gfr)}(1)|_{Z}$ is the restriction of the contact form $\theta_{\can} : T_{\PP(T^{\vee}_{G/P})} \rightarrow \Ocal_{\PP(T^{\vee}_{G/P})}(1)$ (cf. Example~\ref{ex:PT*M}).
\end{proposition}
\begin{proof}
    By the definition of $s_{P}$, we have $(s_{P})^{*}(\Ocal_{\PP(\gfr)}(1)) = \Ocal_{\PP(T^{\vee}_{G/P})}(1)$.
    Since the restriction $s_{P}|_{\widetilde{Z}} : \widetilde{Z} \rightarrow Z$ is a finite covering, we also have $(s_{P}|_{\widetilde{Z}})^{*} (T_{Z}) = T_{\widetilde{Z}}$.
    Since $s_{P}$ is $G$-equivariant, it is enough to show that the contact forms $\theta_{KKS}$ and $\theta_{\can}$ agree at $\tilde{z}=(1,\,[v])^{\bullet}$, which follows from the following commutative diagram:
    \begin{center}
        \begin{tikzcd}
            T_{\widetilde{Z},\,\tilde{z}} \arrow[r,"\theta_{\can}"] \arrow[d, "d s_{P}"] & (\Ocal_{\PP(T^{\vee}_{G/P})}(1))_{\tilde{z}} = (\CC v)^{\vee} \arrow[d, equal] & w \arrow[r, mapsto] & \langle d\pi_{P}(w),\,- \rangle = \kappa (-, w) \\
            T_{Z,\,z} \arrow[r, "\theta_{KKS}"] & (\Ocal_{\PP(\gfr)}(1))_{z} = (\CC v)^{\vee} & u \arrow[r, mapsto]& \kappa(-,\,u).
        \end{tikzcd}
    \end{center}
\end{proof}

If $s_{P}$ is birational, then lines on $\overline{Z}$ behave similarly with contact lines on $\PP(T^{\vee}_{G/P})$.
In the rest of this section, we mainly discuss this situation:
\begin{assumption} \label{assumption:Springer resolution}
    $Z$ is Richardson, and $P \in \Pol(Z)$ such that $z$ is contained in the open dense $\Ad_{P}$-orbit in $\PP(\ufr_{P})$, and $s_{P} : \PP(T^{\vee}_{G/P}) \rightarrow \overline{Z}$ is birational.
\end{assumption}

\begin{example}
\begin{enumerate}
    \item Suppose that $\gfr$ is a simple Lie algebra of type $A$.
    Then every nilpotent orbit is Richardson, admitting a polarization satisfying Assumption~\ref{assumption:Springer resolution}.
    For a simple Lie algebra of other classical type, see for instance \cite[Proposition~3.1]{Fu07extremal} and the references therein.
    \item For any semisimple $\gfr$, an even (i.e., in (\ref{eqn:eigensp decomp of g}), $\gfr_{k} = 0$ for all $k$ odd) nilpotent orbit is Richardson, and its Jacobson--Morozov resolution defines a polarization satisfying Assumption~\ref{assumption:Springer resolution}.
\end{enumerate}    
\end{example}

When we work with Assumption~\ref{assumption:Springer resolution}, we shall use the following notations:
$\tilde{z}$ is the strict transform of $z$ in $\PP(T^{\vee}_{G/P})$, and $m = \pi_{P}(\tilde{z}) (=1.P)$.
As before, we denote by $\Ann(\tilde{z})$ the hyperplane in $T_{G/P,\,m}$ associated to $\tilde{z}\in \PP(T^{\vee}_{G/P,\,m})$.
Given a line $l$ on $\overline{Z}$ through $z$, we denote by $l^{\sim}$ the strict transform of $l$ via $s_{P}$.
Observe that by Proposition~\ref{prop:Springer map is contacto} and the projection formula, $l^{\sim}$ is a contact line on $\PP(T^{\vee}_{G/P})$ through $\tilde{z}$, and every contact line through $\tilde{z}$ is of this form.

\begin{proposition} \label{prop: VMRT of Richardson is Leg}
    Under Assumption~\ref{assumption:Springer resolution}, $F(\overline{Z},\,z)$ is of pure dimension $n-1$, i.e., every irreducible component is Legendrian.
\end{proposition}
\begin{proof}
    Since lines through $z$ and contact lines through $\tilde{z}$ are identified via $s_{P}$, under the isomorphism
    \[
        ds_{P} : \PP(T_{\PP(T^{\vee}_{G/P}),\, \tilde{z}}) \rightarrow \PP(T_{\overline{Z},\,z}),
    \]
    we have
    \[
        F(\overline{Z},\,z) = \bigcup_{\Hscr} ds_{P}(\VMRT^{\Hscr}_{\tilde{z}})
    \]
    where $\Hscr$ runs over families of contact lines on $\PP(T^{\vee}_{G/P})$.
    By Theorem~\ref{thm:covering family is free}, each $ds_{P}(\VMRT^{\Hscr}_{\tilde{z}})$ is of pure dimension $n-1$, and hence the statement follows from Corollary~\ref{coro:dim VMRT is at most n-1}.
\end{proof}

\begin{proposition} \label{prop:contact lines are almost lines}
    Under Assumption~\ref{assumption:Springer resolution}, any contact line $C$ on $\PP(T^{\vee}_{G/P})$ passing through $\tilde{z}$ is smooth.
    Furthermore, either $\pi_{P}(C) = m$ or the restriction $\pi_{P}|_{C} : C \rightarrow \pi_{P}(C)$ is an isomorphism.
\end{proposition}
\begin{proof}
    Observe that $s_{P}|_{C} : C \rightarrow s_{P}(C)$ is birational, and $s_{P}(C)$ is a line by Proposition~\ref{prop:Springer map is contacto} and the projection formula.
    Thus $s_{P}|_{C}$ is an isomorphism and $C$ is smooth.

    Now assume that $\pi_{P}(C)$ is not a point, and show that $\pi_{P}|_{C} : C \rightarrow \pi_{P}(C)$ is an isomorphism.
    \begin{itemize}
        \item $\pi_{P}|_{C} : C \rightarrow \pi_{P}(C)$ is immersive.
        To see this, suppose that $w \in C$ and $d\pi_{P}(T_{C,\,w}) = 0$.
        Write $\pi_{P}(w) = g.m$ for some $g \in G$, and then $C$ is tangent to a line $l$ on $(\pi_{P})^{-1}(g.m) = \PP(T^{\vee}_{G/P,\,g.m})$ at $w$.
        Since $s_{P}|_{C} : C \rightarrow s_{P}(C)$ and $s_{P}|_{(\pi_{P})^{-1}(g.m)} : (\pi_{P})^{-1}(g.m) \rightarrow \PP(\Ad_{g}\ufr_{P})$ are isomorphisms, $s_{P}(C)$ and $s_{P}(l)$ are lines tangent to each other at $s_{P}(w)$, and hence $s_{P}(C) = s_{P}(l)$.
        Therefore $C = (s_{P}(C))^{\sim}$ is contained in $\PP(\Ad_{g}\ufr_{P})^{\sim} = \pi_{P}^{-1}(g.m)$, which is a contradiction.

        \item $\pi_{P}|_{C} : C \rightarrow \pi_{P}(C)$ is bijective.
        To see this, suppose that $w,\,w' \in C$ and $\pi_{P}(w) =\pi_{P}(w')$.
        Then $s_{P}(C)$ is the line joining $s_{P}(w)$ and $s_{P}(w')$.
        If we write $\pi_{P}(w) =\pi_{P}(w') = g.m$ for some $g \in G$, then since $s_{P}(w),\,s_{P}(w') \in s_{P}(\pi_{P}^{-1}(g.m)) = \PP(\Ad_{g}\ufr_{P})$, $s_{P}(C)$ is contained in the linear subspace $\PP(\Ad_{g}\ufr_{P})$.
        Thus $C = (s_{P}(C))^{\sim}$ is contained in $\PP(\Ad_{g}\ufr_{P})^{\sim} = \pi_{P}^{-1}(g.m)$, which is a contradiction.
    \end{itemize}
    Being an immersive bijection, $\pi_{P}|_{C}$ is an isomorphism.
\end{proof}

Now we are ready to apply results of Section~\ref{ssection:projectivized cotangent bundles}.
To this end, note that in the notation of Assumption~\ref{assumption:Springer resolution}, rational curves on $\PP(T^{\vee}_{G/P})$ through $\tilde{z}$ and rational curves on $G/P$ through $m$ are free.
Indeed, for a very general point in $\PP(T^{\vee}_{G/P})$, rational curves through it and those through its image under $\pi_{P}$ are free by Theorem~\ref{thm:covering family is free}, and then since the $G$-orbit $G.\tilde{z}$ is open in $\PP(T^{\vee}_{G/P})$, it holds for $\tilde{z}$ and $m = \pi_{P}(\tilde{z})$.
Therefore Theorems~\ref{thm:lifting criterion} and \ref{thm:VMRT formulation for G/P} can be applied to $\tilde{z} \in \PP(T^{\vee}_{G/P})$.

\begin{theorem}\label{thm:lifting criterion in G/P}
    Under Assumption~\ref{assumption:Springer resolution}, there is a one-to-one correspondence between
    \begin{itemize}
        \item lines $l$ on $\overline{Z}$ through $z$ such that $l \not\subset \PP(\ufr_{P})$, and
        \item pairs $(C, [q])$ of a smooth rational curve $C$ on $G/P$ through $m$ such that $(T_{G/P}|_{C})^{\ge 2} \subset \Ann(\tilde{z})$ but $(T_{G/P}|_{C})^{\ge 1} \not\subset \Ann(\tilde{z})$ and $[q] \in \PP(\Hom_{\split}(T_{G/P}|_{C},\,\Ocal_{\PP^{1}}(1);\,m,\, \Ann(\tilde{z})))$.
    \end{itemize}
    These are related as follows: $C = \pi_{P}(l^{\sim})$, and $l^{\sim}$ is the section of $\PP(T^{\vee}_{G/P})|_{C} \rightarrow C$ determined by $q$.
\end{theorem}
\begin{proof}
    Notice that given a line $l$ on $\overline{Z}$ through $z$, $\pi_{P}(l^{\sim}) = m$ if and only if $l \subset \PP(\ufr_{P})$.
    Then the statement follows from Proposition~\ref{prop:contact lines are almost lines} and Theorem~\ref{thm:lifting criterion}.
\end{proof}

Next, Theorem~\ref{thm:VMRT formulation} can be improved in the following form.
The key is that $\tilde{z}$ is not only a general point of $\PP(T^{\vee}_{G/P})$, but also a general point of $(\pi_{P})^{-1}(m) \simeq \PP(\ufr_{P})$.

\begin{theorem}\label{thm:VMRT formulation for G/P}
    Under Assumption~\ref{assumption:Springer resolution}, let $\Kscr$ be a family of lines on $G/P$ with respect to some polarization.
    Let $a$ be the integer such that $0 \le a \le n$ and $\Kscr$ is unbendable of type $(1^{a},\,0^{n-a})$.
    Then the following are equivalent:
    \begin{enumerate}
        \item\label{item1:thm:VMRT formulation for G/P} We have $a \ge 1$; and
        \item\label{item2:thm:VMRT formulation for G/P} There exists a line $l$ on $\overline{Z}$ through $z$ such that $\pi_{P}(l^{\sim})$ is a member of $\Kscr$.
    \end{enumerate}
    Suppose that one of the conditions holds.
    If $X_{\Kscr}$ is the subset of $F(\overline{Z},\,z)$ consisting of all points of the form $[T_{l,\,z}]$ for $l$ as in the condition (\ref{item2:thm:VMRT formulation for G/P}), then $X_{\Kscr}$ is locally closed in $F(\overline{Z},\,z)$ and of pure dimension $n-1$, and the natural morphism
    \[
        X_{\Kscr} \rightarrow \VMRT^{\Kscr}_{m} \cap \PP(\Ann(\tilde{z})), \quad [T_{l,\,z}] \mapsto [T_{\pi_{P}(l^{\sim}),\,m}]
    \]
    is a surjection whose fibers are irreducible and of dimension $n-a$.
\end{theorem}

\begin{remark}\label{rmk:VMRT of G/P}
    For any rational homogeneous space $G/P$, not necessarily satisfying Assumption~\ref{assumption:Springer resolution}, let $\Kscr$ be a family of lines on $G/P$.
    Then there exists a $G$-equivariant surjection $G/P \rightarrow G/Q$ such that members of $\Kscr$ are contracted and $Q/P$ is a rational homogeneous space of Picard number one.
    Thus the VMRT of $\Kscr$ coincides with that of a (unique) family of lines on $Q/P$, which is known to be irreducible (see for example \cite{LandsbergManivel2003ProjectiveGeometry}).
    That is, $\VMRT^{\Kscr}_{1.P}$ is a smooth projective irreducible subvariety of $\PP(T_{G/P,\,1.P})$ (cf. Theorem~\ref{thm:VMRT formulation}).

    Now in the setup of Theorem~\ref{thm:VMRT formulation for G/P}, it follows that if $a (=\dim \VMRT^{\Kscr}_{m}) \ge 2$, then $\VMRT^{\Kscr}_{m} \cap \PP(\Ann(\tilde{z}))$ is smooth and irreducible, thanks to the Bertini theorem.
    Indeed, by the choice of $\tilde{z}$, the $\Ad_{P}$-orbit of $\tilde{z}$ is open in $\ufr_{P} \simeq T^{\vee}_{G/P,\,m}$, while $\VMRT^{\Kscr}_{m}$ is $P$-invariant.
\end{remark}

\begin{proof}[Proof of Theorem~\ref{thm:VMRT formulation for G/P}]
    Thanks to Proposition~\ref{prop:contact lines are almost lines}, the condition (\ref{item2:thm:VMRT formulation for G/P}) is equivalent to the condition (\ref{item2:thm:VMRT formulation}) in Theorem~\ref{thm:VMRT formulation}.
    Apparently, the condition (\ref{item1:thm:VMRT formulation}) in Theorem~\ref{thm:VMRT formulation} implies the condition (\ref{item1:thm:VMRT formulation for G/P}).
    So by Theorem~\ref{thm:VMRT formulation}, to show the equivalence of the two conditions in the statement, it is enough to show that the condition (\ref{item1:thm:VMRT formulation for G/P}) implies the condition (\ref{item1:thm:VMRT formulation}) in Theorem~\ref{thm:VMRT formulation}.
    In fact, if $a\ge 1$, then since $\dim \VMRT^{\Kscr}_{m} = a$ and $\VMRT^{\Kscr}_{m} \cap \PP(\Ann(\tilde{z}))$ is a general hyperplane section in $\PP(T_{G/P,\,m})$ by the choice of $\tilde{z}$ (cf. Remark~\ref{rmk:VMRT of G/P}), the intersection is transversal by the Bertini theorem, and hence the condition (\ref{item1:thm:VMRT formulation}) in Theorem~\ref{thm:VMRT formulation} is satisfied.

    Now it follows that the conditions (\ref{item1:thm:VMRT formulation for G/P}) and (\ref{item2:thm:VMRT formulation for G/P}) are equivalent, and if one of them holds, then Theorem~\ref{thm:VMRT formulation} applies.
    Hence we have a unique family $\Hscr$ of contact lines on $\PP(T^{\vee}_{G/P})$ such that $\Hscr_{\tilde{z}}$ consists of all contact lines through $\tilde{z}$ sent onto members of $\Kscr$ (cf. Proposition~\ref{prop:contact lines are almost lines}).
    Its VMRT is of pure dimension $n-1$ (cf. Theorem~\ref{thm:covering family is free}), and equipped with the morphism
    \[
        \Pi : \VMRT^{\Hscr}_{\tilde{z}} \rightarrow \VMRT^{\Kscr}_{m} \cap \PP(\Ann(\tilde{z}))
    \]
    with irreducible fibers of dimension $n-a$.
    Furthermore, $\Pi$ is surjective, since the condition (\ref{item1:thm:VMRT formulation}) in Theorem~\ref{thm:VMRT formulation} is satisfied for any point of $\VMRT^{\Kscr}_{m} \cap \PP(\Ann(\tilde{z}))$ by the transversality of the intersection.
    Since $s_{P}$ identifies lines through $z$ and contact lines through $\tilde{z}$, we have $X_{\Kscr} = ds_{P}(\VMRT^{\Hscr}_{\tilde{z}})$.
\end{proof}

\begin{corollary} \label{coro:irreducible comp over VMRT of lines}
    Under Assumption~\ref{assumption:Springer resolution}, let $\Kscr$ be a family of lines on $G/P$ such that $\Kscr$ is unbendable of type $(1^{a},\,0^{n-a})$ for some $2 \le a \le n$.
    Then there exists a unique Legendrian component $\overline{X_{\Kscr}}$ of $F(\overline{Z},\,z)$ such that if $l$ is a line on $\overline{Z}$ through $z$ such that $\pi_{P}(l^{\sim})$ is a member of $\Kscr$, then $[T_{l,\,z}] \in \overline{X_{\Kscr}}$.
\end{corollary}
\begin{proof}
    By Remark~\ref{rmk:VMRT of G/P}, $\VMRT^{\Kscr}_{m} \cap \PP(\Ann(\tilde{z}))$ is irreducible, and thus $X_{\Kscr}$ in Theorem~\ref{thm:VMRT formulation for G/P} is also irreducible (see for example \cite{MuIrreducibility}).
\end{proof}

\subsection{Stratified Mukai flops}\label{ssection:Stratified Mukai flops}

Important examples satisfying Assumption~\ref{assumption:Springer resolution} arise from the \emph{stratified Mukai flops}.
In this section, we apply the results of the previous section to the case where $\overline{Z}$ is associated to a stratified Mukai flop.

Let $\gfr$ be a simple Lie algebra, and $P_{i}$ ($i=1,\,2$) maximal parabolic subgroups of $G$ associated to the single-marked Dynkin diagrams in Table~\ref{tab:Stratified Mukai flops}.
These are the only parabolic subgroups $P$ such that $G/P$ is of Picard number one and the moment map $T^{\vee}_{G/P} \rightarrow \overline{O_{P}}$ is a small resolution (cf. \cite[Proposition~5.1]{Na06}).
Moreover, $P_{1}$ and $P_{2}$ are not conjugate to each other in $G$, but have the same moment image $\overline{O_{P_{1}}} = \overline{O_{P_{2}}}$ (cf. \cite[Remark~5.2]{Na06}).
Therefore if we denote by $\overline{Z}$ the nilpotent orbit closure $\PP(\overline{O_{P_{1}}}) = \PP(\overline{O_{P_{2}}})$, then we have a diagram
\begin{equation} \label{eqn:Mukai flop}
    \begin{tikzcd}
        \PP(T^{\vee}_{G/P_{1}}) \arrow[rd, "s_{P_{1}}"] \arrow[d,"\pi_{P_{1}}"] \arrow[rr, dotted]&  & \PP(T^{\vee}_{G/P_{2}}) \arrow[ld, "s_{P_{2}}"] \arrow[d,"\pi_{P_{2}}"]\\
        G/P_{1} & \overline{Z} & G/P_{2},
    \end{tikzcd}
\end{equation}
satisfying Assumption~\ref{assumption:Springer resolution}.
The rational map $\PP(T^{\vee}_{G/P_{1}}) \dashrightarrow \PP(T^{\vee}_{G/P_{2}})$ in (\ref{eqn:Mukai flop}) is called the \emph{stratified Mukai flop} of type as in Table~\ref{tab:Stratified Mukai flops}.
In Table~\ref{tab:Stratified Mukai flops}, the nilpotent orbit closure $\overline{Z} = \PP(\overline{O})$ in (\ref{eqn:Mukai flop}) is also given, by specifying the label of $O$ in the notation of \cite{CollingwoodMcGovern1993NilpotentOrbits} (in type $A$ and $D$, see also Section~\ref{section:square zero}).

The number $d$ in Table~\ref{tab:Stratified Mukai flops} is defined as follows.
Notice that for a given point $z\in Z$, $F(\overline{Z},\,z)$ contains at least two components: $\PP(T_{\PP(\ufr_{P_{1}}),\,z})$ and $\PP(T_{\PP(\ufr_{P_{2}}),\,z})$.
We denote by $d$ the degree of rational curves $\pi_{P_{1}}(l^{\sim})$ on $G/P_{1}$ where $l$ is a general line on $\PP(\ufr_{P_{2}})$.
The degree $d$ is computed in \cite[\S5.D]{FL24}, and the result is given in Table~\ref{tab:Stratified Mukai flops}.
Another description of $d$ is the following:
\begin{proposition}[{Corollary of \cite[Proposition~5.8 and \S5.D]{FL24}}] \label{prop:deg in Mukai flop}
    In the diagram (\ref{eqn:Mukai flop}), let $H_{i}$ ($i=1,\,2$) be the ample Picard generator of $G/P_{i}$.
    Let $l$ be a line on $\overline{Z}$, intersecting with the open orbit $Z$.
    Denote by $l^{\sim}_{i}$ the strict transform of $l$ via $s_{P_{i}}$.
    Then we have
    \[
        \deg_{l^{\sim}_{1}} (\pi_{P_{1}}^{*} H_{1}) + \deg_{l^{\sim}_{2}} (\pi_{P_{2}}^{*} H_{2}) = d.
    \]
\end{proposition}
\begin{proof}
    By \cite[Proposition~5.8]{FL24}, we have
    \[
        \pi_{P_{1}}^{*}(H_{1}) + \mu^{*}\pi_{P_{2}}^{*}(H_{2}) \equiv_{\text{num}} \Ocal_{\PP(T^{\vee}_{G/P_{1}})}(d)
    \]
    where $\mu$ is the stratified Mukai flop $\PP(T^{\vee}_{G/P_{1}}) \dashrightarrow \PP(T^{\vee}_{G/P_{2}})$.
    In \cite[Proposition~5.8]{FL24}, there is a coefficient of $\pi_{P_{1}}^{*}(H_{1})$ written as $b(H)$, which turns out to be equal to 1 in \cite[\S5.D]{FL24}.
    Then the statement follows from Proposition~\ref{prop:Springer map is contacto}.
\end{proof}

\begin{remark}\label{rmk:rank of IHSS}
    In type $A$, $D$ and $E_{6,\,I}$, $G/P_{i}$ are irreducible Hermitian symmetric spaces, and their ranks as Riemannian symmetric spaces are equal to $d$ (cf. \cite[Table~V, \S X.6]{Helgason1979DifferentialGeometry}).
\end{remark}

\begin{table}
    \centering
    \begin{tabular}{c|c|c|c|c|c|p{2cm}}
        Type & $\gfr$ & $P_{1}$ & $P_{2}$ & $\overline{Z}$ & $d$ & Condition \\ \hline\hline
        $A$ & $A_{N-1}$ & \dynkin[labels={1, k, N-1}, edge length=.75cm] A{o.x.o} & \dynkin[labels={1, N-k, N-1}, edge length=.75cm] A{o.x.o} & $[2^{k},\,1^{N-2k}]$ & $k$ & $1 \le k < \frac{N}{2}$ \\ \hline
        $D$ & $D_{N/2}$ & \dynkin[labels={1,,N/2-1,N/2},
        edge length=.75cm]D{o.oxo} & \dynkin[labels={1,,N/2-1,N/2},
        edge length=.75cm]D{o.oox} & $[2^{N/2-1},\,1^{2}]$ & $\frac{N-2}{4}$ & $N \ge 10$ even, \newline $N/2$ odd \\ \hline
        $E_{6,\,I}$ & $E_{6}$ & \dynkin[labels={1, 2, 3, 4, 5, 6},
        edge length=.75cm]E{xooooo} & \dynkin[labels={1, 2, 3, 4, 5, 6},
        edge length=.75cm]E{ooooox} & $2A_{1}$ & 2 & - \\ \hline
        $E_{6,\,II}$ & $E_{6}$ & \dynkin[labels={1, 2, 3, 4, 5, 6},
        edge length=.75cm]E{ooxooo} & \dynkin[labels={1, 2, 3, 4, 5, 6},
        edge length=.75cm]E{ooooxo} & $A_{2}+2A_{1}$ & 4 & -
    \end{tabular}
    \caption{Stratified Mukai flops of type $A$, $D$, $E_{6,\,I}$ and $E_{6,\,II}$.}
    \label{tab:Stratified Mukai flops}
\end{table}

\begin{corollary} \label{coro: Mukai flop E6I}
    In the notation of Proposition~\ref{prop:deg in Mukai flop}, suppose that the diagram (\ref{eqn:Mukai flop}) is of type $E_{6,\,I}$, and let $z \in Z$ be a point.
    Then $F(\overline{Z},\,z)$ consists of three Legendrian components $X_{0}$, $X_{1}$ and $X_{2}$ such that for a line $l$ on $\overline{Z}$ through $z$, if $\deg_{l^{\sim}_{1}}(\pi_{P_{1}}^{*} H_{1}) = i$, then $[T_{l,\,z}] \in X_{i}$ ($i=0,\,1,\,2$).
\end{corollary}
\begin{proof}
    Let $l$ be a line on $\overline{Z}$ through $z$.
    Since $d = 2$, by Proposition~\ref{prop:deg in Mukai flop}, there are three possibilities:
    \begin{itemize}
        \item $\deg_{\pi_{P_{1}}}(l^{\sim}_{1}) = 0$.
        In this case, $l^{\sim}_{1}$ is contracted by $\pi_{P_{1}}$, and hence $[T_{l,\,z}] \in \PP(T_{\PP(\ufr_{P_{1}}),\,z})$.

        \item $\deg_{\pi_{P_{1}}}(l^{\sim}_{1}) = 2$.
        In this case, $\deg_{\pi_{P_{2}}}(l^{\sim}_{2}) = 0$ and so $l^{\sim}_{2}$ is contracted by $\pi_{P_{2}}$.
        Hence $[T_{l,\,z}] \in \PP(T_{\PP(\ufr_{P_{2}}),\,z})$.

        \item $\deg_{\pi_{P_{1}}}(l^{\sim}_{1}) = 1$.
        In this case, tangent directions of such $l$'s form an irreducible component by Corollary~\ref{coro:irreducible comp over VMRT of lines}.
        Indeed, the $E_{6}$-irreducible Hermitian symmetric space $E_{6}/P_{1}$ admits a unique family of ($H_{1}$-)lines.
        Since $\dim E_{6}/P_{1}=16$ and the Fano index of $E_{6}/P_{1}$ is $12$, lines are unbendable of type $(1^{10},\,0^{5})$.
    \end{itemize}
\end{proof}

In particular, in the $E_{6,\,I}$-case, we have $|\Leg(\overline{Z},\,z)| = d+1$.
In the next section, we shall see that the same holds true in the cases of type $A$ and $D$.

\section{Square-zero cases} \label{section:square zero}

In this section, we focus on simple Lie algebras of classical type.
Recall their standard representations:
\begin{align*}
    \mathfrak{sl}_{N} &= \{\text{traceless $N \times N$ matrices}\} \quad (N \ge 2), \\
    \mathfrak{so}_{N} & = \mathfrak{so}(\Omega_{1})\\
    &= \{\text{$N \times N$ matrices $A$ s.t. $A^{tr} \Omega_{1} + \Omega_{1}A = 0$}\} \quad (N \ge 3, \,\not=4), \\
    \mathfrak{sp}_{N} &= \mathfrak{sp}(\Omega_{-1}) \\
    &= \{\text{$N \times N$ matrices $A$ s.t. $A^{tr} \Omega_{-1} + \Omega_{-1}A = 0$}\}  \quad (N \text{ even, } \ge 2)
\end{align*}
where $\Omega_{1}$ and $\Omega_{-1}$ are nondegenerate symmetric and skew-symmetric 2-forms on $\CC^{N}$, respectively.
When $\gfr$ is one of those, nilpotent orbits in $\gfr$ are adjoint orbits of nilpotent matrices, and can be classified in terms of the Jordan type.
For a partition $\dbf = [d_{1},\, \cdots ,\, d_{k}]$ ($d_{1} \ge \cdots \ge d_{k} > 0$, $\sum_{i}d_{i}=N$) of $N$, it is customary to write $O_{\dbf}$ for a nilpotent orbit in $\gfr$ consisting of nilpotent matrices of Jordan type $[d_{1},\, \cdots ,\, d_{k}]$, i.e., those conjugate to the $N \times N$ Jordan matrix
\[
    \begin{pmatrix}
        \JJ_{d_{1}} & & \\
        & \ddots & \\
        && \JJ_{d_{k}}
    \end{pmatrix}
\]
where $\JJ_{d}$ is the $d \times d$ elementary Jordan matrix
\[
    \JJ_{d} \coloneqq \begin{pmatrix}
        0 &1 & & & \\
        & 0 & 1 & & \\
        &&\ddots&\ddots& \\
        && & 0 & 1 \\
        &&& & 0 
    \end{pmatrix} \quad (\forall d \ge 2), \quad \JJ_{1} = \begin{pmatrix}
        0
    \end{pmatrix}.
\]
We often write $\dbf = [d_{1},\, \cdots ,\, d_{k}] = [\delta_{1}^{n_{1}},\, \cdots ,\,\delta_{l}^{n_{l}}]$ for the integers $\delta_{1} > \cdots > \delta_{l} > 0$ and $n_{i} > 0$ $\forall i$ such that $k = \sum_{i}n_{i}$ and
\[
    \delta_{1} = d_{1} = \cdots = d_{n_{1}} > \delta_{2} = d_{n_{1}+1} =\cdots = d_{n_{1} + n_{2}} > \cdots > \delta_{l} = d_{n_{1} + \cdots + n_{l}} = \cdots = d_{k}.
\]
Such a nilpotent orbit $O_{\dbf}$ exists (only) for any $\dbf$ when $\gfr = \mathfrak{sl}_{N}$; $\dbf$ such that $n_{i}$ is even whenever $\delta_{i}$ is even when $\gfr = \mathfrak{so}_{N}$; $\dbf$ such that $n_{i}$ is even whenever $\delta_{i}$ is odd $\gfr = \mathfrak{sp}_{N}$.
$O_{\dbf}$ is unique and hence well defined, except when $\gfr = \mathfrak{so}_{N}$ and $\dbf$ is \emph{very even}, i.e., every $\delta_{i}$ is even: in this exception, there are two $O_{\dbf}$ in $\mathfrak{so}_{N} = \mathfrak{so}(\Omega_{1})$, denoted by $O_{\dbf}^{I}$ and $O_{\dbf}^{II}$, which are $\O(\Omega_{1})$-conjugate to each other.
In the case where it is unnecessary to distinguish $O_{\dbf}^{I}$ and $O_{\dbf}^{II}$, we simply denote by $O_{\dbf}$ one of them.
For details, we refer to \cite[Ch.~5]{CollingwoodMcGovern1993NilpotentOrbits}.

\begin{definition}
    Let $\gfr$ be a simple Lie algebra of classical type.
    For a partition $\dbf \not=[1^{N}]$, we denote by $Z_{\dbf}$ the nilpotent orbit $\PP(O_{\dbf})$ in $\PP(\gfr)$.
    In this notation, we say that $Z_{\dbf}$ and $O_{\dbf}$ are \emph{square-zero orbits} if $\dbf = [2^{r},\, 1^{N-2r}]$ for some $1 \le r \le \frac{N}{2}$, or equivalently, $O_{\dbf}$ consists of square-zero matrices with respect to the standard representation of $\gfr$.
\end{definition}

Observe that exactly one type of very even partition can define a square-zero orbit: $\dbf = [2^{N/2}]$ for $\mathfrak{so}_{N}$ with $N$ divisible by $4$.
In this case, our choice of $Z_{[2^{N/2}]}$ between $\PP(O^{I}_{[2^{N/2}]})$ and $\PP(O^{II}_{[2^{N/2}]})$ will be specified in Remark~\ref{rmk:choice when very even}.
This is enough for our purpose since $F(\overline{Z},\,z)$'s for $\PP(O^{I}_{[2^{N/2}]})$ and $\PP(O^{II}_{[2^{N/2}]})$ are isomorphic to each other.

\begin{remark}
\begin{enumerate}
    \item For a nilpotent orbit in $\mathfrak{so}(\Omega_{1})$ (resp. $\mathfrak{sp}(\Omega_{-1})$), the representative partition $\dbf$ does not depend on the choice of $\Omega_{1}$ (resp. $\Omega_{-1}$).
    Hence being square-zero is independent of the choice of $\Omega_{1}$ (resp. $\Omega_{-1}$).
    \item Being square-zero depends on the choice of representations of $\gfr$.
    For example, $\mathfrak{so}_{5}$ and $\mathfrak{sp}_{4}$ are isomorphic to each other as Lie algebras, while their subregular nilpotent orbits are $O_{[3,\,1^{2}]}$ and $O_{[2^{2}]}$, respectively (cf. \cite[Proposition~5.4.1]{CollingwoodMcGovern1993NilpotentOrbits}).
    Similarly, $O_{[2]}$ in $\mathfrak{sl}_{2}$ is $O_{[3]}$ in $\mathfrak{so}_{3}$, and $O_{[2^{2}]}$ in $\mathfrak{sl}_{4}$ is $O_{[3,\,1^{3}]}$ in $\mathfrak{so}_{6}$.
    \item Given a square-zero orbit $Z_{\dbf}$ in $\PP(\gfr)$, say $\dbf = [2^{r},\,1^{N-2r}]$, its closure $\overline{Z}_{\dbf}$ is naturally stratified as follows: 
    \begin{enumerate}
        \item If $\gfr = \mathfrak{sl}_{N}$ ($N \ge 2$) or $\mathfrak{sp}_{N}$ ($N$ even, $\ge 2$), then
        \[
        \overline{Z}_{\dbf} = \overline{Z}_{[2^{r},\,1^{N-2r}]} \supset \overline{Z}_{[2^{r-1},\,1^{N-2(r-1)}]} \supset \cdots \supset \overline{Z}_{[2^{2},\,1^{N-4}]}\supset \overline{Z}_{[2,\,1^{N-2}]}.
        \]
        \item If $\gfr = \mathfrak{so}_{N}$ ($N \ge 3,\,\not=4$), then $r$ is even and
        \[
        \overline{Z}_{\dbf} = \overline{Z}_{[2^{r},\,1^{N-2r}]} \supset \overline{Z}_{[2^{r-2},\,1^{N-2(r-2)}]} \supset \cdots\supset \overline{Z}_{[2^{4},\,1^{N-8}]} \supset \overline{Z}_{[2^{2},\,1^{N-4}]}.
        \]
    \end{enumerate}
    In any case, each stratum $Z_{[2^{k},\, 1^{N-2k}]}$ ($1 \le k \le \frac{N}{2}$) consists of square-zero matrices of rank $k$.
    The smallest strata are the adjoint varieties in $\PP(\gfr)$.
\end{enumerate}
\end{remark}

\begin{example} \label{ex:sq-zero Richardson}
    Using the classification of Richardson orbits in simple Lie algebras of classical type (see for instance \cite[Theorems~4.4 and 4.5]{Na06}), one can easily show that a square-zero orbit $Z_{\dbf}$ is Richardson if and only if one of the following holds:
    \begin{itemize}
        \item $\gfr = \mathfrak{sl}_{N}$ and $\dbf$ is any.
        \item $\gfr = \mathfrak{so}_{N}$, $N$ is even, and $\dbf = [2^{N/2}]$ or $[2^{N/2-1},\,1^{2}]$.
        \item $\gfr = \mathfrak{sp}_{N}$ and $\dbf = [2^{N/2}]$ or $[2^{2},\,1^{N-4}]$.
    \end{itemize}
\end{example}

\subsection{Idea of proof}\label{ssection:idea of proof}

In the following three sections, we shall compute $\Leg(\overline{Z}_{\dbf},\, z)$ for square-zero $Z_{\dbf}$.
Our strategy is to describe Legendrian components of $\locus(F(\overline{Z}_{\dbf},\,z))$, which is isomorphic to the cone over $F(\overline{Z}_{\dbf},\, z)$, based on the following observation:

\begin{proposition}\label{prop:locus of lines on square zero}
    Let $\gfr$ be a simple Lie algebra of classical type, and $\overline{Z}_{\dbf}$ a square-zero orbit closure in $\PP(\gfr)$ with the associated partition $\dbf = [2^{r},\,1^{N-r}]$ for $1 \le r \le \frac{N}{2}$.
    For $v \in O_{\dbf}$, we have
    \[
        \locus(F(\overline{Z}_{\dbf},\,[v])) = \left\{ [A] \in \PP(\ad_{\gfr}v) : A^{2} = 0, \text{ and } \rank(A+tv) \le r,\, \forall t\in\CC  \right\},
    \]
    provided that $F(\overline{Z}_{\dbf},\,[v])$ is not empty (cf. Corollary~\ref{coro:covered by lines except for adj type C}).
\end{proposition}
\begin{proof}
    Suppose that we have $A = \ad_{B}v$ for some $B \in \gfr$ such that $A^{2} = 0$ and $\rank(A+tv) \le r$ for all $t \in \CC$.
    Then we have
    \[
        vA + Av = v(Bv-vB) + (Bv-vB)v = 0,
    \]
    and thus $(A+tv)^{2} = 0$ for all $t \in \CC$.
    Since the matrix $A+tv$ is square-zero and of rank at most $r$, $[A+tv] \in \overline{Z}_{\dbf}$.
    Thus either $[A] = [v]$, or $[A] \not= [v]$ and the line joining $[A]$ and $[v]$ is contained in $\overline{Z}_{\dbf}$.

    Conversely, let $l$ be a line on $\overline{Z}_{\dbf}$ passing through $z$.
    Recall that the projective tangent space of $\overline{Z}_{\dbf}$ at $z$ is $\PP(\ad_{\gfr}v)$, and thus $l \subset \PP(\ad_{\gfr}v)$.
    Now for any $[A] \in l \subset \overline{Z}_{\dbf}$, $A$ is square-zero, and since $A+tv \in \overline{O}_{\dbf}$ for all $t \in \CC$, we have $\rank(A+tv) \le r$.
\end{proof}

For a square-zero matrix $v$ of rank $r$, define
\[
    \widehat{F}(v) \coloneqq \left\{ A \in \ad_{\gfr}v : A^{2} = 0, \text{ and } \rank(A+tv) \le r,\, \forall t\in\CC  \right\},
\]
which is nothing but the affine cone of the right hand side of the equality in Proposition~\ref{prop:locus of lines on square zero}.
Observe that $\Stab_{G}([v])$ acts on $\widehat{F}(v)$ via the adjoint action.

\begin{corollary} \label{coro: cone over cone over Fano var}
    Let $\gfr$ be a simple Lie algebra of classical type, and $\overline{Z}_{\dbf}$ a square-zero orbit closure in $\PP(\gfr)$ with the associated partition $\dbf = [2^{r},\,1^{N-r}]$ for $1 \le r \le \frac{N}{2}$.
    Choose $z = [v] \in Z_{\dbf}$ and write $2n+1 = \dim Z_{\dbf}$.
    Assume that $F(\overline{Z}_{\dbf},\,z)$ is nonempty, and identify $T_{Z_{\dbf},\,z}$ with $\ad_{\gfr}v/\CC v$ so that $F(\overline{Z}_{\dbf},\,z)$ is a subvariety of $\PP(\ad_{\gfr}v / \CC v)$.
    Then we have a $\Stab_{G}(z)$-equivariant $\CC \times \CC^{\times}$-bundle
    \[
        \widehat{F}(v) \setminus \CC v \rightarrow F(\overline{Z}_{\dbf},\,z), \quad A \mapsto [A \mod \CC v].
    \]
    Under this morphism, every $(n+1)$-dimensional component of $\widehat{F}(v)$ minus $\CC v$ is sent onto a Legendrian component of $F(\overline{Z}_{\dbf},\,z)$, and conversely, any Legendrian component is the image of an $(n+1)$-dimensional component minus $\CC v$.
\end{corollary}

\begin{proof}
    This is an immediate consequence of Proposition~\ref{prop:locus of lines on square zero} since Legendrian components are of dimension $n-1$.
\end{proof}

That is, finding Legenderian components of $F(\overline{Z}_{\dbf},\,[v])$ amounts to finding $(n+1)$-dimensional components of $\widehat{F}(v)$.
The latter is what is presented in the following three sections.

\subsection{\texorpdfstring{$\mathfrak{sl}_{N}$}{slN}}
\label{ssection::slN}

In this section, we consider square-zero orbit closures $\overline{Z}_{\dbf}$ in $\PP(\mathfrak{sl}_{N})$.
It is convenient to use $\GL_{N}$ instead of $\SL_{N}$, and geometrically this does not make any difference since the $\Ad_{\GL_{N}}$-action factors through the $\Ad_{\SL_{N}}$-action.
For $z \in Z_{\dbf}$, $F(\overline{Z}_{\dbf},\,z)$ is empty if and only if $N = 2$ (cf. Corollary~\ref{coro:covered by lines except for adj type C}), and so we assume that $N \ge 3$.

Let us introduce additional notations.
In what follows, we denote by $M_{k_{1} \times k_{2}}$ the vector space of $k_{1} \times k_{2}$ matrices.
Write $\dbf = [2^{r},\, 1^{N-r}]$ for $1 \le r \le \frac{N}{2}$.
Our base point $z = [v]$ is chosen by setting
\[
    v \coloneqq \begin{pmatrix}
        0_{r \times r} & 0_{r \times (N-2r)} & id_{r \times r} \\
        0_{(N-2r) \times r} & 0_{(N-2r) \times (N-2r)} &0_{(N - 2r) \times r} \\
        0_{r \times r} & 0_{r \times (N-2r)} & 0_{r \times r}
    \end{pmatrix}.
\]
Then $\Stab_{\GL_{N}}(z)$ consists of all matrices of the form
\begin{equation} \label{eqn:stab of base pt in SLn}
    \begin{pmatrix}
        g&*_{r \times (N-2r)}& *_{r \times r} \\
        0_{(N-2r) \times r}&h&*_{(N-2r) \times r} \\
        0_{r \times r}&0_{r \times (N-2r)}&\lambda \cdot g
    \end{pmatrix}, \quad g \in \GL_{r}, \quad h \in \GL_{N-2r}, \quad \lambda \in \CC^{\times},
\end{equation}
and $\ad_{\mathfrak{sl}_{N}}v (=\ad_{\mathfrak{gl}_{N}}v)$ consists of all matrices of the form
\begin{equation}\label{eqn:ad sl_N}
    \begin{pmatrix}
        A & B & C \\
        0_{(N - 2r) \times r} & 0_{(N - 2r) \times (N-2r)} & D \\
        0_{r \times r}&0_{r \times (N-2r)}& - A
    \end{pmatrix}, \quad A,\, C \in M_{r \times r}, \quad B \in M_{r \times (N-2r)}, \quad D \in M_{(N-2r) \times r}.
\end{equation}
Thus $\widehat{F}(v)$ consists of all matrices of the form (\ref{eqn:ad sl_N}) such that
\begin{equation}\label{eqn:F(v) when N not 2r in slN}
    A^{2} = 0, \ AB = 0, \ [A,\,C]+BD = 0, \ DA =0, \ \rank \begin{pmatrix}
        A & B & C + t\cdot id_{r} \\
        0& 0& D \\
        0&0&-A
    \end{pmatrix} \le r, \ \forall t \in \CC \quad (\text{if }N > 2r),
\end{equation}
and
\begin{equation}\label{eqn:F(v) when N = 2r in slN}
    A^{2} = 0, \quad \quad [A,\,C] = 0, \quad \rank \begin{pmatrix}
        A & C + t\cdot id_{r} \\
        0&-A
    \end{pmatrix} \le r, \quad \forall t \in \CC \quad (\text{if } N = 2r),
\end{equation}
and hence we have a morphism
\begin{equation}\label{eqn:projection from locus of lines in sl_N}
    \Psi:\widehat{F}(v) \rightarrow \overline{O}_{[2^{\lfloor r/2 \rfloor},\,1^{r-2\lfloor r/2 \rfloor}]} \subset \mathfrak{sl}_{r}, \quad \begin{pmatrix}
        A & B & C \\
        0& 0& D \\
        0&0&-A
    \end{pmatrix} \mapsto A.
\end{equation}
The morphism $\Psi$ is equivariant under the adjoint actions of $\Stab_{\GL_{N}}(z)$ and $\GL_{r}$.

\begin{theorem} \label{thm: square zero in slN}
    The morphism $\Psi$ in (\ref{eqn:projection from locus of lines in sl_N}) is surjective, and for $0 \le a \le \frac{r}{2}$ and the preimage $\Psi^{-1}(O_{[2^{a},\,1^{r-2a}]})$, we have
    \[
        \dim \Psi^{-1}(O_{[2^{a},\,1^{r-2a}]}) = r(N-r),
    \]
    and
    \[
        \#(\text{$r(N-r)$-dimensional components of $\Psi^{-1}(O_{[2^{a},\,1^{r-2a}]})$}) = \left\{\begin{array}{cc}
            1 & \text{if $N=2r$ or $r = 2a$,} \\
            2 & \text{otherwise.}
        \end{array}\right.
    \]
    Maximal dimensional components of $\widehat{F}(v)$ are precisely the closures of $r(N-r)$-dimensional components of $\Psi^{-1}(O_{[2^{a},\,1^{r-2a}]})$ for $0 \le a \le \frac{r}{2}$.
\end{theorem}

\begin{remark}
    $\overline{Z}_{\dbf}$ is the image of the Springer morphism defined by $\Gr(r,\,\CC^{N}) = \SL_{N}/P_{r}$.
    Thus by Proposition~\ref{prop: VMRT of Richardson is Leg} and Corollary~\ref{coro: cone over cone over Fano var}, $\widehat{F}(v)$ is of pure dimension $r(N-r)$.
    This fact is, however, not used in this section.
\end{remark}

\begin{corollary}\label{coro:Leg of square zero in slN}
    Assume that $N \ge 3$ and $1 \le r \le \frac{N}{2}$.
    Let $\dbf = [2^{r},\,1^{N-2r}]$ and $z\in Z_{\dbf} \subset \PP(\mathfrak{sl}_{N})$.
    Then
        \[
        |\Leg(\overline{Z}_{\dbf},\,z)| = \left\{\begin{array}{cc}
            1 + \lfloor r/2 \rfloor & \text{if $N=2r$,} \\
            1+r & \text{otherwise.}
        \end{array}\right.
        \]
        Furthermore, every element of $\Leg(\overline{Z}_{\dbf},\,z)$ is $\Stab_{\GL_{N}}(z)$-invariant.
\end{corollary}
\begin{proof}[Proof of Corollary~\ref{coro:Leg of square zero in slN}]
    The $\Stab_{\GL_{N}}(z)$-invariance follows from the connectedness of $\Stab_{\GL_{N}}(z)$ (see (\ref{eqn:stab of base pt in SLn})).
    By (\ref{eqn:ad sl_N}), $\dim \overline{Z}_{\dbf} = 2r(N-r)-1$, and thus elements of $\Leg(\overline{Z}_{\dbf},\,z)$ are $(r(N-r)-2)$-dimensional components of $F(\overline{Z}_{\dbf},\,z)$.
    By Corollary~\ref{coro: cone over cone over Fano var}, it suffices to show that the number of $r(N-r)$-dimensional components of $\widehat{F}(v)$ is $1 + \lfloor r/2 \rfloor$ if $N=2r$ and $1+r$ otherwise, which follows from Theorem~\ref{thm: square zero in slN}.
\end{proof}

The rest of this section is devoted to the proof of Theorem~\ref{thm: square zero in slN}.
Observe that the last statement of Theorem~\ref{thm: square zero in slN} follows from the previous part: since we have a finite stratification given by
        \[
            \overline{O}_{[2^{\lfloor r/2 \rfloor},\,1^{r-2\lfloor r/2 \rfloor}]} \supset \overline{O}_{[2^{\lfloor r/2 \rfloor-1},\,1^{r-2(\lfloor r/2 \rfloor-1)}]} \supset \cdots \supset \overline{O}_{[1^{r}]},
        \]
        any irreducible component of $\widehat{F}(v)$ is the closure of a component of $\Psi^{-1}(O_{[2^{a},\,1^{r-2a}]})$ for some $0 \le a \le \frac{r}{2}$.
        In particular, $\dim \widehat{F}(v) \le \dim \Psi^{-1}(O_{[2^{a},\,1^{r-2a}]})= r(N-r)$, and so conversely, the closure of an $r(N-r)$-dimensional component of $\Psi^{-1}(O_{[2^{a},\,1^{r-2a}]})$ is a component of $\widehat{F}(v)$.

To complete the proof of Theorem~\ref{thm: square zero in slN}, we determine the maximal dimensional components of $\Psi^{-1}(O_{[2^{a},\,1^{r-2a}]})$.
By the equivariance of $\Psi$, the restriction
\[
    \Psi_{a} \coloneqq \Psi|_{\Psi^{-1}(O_{[2^{a},\,1^{r-2a}]})} : \Psi^{-1}(O_{[2^{a},\,1^{r-2a}]}) \rightarrow O_{[2^{a},\,1^{r-2a}]} \quad (0 \le a \le \frac{r}{2})
\]
is a homogeneous fiber bundle with respect to the adjoint action of $\Stab_{\GL_{N}}(z)$.
To analyze its fibers, put
\[
    A_{a} \coloneqq \begin{pmatrix}
        0_{a \times (r-a)} & id_{a \times a} \\
        0_{(r-a) \times (r-a)} &0_{(r-a) \times a}
    \end{pmatrix} \in O_{[2^{a},\,1^{r-2a}]}.
\]
Then $\Psi^{-1}(A_{a})$ consists of all matrices satisfying (\ref{eqn:F(v) when N not 2r in slN}, \ref{eqn:F(v) when N = 2r in slN}) with $A=A_{a}$.
In fact, one can easily show that when $A = A_{a}$, the conditions (\ref{eqn:F(v) when N not 2r in slN}, \ref{eqn:F(v) when N = 2r in slN}) are simplified as follows:
\begin{itemize}
    \item If $N = 2r$, then
    \[
        A = A_{a}, \quad C = \begin{pmatrix}
           C_{1} & * & * \\
           0 & * & * \\
           0 & 0 & C_{1}
        \end{pmatrix}
    \]
    where $C_{1} \in M_{a \times a}$.
    (The rank condition in (\ref{eqn:F(v) when N = 2r in slN}) is always satisfied.)
    Thus
    \[
        \Psi^{-1}(A_{a}) \simeq \CC^{r^{2}-2ar+2a^{2}},
    \]
    and hence $\Psi^{-1}(O_{[2^{a},\,1^{r-2a}]})$ is irreducible and of dimension
    \[
        \dim \Psi^{-1}(A_{a}) + \dim O_{[2^{a},\,1^{r-2a}]} = (r^{2}-2ar+2a^{2}) + 2a(r-a) = r^{2}.
    \]
    This proves Theorem~\ref{thm: square zero in slN} in the case $N=2r$.

    \item If $N > 2r$ and $r = 2a$, then
    \[
    A = A_{a}, \quad B = \begin{pmatrix}
        B_{1} \\
        0
    \end{pmatrix}, \quad D = \begin{pmatrix}
        0 & D_{1} 
    \end{pmatrix}, \quad C = \begin{pmatrix}
       C_{1} &  * \\
       0 & C_{1}-B_{1}D_{1}
    \end{pmatrix}
    \]
    where $B_{1}\in M_{a \times (N-2r)}$, $D_{1} \in M_{(N-2r) \times a}$ and $C_{1} \in M_{a \times a}$.
    (Again the rank condition in (\ref{eqn:F(v) when N not 2r in slN}) is always satisfied.)
    Thus
    \[
        \Psi^{-1}(A_{a}) \simeq \CC^{2aN-3ar},
    \]
    and hence $\Psi^{-1}(O_{[2^{a},\,1^{r-2a}]})$ is irreducible and of dimension
    \[
        \dim \Psi^{-1}(A_{a}) + \dim O_{[2^{a},\,1^{r-2a}]} = (2aN-3ar) + 2a(r-a) = 2aN-ar-2a^2 = rN - r^{2}.
    \]
    This proves Theorem~\ref{thm: square zero in slN} in the case where $N>2r$ and $r = 2a$.

\item If $N > 2r$ and $r > 2a$, then
    \[
    A = A_{a}, \ B = \begin{pmatrix}
        B_{1} \\
        B_{2} \\
        0
    \end{pmatrix}, \ D = \begin{pmatrix}
        0 & D_{2} & D_{1}
    \end{pmatrix}, \ C = \begin{pmatrix}
       C_{1} & * & * \\
       B_{2}D_{1} & C_{2} & * \\
       0 & -B_{1}D_{2} & C_{1}-B_{1}D_{1}
    \end{pmatrix}, \ B_{2}D_{2} = 0
\]
and
\[
    \rank \begin{pmatrix}
        B_{2} & C_{2} + t \cdot id_{r-2a} \\
        0 & D_{2}
    \end{pmatrix} \le r-2a, \quad \forall t \in \CC
\]
where $B_{1}\in M_{a \times (N-2r)}$, $B_{2} \in M_{(r-2a) \times (N-2r)}$, $D_{2} \in M_{(N-2r) \times (r-2a)}$, $D_{1} \in M_{(N-2r) \times a}$, $C_{1} \in M_{a \times a}$ and $C_{2} \in M_{(r-2a) \times (r-2a)}$.
Since $B_{1}$, $D_{1}$ and part of $C$ given by
\[
    \begin{pmatrix}
        C_{1} & * & * \\
        & & * \\
        & & 
    \end{pmatrix}
\]
form the affine space of dimension $2aN - 2ar-2a^2$, we have
\[
    \Psi^{-1}(A_{a}) \simeq \CC^{2aN - 2ar-2a^2} \times \left\{(B_{2},\,D_{2},\,C_{2}): B_{2}D_{2} = 0, \, \rank \begin{pmatrix}
        B_{2} & C_{2} + t \cdot id_{r-2a} \\
        0 & D_{2}
    \end{pmatrix} \le r-2a, \, \forall t \in \CC \right\}.
\]
By the following Lemma~\ref{lemma: dimension in slN}, $\Psi^{-1}(A_{a})$ has exactly two maximal dimensional components, one defined by $B_{2} = 0$ and the other defined by $D_{2} = 0$, and both are of dimension
\[
    (2aN - 2ar - 2a^{2}) + (r-2a)(N-r-2a) = rN-r^{2} - 2ar+2a^{2}.
\]
It is easy to see that each of them is invariant under the action of the stabilizer of $A_{a}$ in $\Stab_{\GL_{N}}(z)$.
Hence $\Psi^{-1}(O_{[2^{a},\,1^{r-2a}]})$ is of dimension
\[
    \dim \Psi^{-1}(A_{a}) + \dim O_{[2^{a},\,1^{r-2a}]} = (rN-r^{2} - 2ar+2a^{2}) + 2a(r-a) = rN-r^{2}
\]
and contains precisely two maximal dimensional components.
This proves the remaining cases in Theorem~\ref{thm: square zero in slN}.
\end{itemize}

\begin{lemma}\label{lemma: dimension in slN}
    Let $a\ge0$, $b \ge 0$ and $c \ge 1$ be integers.
    Define the following two closed subsets in $M_{(a+c) \times (b+c)}$:
    \[
        M(a,\,b,\,c) \coloneqq \left\{ A \in M_{(a+c) \times (b \times c)} : \rank \left(
            A + \begin{pmatrix}
               0_{c \times b} & t\cdot id_{c} \\
               0_{a \times b} & 0_{a \times c}
            \end{pmatrix} \right) \le c, \, \forall t \in \CC \right\}
    \]
    and
    \[
        M'(a,\,c) \coloneqq \left\{ \begin{pmatrix}
            A_{1} & *_{c \times c} \\
            *_{a \times a} & A_{2}
        \end{pmatrix} \in M(a,\,a,\,c) : A_{1}A_{2} = 0 \right\}.
    \]
    Then the following hold:
    \begin{enumerate}
        \item\label{item1:lemma: dimension in slN} Every $A \in M(a,\,b,\,c)$ is of the form
        \[
            A = \begin{pmatrix}
                *_{c \times b} & *_{c \times c} \\
                0_{a \times b} & *_{a \times c}
            \end{pmatrix}.
        \]
        \item $\GL_{a} \times \GL_{b} \times \GL_{c}$ acts on $M(a,\,b,\,c)$ as follows:
        \[
            (g,\, h,\,k). \begin{pmatrix}
                A_{1} & A_{2} \\
                0 & A_{3}
            \end{pmatrix} \coloneqq \begin{pmatrix}
                k & 0  \\
                0 & g   
            \end{pmatrix} \begin{pmatrix}
                A_{1} & A_{2} \\
                0 & A_{3}
            \end{pmatrix} \begin{pmatrix}
                h & 0 \\
                0 & k
            \end{pmatrix}^{-1} = \begin{pmatrix}
                kA_{1}h^{-1} & kA_{2}k^{-1} \\
                0 & gA_{3}k^{-1}
            \end{pmatrix}.
        \]
        When $a=b$, this induces a $\diag(\GL_{a}) \times \GL_{c}$-action on $M'(a,\,c)$.
        \item We have
        \[
            \dim M(a,\,b,\,c) = (\max(a,\,b) + c)c, \quad \dim M'(a,\,c) = (a+c)c,
        \]
        and the only maximal dimensional components of $M'(a,\,c)$ are given by
        \[
            \left\{ \begin{pmatrix}
            *_{c \times a} & *_{c \times c} \\
            0_{a \times a} & 0_{a \times c}
        \end{pmatrix} \right\} \quad \text{and} \quad \left\{ \begin{pmatrix}
            0_{c \times a} & *_{c \times c} \\
            0_{a \times a} & *_{a \times c}
        \end{pmatrix} \right\}
        \]
        (which are same when $a=0$).
    \end{enumerate}
\end{lemma}
\begin{proof}
    \begin{enumerate}
        \item Suppose that $A \in M(a,\,b,\,c)$, and write
        \[
            A = \begin{pmatrix}
                A_{1} & A_{2} \\
                A_{3} & A_{4}
            \end{pmatrix}, \quad A_{1} \in M_{c \times b}, \quad A_{2} \in M_{c \times c}, \quad A_{3} \in M_{a \times b}, \quad A_{4} \in M_{a \times c}.
        \]
        Denote by $\rbf_{i}(A_{k})$ (resp. $\cbf(A_{k})$; $(A_{k})_{ij}$) the $i$th row (resp. the $j$th column; the entry at the $i$th row and the $j$th column) of $A_{k}$.
        Consider $(c + 1) \times (c+1)$ matrices
        \[
            \begin{pmatrix}
                \cbf_{j}(A_{1}) & A_{2} + t \cdot id_{c} \\
                (A_{3})_{ij} & \rbf_{i}(A_{4})
            \end{pmatrix}, \quad 1 \le i \le a,\quad 1\le j \le b, \quad t \in \CC.
        \]
        Since this is a minor of the matrix
        \[
            \begin{pmatrix}
                A_{1} & A_{2} + t \cdot id_{c} \\
                A_{3} & A_{4}
            \end{pmatrix},
        \]
        we have
        \[
            \det\begin{pmatrix}
                \cbf_{j}(A_{1}) & A_{2} + t \cdot id_{c} \\
                (A_{3})_{ij} & \rbf_{i}(A_{4})
            \end{pmatrix} = 0, \quad 1 \le i \le a,\quad 1\le j \le b, \quad t \in \CC.
        \]
        Since the left hand side is a polynomial in $t$ and the coefficient of the top degree term $t^{c}$ is $(A_{3})_{ij}$, we conclude that $(A_{3})_{ij} = 0$ for all $i$ and $j$, and hence $A_{3} = 0$.

        \item Omitted.

        \item Observe that
        \[
        \left\{ \begin{pmatrix}
            *_{c \times b} & *_{c \times c} \\
            0 & 0
        \end{pmatrix} \right\} \cup \left\{ \begin{pmatrix}
            0 & *_{c \times c} \\
            0 & *_{a \times c}
        \end{pmatrix} \right\} \subset M(a,\,b,\,c), \quad \left\{ \begin{pmatrix}
            *_{c \times a} & *_{c \times c} \\
            0 & 0
        \end{pmatrix} \right\} \cup \left\{ \begin{pmatrix}
            0 & *_{c \times c} \\
            0 & *_{a \times c}
        \end{pmatrix} \right\} \subset M'(a,\,c),
        \]
        and so $M(a,\,b,\,c)$ and $M'(a,\,c)$ are of dimension at least $(\max(a,\,b)+c)c$ and $(a+c)c$, respectively.

        First we show that $\dim M(a,\,b,\,c) = (\max(a,\,b)+c)c$, by induction on $b+c \ge 1$.
        The base case is trivial since $M(a,\,0,\,c) = M_{(a+c) \times c}$ whenever $b=0$.
        Assuming $b \ge 1$, define a morphism
        \[
            \Xi = \Xi_{a,\,b,\,c} : M(a,\,b,\,c) \rightarrow M_{c \times b}, \quad \begin{pmatrix}
                A & *\\
                0 &*
            \end{pmatrix} \mapsto A,
        \]
        and then $\Xi$ is surjective and equivariant under the action of $\GL_{a} \times \GL_{b} \times \GL_{c}$.
        Recall that $M_{c \times b}$ is stratified by
        \[
            (M_{c \times b})_{x} \coloneqq \{A \in M_{c \times b} : \rank A = x\}, \quad 0 \le x \le \min(b,\,c),
        \]
        and each stratum $(M_{c \times b})_{x}$ is $\GL_{b} \times \GL_{c}$-homogeneous.
        Choose a base point
        \[
            A_{x} \coloneqq \begin{pmatrix}
                0 & id_{x} \\
                0 & 0
            \end{pmatrix} \in (M_{c \times b})_{x},
        \]
        and then $\Xi^{-1}(A_{x})$ consists of matrices of the form
        \begin{equation}\label{eqn:lemma: dimension in slN}
            \begin{pmatrix}
                A_{x} & B \\
                0 & C
            \end{pmatrix}
        \end{equation}
        where
        \[
            B = \begin{pmatrix}
                * & * \\
                B_{1} & B_{2}
            \end{pmatrix},\quad B_{1} \in M_{(c-x) \times x}, \quad B_{2} \in M_{(c-x) \times (c-x)}
        \]
        and
        \[
            C = \begin{pmatrix}
                C_{1} & C_{2}
            \end{pmatrix}, \quad C_{1} \in M_{a \times x},\quad C_{2} \in M_{a \times (c-x)}
        \]
        such that
        \[
            \rank \begin{pmatrix}
                B_{1} & B_{2} + t \cdot id_{c-x} \\
                C_{1} & C_{2}
            \end{pmatrix} \le c-x.
        \]
        Therefore
        \[
            \Xi^{-1}(A_{x}) \simeq \left\{\begin{array}{cc}
                \CC^{xc} \times M(a,\,x,\,c-x) & \text{if $x < c$,} \\
                \CC^{c^{2}} & \text{if $x = c$.}
            \end{array} \right.
        \]
        By the induction, for any $0 \le x \le \min(b,\,c)$,
        \begin{align*}
            \dim \Xi^{-1}(A_{x}) &= xc + (\max(a,\,x) + c-x)(c-x) \\
            &\le xc + (\max(a,\,b) + c-x)(c-x) \\
            &= x^{2} - (\max(a,\,b) + c)x + (\max(a,\,b) + c)c
        \end{align*}
        and thus
        \begin{align*}
            \dim \Xi^{-1}((M_{c \times b})_{x}) & = \dim \Xi^{-1}(A_{x}) + \dim (M_{c \times b})_{x} \\
            &\le (x^{2} - (\max(a,\,b) + c)x + (\max(a,\,b) + c)c) + ((b+c)x - x^{2}) \\
            &= (b-\max(a,\,b))x + (\max(a,\,b) + c)c,
        \end{align*}
        which is at most $(\max(a,\,b) + c)c$.
        From the stratification
        \[
            M_{c \times b} \supset \cdots \supset \overline{(M_{c\times b})_{x}} \supset \cdots \overline{(M_{c\times b})_{0}} = \{(0)\},
        \]
        it follows that $\dim M(a,\,b,\,c) \le (\max(a,\,b) + c)c$, and so $\dim M(a,\,b,\,c) = (\max(a,\,b) + c)c$.

        It remains to show that the only maximal dimensional components of $M'(a,\,c)$ are
        \[
             \left\{ \begin{pmatrix}
            * & * \\
            0 & 0
        \end{pmatrix} \right\} \quad \text{and}\quad \left\{ \begin{pmatrix}
            0 & * \\
            0 & *
        \end{pmatrix} \right\}.
        \]
        The case $a=0$ is trivial, and we assume that $a \ge 1$.
        Define
        \[
            \xi =\xi_{a,\,c} \coloneqq \Xi_{a,\,a,\,c}|_{M'(a,\,c)} : M'(a,\,c) \rightarrow M_{c \times a},
        \]
        and then $\xi$ is surjective and equivariant under the action of $\diag(\GL_{a}) \times \GL_{c}$.
        Furthermore, for $0 \le x \le \min(a,\,c)$, $\xi^{-1}(A_{x})$ consists of the matrices of the form (\ref{eqn:lemma: dimension in slN}) such that
        \[
            C = \begin{pmatrix}
                *_{(a-x) \times c} \\
                0_{x \times c}
            \end{pmatrix}.
        \]
        Hence
        \[
            \xi^{-1}(A_{x}) \simeq \left\{ \begin{array}{cc}
                \CC^{xc} \times M(a-x,\,x,\,c-x) & \text{if $x < c$,} \\
                \CC^{c^{2}} & \text{if $x = c$,}
            \end{array}  \right.
        \]
        and by the previous result on $\dim M(a,\,b,\,c)$,
        \[
            \dim \xi^{-1}(A_{x}) = xc + (\max(a-x,\,x)+c-x)(c-x) = \left\{ \begin{array}{cc}
                2x^{2}-(a+2c)x+(a+c)c & \text{if $x \le a/2$,} \\
                c^{2} & \text{if $x \ge a/2$,}
            \end{array} \right.
        \]
        and
        \[
            \dim \xi^{-1}((M_{c \times a})_{x}) = \left\{ \begin{array}{cc}
                x^{2}-cx+(a+c)c  & \text{if $x \le a/2$,} \\
                -x^{2} + (a+c)x + c^{2} & \text{if $x \ge a/2$.}
            \end{array} \right.
        \]
        Now it is straightforward to show that when $x \le a/2$ (resp. $x \ge a/2$), $\dim \xi^{-1}((M_{c \times a})_{x})$ is maximal at $x = 0,\,c$ (resp. $x = \min(a,\,c)$).
        Therefore
        \[
            \dim \xi^{-1}((M_{c \times a})_{x}) \le (a+c)c
        \]
        with the equality if and only if $x = 0,\, \min(a,\,c)$.
        Again from the stratification of $M_{c \times a}$ by the matrix rank $x$, it follows that $\dim M'(a,\,c) = (a+c)c$, and its maximal dimensional components are the closures of those of $\xi^{-1}((M_{c \times a})_{x})$ for $x =0,\,\min(a,\,c)$.
        Finally, from the definition of $\xi$, one can easily see that
        \[
            \xi^{-1}((M_{c \times a})_{0}) = \left\{ \begin{pmatrix}
                0 & * \\
                0 & *
            \end{pmatrix} \right\},
        \]
        and
        \[
            \xi^{-1}((M_{c \times a})_{\min(a,\,c)}) = \left\{ \begin{pmatrix}
                A & * \\
                0 & 0
            \end{pmatrix} : \text{$A \in M_{c \times a}$ is of full rank} \right\},
        \]
        which completes the proof.
    \end{enumerate}
\end{proof}

\subsection{\texorpdfstring{$\mathfrak{so}_{N}$}{soN} and \texorpdfstring{$\mathfrak{sp}_{N}$}{spN}} \label{ssection::soN spN}

In this section, we consider square-zero orbit closures $\overline{Z}_{[2^{r},\,1^{N-2r}]}$ in $\PP(\mathfrak{so}_{N})$ and $\PP(\mathfrak{sp}_{N})$.
Recall that in the former, $r$ is assumed to be even, and in the latter, $N$ is assumed to be even.
The two cases are similar to each other and so presented together.
For this purpose, the following notation is used throughout this section:
\begin{assumption} \label{assumption:notation for soN and spN}
    Let $\epsilon \in \{-1,\,1\}$.
    We choose integers $N$, $r$ and $x$ as follows:
    \begin{itemize}
        \item When $\epsilon = 1$, then $N \ge 5$, $1 \le r \le \frac{N}{2}$ is even, and $0 \le x \le \frac{r}{2}$ is even.
        This notation is for $\overline{Z}_{[2^{r},\,1^{N-2r}]} \subset \PP(\mathfrak{so}_{N})$.
        (See also Remark~\ref{rmk:choice when very even} for the very even partition $[2^{N/2}]$.)

        \item When $\epsilon = -1$, then $N \ge 4$ is even, $1 \le r \le \frac{N}{2}$, and $0 \le x \le \frac{r}{2}$.
        This notation is for $\overline{Z}_{[2^{r},\,1^{N-2r}]} \subset \PP(\mathfrak{sp}_{N})$.
    \end{itemize}
\end{assumption}

Now we introduce our symmetric and symplectic forms.
First choose $\epsilon \in \{-1,\,1\}$.
For an integer $l \ge 1$, define antidiagonal $l \times l$ matrices $\Omega'_{\epsilon}$ as
\[
    \Omega'_{l,\,1} \coloneqq \begin{pmatrix}
        && & & 1 \\
        && &1 & \\
        && \iddots & & \\
        &1&&& \\
        1 & & & &
    \end{pmatrix},
\]
and
\[
    \Omega'_{l,\,-1} \coloneqq \begin{pmatrix}
        &  && & 1 \\
        &  && -1 & \\
        &  &\iddots & & \\
        & 1 & & & \\
        -1 & & & &
    \end{pmatrix} \quad \text{(for $l$ even)}.
\]
Under Assumption~\ref{assumption:notation for soN and spN}, define an antidiagonal $N \times N$ matrix
\[
    \Omega_{\epsilon} \coloneqq \begin{pmatrix}
        &&\Omega'_{r,\,-\epsilon}\\
        &\Omega'_{N-2r,\,\epsilon}&\\
        -\Omega'_{r,\,-\epsilon}&&
    \end{pmatrix},
\]
the associated matrix algebra
\[
    \sfr(\Omega_{\epsilon}) \coloneqq \{A \in M_{N \times N} : A^{tr} \Omega_{\epsilon} + \Omega_{\epsilon}A = 0\},
\]
its orthogonal complement
\[
    \sfr(\Omega_{\epsilon})^{\perp} \coloneqq \{A \in M_{N \times N} : A^{tr} \Omega_{\epsilon} - \Omega_{\epsilon}A = 0\},
\]
and the isometry groups
\[
    \II(\Omega_{\epsilon}) \coloneqq \{g \in \GL_{N} : g^{tr}\Omega_{\epsilon}g = \Omega_{\epsilon}\}, \quad \SI(\Omega_{\epsilon}) \coloneqq \{g \in \II(\Omega_{\epsilon}) : \det (g) = 1\}.
\]
When $\epsilon = 1$ (resp. $= -1$), $\Omega_{\epsilon}$ is symmetric (resp. skew-symmetric) and hence $\sfr(\Omega_{\epsilon}) = \mathfrak{so}_{N}$ (resp. $=\mathfrak{sp}_{N}$) and $\II(\Omega_{\epsilon}) = \O_{N}$ (resp. $=\Sp_{N} = \SI(\Omega_{-1})$).
Explicitly, $\sfr(\Omega_{\epsilon})$ consists of all matrices of the form
\[
    \begin{pmatrix}
        A_{1} \in M_{r \times r} & A_{2} \in M_{r \times (N-2r)} & A_{3}\in \sfr(\Omega'_{r,\,-\epsilon})^{\perp} \\
        A_{4} \in M_{(N-2r) \times r} & A_{5} \in \sfr(\Omega'_{N-2r,\,\epsilon}) & -(\Omega'_{N-2r,\,\epsilon})^{-1}A_{2}^{tr} \Omega'_{r,\,-\epsilon} \\
        A_{6}\in \sfr(\Omega'_{r,\,-\epsilon})^{\perp} & -(\Omega'_{r,\,-\epsilon})^{-1}A_{4}^{tr} \Omega'_{N-2r,\,\epsilon} & - (\Omega'_{r,\,-\epsilon})^{-1} A_{1}^{tr} \Omega'_{r,\,\epsilon}
    \end{pmatrix}.
\]
An advantage of the choice of $\Omega_{\epsilon}$ is that we can choose the same base points as in the case of $\mathfrak{sl}_{N}$:
\begin{equation} \label{eqn:base pt in soN and spN}
    v \coloneqq \begin{pmatrix}
        0 & 0 & id_{r} \\
        0 & 0 & 0 \\
        0 & 0 & 0
    \end{pmatrix} \in O_{[2^{r},\,1^{N-2r}]} \subset \sfr(\Omega_{\epsilon}),
\end{equation}
and
\begin{equation} \label{eqn:2nd base pt in soN and spN}
    A_{x} \coloneqq \begin{pmatrix}
        0_{x \times (r-x)} & id_{x} \\
        0_{(r-x) \times (r-x)} & 0_{(r-x) \times x}
    \end{pmatrix} \in \sfr(\Omega'_{r,\,-\epsilon})^{\perp}.
\end{equation}
\begin{remark} \label{rmk:choice when very even}
    Recall that when $\epsilon = 1$ and $N$ is divisible by $4$, we have two nilpotent orbits in $\sfr(\Omega_{1}) = \mathfrak{so}_{N}$ indexed by $[2^{N/2}]$: $O^{I}_{[2^{N/2}]}$ and $O^{II}_{[2^{N/2}]}$.
    In this section, we denote by $O_{[2^{N/2}]}$ the one containing our base point $v$ (\ref{eqn:base pt in soN and spN}).
    Nonetheless, since $O^{I}_{[2^{N/2}]}$ and $O^{II}_{[2^{N/2}]}$ are $\II(\Omega_{\epsilon})$-conjugate to each other, our main result is valid for both of them.
\end{remark}

Putting $z = [v] \in Z_{[2^{r},\,1^{N-2r}]}$, $\Stab_{\SI(\Omega_{\epsilon})}(z)$ consists of elements of $\SI(\Omega_{\epsilon})$ of the form (\ref{eqn:stab of base pt in SLn}).
Notice that for the conformal group
\[
    \CI(\Omega'_{r,\,-\epsilon}) \coloneqq \{g \in \GL_{r} : \exists \lambda \in \CC^{\times},\, \lambda g^{tr} \Omega'_{r,\,-\epsilon} g = \Omega'_{r,\,-\epsilon}\}
\]
and the morphism
\[
    \lambda_{\epsilon} : \CI(\Omega'_{r,\,-\epsilon}) \rightarrow \CC^{\times}, \quad (g \text{ s.t. } \lambda g^{tr} \Omega'_{r,\,-\epsilon} g = \Omega'_{r,\,-\epsilon}) \mapsto \lambda,
\]
if the matrix in (\ref{eqn:stab of base pt in SLn}) is an element of $\Stab_{\SI(\Omega_{\epsilon})}(z)$, then
\begin{equation} \label{eqn:stab of base pt in SON SpN}
    g \in \CI(\Omega'_{r,\,-\epsilon}), \quad h \in \SI(\Omega'_{N-2r,\,\epsilon}), \quad \lambda = \lambda_{\epsilon}(g).
\end{equation}
Since $\ad_{\sfr(\Omega_{\epsilon})}(v)$ consists of all matrices of the form
\begin{equation}\label{eqn:ad soN and spN}
    \begin{pmatrix}
        A & B & C \\
        0_{(N - 2r) \times r} & 0_{(N - 2r) \times (N-2r)} & - (\Omega'_{N-2r,\,\epsilon})^{-1}B^{tr} \Omega'_{r,\, -\epsilon} \\
        0_{r \times r}&0_{r \times (N-2r)}& - A
    \end{pmatrix}, \quad A,\, C \in \sfr(\Omega'_{r,\,-\epsilon})^{\perp}, \quad B \in M_{r \times (N-2r)},
\end{equation}
$\widehat{F}(v)$ consists of all matrices of the form (\ref{eqn:ad soN and spN}) such that
\begin{equation}\label{eqn:F(v) when N not 2r in soN spN}
    {\begin{array}{c}
        A^{2} = 0, \quad AB = 0, \quad [A,\,C]=B(\Omega'_{N-2r,\,\epsilon})^{-1}B^{tr} \Omega'_{r,\, -\epsilon}, \\
         \rank \begin{pmatrix}
        A & B & C + t\cdot id_{r} \\
        0& 0& -(\Omega'_{N-2r,\,\epsilon})^{-1}B^{tr} \Omega'_{r,\, -\epsilon} \\
        0&0&-A
    \end{pmatrix} \le r, \quad \forall t \in \CC
    \end{array}}
     \  \quad (\text{if }N > 2r),
\end{equation}
and
\begin{equation}\label{eqn:F(v) when N = 2r in soN spN}
    A^{2} = 0, \quad \quad [A,\,C] = 0, \quad \rank \begin{pmatrix}
        A & C + t\cdot id_{r} \\
        0&-A
    \end{pmatrix} \le r, \quad \forall t \in \CC \quad (\text{if } N = 2r).
\end{equation}
As before, we have a morphism
\begin{equation} \label{eqn:projection from locus of lines in soN spN}
    \Psi_{\epsilon} : \widehat{F}(v) \rightarrow \overline{O}^{\mathfrak{sl}}_{[2^{\lfloor r/2 \rfloor},\,1^{r-2\lfloor r/2 \rfloor}]} \cap \sfr(\Omega'_{r,\,-\epsilon})^{\perp}, \quad \begin{pmatrix}
        A & B & C \\
        0 & 0 & - (\Omega'_{N-2r,\,\epsilon})^{-1}B^{tr} \Omega'_{r,\, -\epsilon} \\
        0&0& - A
    \end{pmatrix} \mapsto A
\end{equation}
where $\overline{O}^{\mathfrak{sl}}_{\dbf}$ is the nilpotent orbit closure in $\mathfrak{sl}_{r}$ labeled by a partition $\dbf$.
The morphism $\Psi_{\epsilon}$ is equivariant under the adjoint actions of $\Stab_{\SI(\Omega_{\epsilon})}(z)$ and $\CI(\Omega'_{r,\,-\epsilon})$.
In fact, since $\CI(\Omega'_{r,\,-\epsilon})$ is generated by the homothety subgroup and $\II(\Omega'_{r,\,-\epsilon})$, the adjoint action of $\CI(\Omega'_{r,\,-\epsilon})$ factors through that of $\II(\Omega'_{r,\,-\epsilon})$.
Moreover, for $x$ as in Assumption~\ref{assumption:notation for soN and spN},
\[
    \Ocal_{x} \coloneqq O^{\mathfrak{sl}}_{[2^{x},\,1^{r-2x}]} \cap \sfr(\Omega'_{r,\,-\epsilon})^{\perp}
\]
is $\II(\Omega'_{r,\,-\epsilon})$-homogeneous by \cite[Theorem~20]{HM95Contragredient}.
Note that $O^{\mathfrak{sl}}_{[2^{y},\,1^{r-2y}]} \cap \sfr(\Omega'_{r,\,-\epsilon})^{\perp}$ is empty when $\epsilon = 1$ and $y$ is odd, and hence the orbital decomposition of the target of $\Psi_{\epsilon}$ is given by
\[
    \overline{O}^{\mathfrak{sl}}_{[2^{\lfloor r/2 \rfloor},\,1^{r-2\lfloor r/2 \rfloor}]} \cap \sfr(\Omega'_{r,\,-\epsilon})^{\perp} = \bigcup_{x \text{ as in Assumption~\ref{assumption:notation for soN and spN}}} \Ocal_{x}.
\]
Since $A_{x} \in \Ocal_{x}$ (cf. (\ref{eqn:2nd base pt in soN and spN})), by computing $\ad_{\sfr(\Omega'_{r,\,-\epsilon})} A_{x}$ and $\Stab_{\II(\Omega'_{r,\,-\epsilon})}(A_{x})$, one can easily show the following (details omitted):
\begin{proposition} \label{prop:components of Ox in soN spN}
    The $\II(\Omega'_{r,\,-\epsilon})$-orbit $\Ocal_{x}$ is of dimension $x(r-x)$.
    It is connected unless $\epsilon=-1$ and $r=2x$, in which case $\Ocal_{x}$ has two connected components.
\end{proposition}

Now by the same argument as in the case of $\mathfrak{sl}_{N}$, one can compute maximal dimensional components of $\widehat{F}(v)$.
Instead, to present a clearer exposition, we only record those related to Legendrian components of $F(\overline{Z}_{[2^{r},\,1^{N-2r}]},\,z)$.

\begin{theorem}\label{thm: square zero in soN spN}
    The morphism $\Psi_{\epsilon}$ in (\ref{eqn:projection from locus of lines in soN spN}) is surjective, and we have
    \[
        \dim (\Psi_{\epsilon})^{-1}(\Ocal_{x}) \le \frac{rN-r^{2}-\epsilon r}{2} =: d_{\epsilon},
    \]
    and the equality holds if and only if one of the following holds:
    \begin{enumerate}
        \item\label{item1:thm: square zero in soN spN} We have one of the following:
        \begin{enumerate}
            \item $\epsilon = -1$, $N=2r$ and $r > 2x$.
            \item $\epsilon = 1$ and $N = 2r$. \item $\epsilon = 1$ and $r=2x$.
            \item $\epsilon = 1$, $N \ge 2r+3$ and $r = 2x+2$.
        \end{enumerate}
        In this case, $(\Psi_{\epsilon})^{-1}(\Ocal_{x})$ is irreducible.
        
        \item\label{item2:thm: square zero in soN spN} $\epsilon = -1$ and $r = 2x$.
        In this case, $(\Psi_{-1})^{-1}(\Ocal_{x})$ consists of two $d_{-1}$-dimensional components, and they are identified by the $\Stab_{\SI(\Omega_{-1})}(z)$-action.
        \item\label{item3:thm: square zero in soN spN} $\epsilon = 1$, $N = 2r+2$ and $r > 2x$.
        In this case, $(\Psi_{1})^{-1}(\Ocal_{x})$ consists of two $d_{1}$-dimensional components, each of which is $\Stab_{\SI(\Omega_{1})}(z)$-invariant.
    \end{enumerate}
\end{theorem}

Since $\dim \overline{Z}_{[2^{r},\,1^{N-2r}]} = rN - r^{2} - \epsilon r - 1$ (cf. (\ref{eqn:ad soN and spN})), as in Corollary~\ref{coro:Leg of square zero in slN}, we deduce the following:

\begin{corollary} \label{coro:Leg of square zero in soN spN}
    Under Assumption~\ref{assumption:notation for soN and spN}, put $\dbf =[2^{r},\,1^{N-2r}]$.
    \begin{enumerate}
        \item If $\epsilon = 1$, then for $z \in Z_{\dbf} \subset \PP(\mathfrak{so}_{N})$, we have
        \[
            |\Leg(\overline{Z}_{\dbf},\,z)| = |\Leg(\overline{Z}_{\dbf},\,z)/\Stab_{\SO_{N}}(z)| = \left\{\begin{array}{cc}
                \lfloor \frac{r}{4}\rfloor + 1 & \text{if $N=2r$} \\
                 0 & \text{if $N=2r+1$ and $r/2$ is odd,} \\
                 \frac{r}{2} + 1& \text{if $N=2r+2$,} \\
                 1& \text{otherwise.}
            \end{array}\right.
        \]

        \item If $\epsilon = -1$, then for $z \in Z_{\dbf} \subset \PP(\mathfrak{sp}_{N})$, we have
        \[
            |\Leg(\overline{Z}_{\dbf},\,z)| = \left\{\begin{array}{cc}
                 \lfloor \frac{r}{2} \rfloor + 1 & \text{if $N=2r$ and $r$ is odd,} \\
                 \frac{r}{2} + 2 & \text{if $N=2r$ and $r$ is even,} \\
                 2 & \text{if $N\ge 2r+2$ and $r$ is even,} \\
                 0& \text{otherwise,}
            \end{array}\right.
        \]
        and
        \[
            |\Leg(\overline{Z}_{\dbf},\,z)/\Stab_{\Sp_{N}}(z)| = \left\{\begin{array}{cc}
                 \lfloor \frac{r}{2} \rfloor + 1 & \text{if $N=2r$,} \\
                 1 & \text{if $N\ge 2r+2$ and $r$ is even,} \\
                 0& \text{otherwise.}
            \end{array}\right.
        \]
    \end{enumerate}
\end{corollary}

In the rest of this section, we sketch the proof of Theorem~\ref{thm: square zero in soN spN}.
In fact, the details are similar with those of the previous section, and so we only explain the main differences.
Thanks to Proposition~\ref{prop:components of Ox in soN spN}, Theorem~\ref{thm: square zero in soN spN} is reduced to show that:
\begin{itemize}
    \item We have
    \begin{equation} \label{eqn:expected fib dim in soN spN}
        \dim (\Psi_{\epsilon})^{-1}(A_{x}) \le d_{\epsilon} - x(r-x) = \frac{rN-r^{2}-\epsilon r - 2xr + 2x^{2}}{2}
    \end{equation}
    with the equality exactly in the cases (\ref{item1:thm: square zero in soN spN}-\ref{item3:thm: square zero in soN spN}) in Theorem~\ref{thm: square zero in soN spN}, and
    \item In the cases (\ref{item1:thm: square zero in soN spN}-\ref{item3:thm: square zero in soN spN}), $(\Psi_{\epsilon})^{-1}(A_{x})$ is of pure dimension, and the number of irreducible components is 1, 1, and 2, respectively.
    In the last case, each irreducible component is invariant under the action of the stabilizer of $A_{x}$ in $\Stab_{\SI(\Omega_{1})}(z)$ (cf. (\ref{eqn:stab of base pt in SON SpN})).
\end{itemize}

As in the case of $\mathfrak{sl}_{N}$, if $N=2r$ or $r=2x$, then the conditions (\ref{eqn:F(v) when N not 2r in soN spN}, \ref{eqn:F(v) when N = 2r in soN spN}) are simpler and $(\Psi_{\epsilon})^{-1}(A_{x})$ is isomorphic to the affine space of the desired dimension (\ref{eqn:expected fib dim in soN spN}).
Even in the case where $N > 2r$ and $r > 2x$, the argument of the previous section still works, showing that
\[
    (\Psi_{\epsilon})^{-1}(A_{x}) \simeq \CC^{x(N-r-x-\epsilon)} \times M'_{\epsilon}(N-2r,\,r-2x)
\]
in the notation of the following Lemma~\ref{lemma: dimension in soN spN}.
Now since
\[
    \frac{rN-r^{2}-\epsilon r - 2xr + 2x^{2}}{2} - x(N-r-x-\epsilon) = \frac{1}{2}(r-2x)(N-r-2x-\epsilon),
\]
we complete the proof of Theorem~\ref{thm: square zero in soN spN} by (\ref{eqn:expected fib dim in soN spN}) and the following Lemma~\ref{lemma: dimension in soN spN}.

\begin{lemma} \label{lemma: dimension in soN spN}
    Let $b \ge 0$ and $c \ge 1$ be integers such that $b$ is even whenever $\epsilon=-1$, and $c$ is even whenever $\epsilon = 1$.
    Define
    \[
        M_{\epsilon}(b,\,c) \coloneqq \left\{ (B,\,C) \in M_{c \times b} \times \sfr(\Omega'_{c,\,-\epsilon})^{\perp} : \rank \begin{pmatrix}
        B & C + t \cdot id_{c} \\
        0 & - (\Omega'_{b,\,\epsilon})^{-1}B^{tr} \Omega'_{c,\, -\epsilon}
    \end{pmatrix} \le c, \quad \forall t \in \CC \right\},
    \]
    and
    \[
        M'_{\epsilon}(b,\,c) \coloneqq \left\{(B,\,C) \in M_{\epsilon}(b,\,c) : B (\Omega'_{b,\,\epsilon})^{-1}B^{tr} = 0\right\}.
    \]
    Consider the $\CI(\Omega'_{c,\,-\epsilon}) \times \SI(\Omega'_{b,\,\epsilon})$-action defined by
    \[
        (g,\,h). (B,\,C) = (gBh^{tr},\,gCg^{-1}).
    \]
    Then 
        \begin{equation} \label{eqn:dim M (b c) in soN spN}
            \dim M_{\epsilon}(b,\,c) = \left\{
            \begin{array}{cc}
                 \frac{c(c-\epsilon)}{2}& \text{if $b=0$,} \\
                 \frac{c^{2}+bc-b+1}{2}& \text{if $b \ge 1$ and $c$ is odd,} \\
                 \frac{c^{2}+bc}{2}& \text{if $b \ge 1$ and $c$ is even,}
            \end{array}\right.
        \end{equation}
        and
        \[
            \dim M'_{\epsilon}(b,\,c) \le \frac{c(c+b-\epsilon)}{2}
        \]
        with the equality if and only if one of the following holds:
        \begin{enumerate}
            \item $b = 0$.
            In this case, $M'_{\epsilon}(0,\,c) = M_{\epsilon}(0,\,c) = \sfr(\Omega'_{c,\,-\epsilon})^{\perp}$.

            \item $\epsilon = 1$, $b \ge 3$ and $c=2$.
            In this case, $M'_{1}(b,\,2)$ is irreducible.
            
            \item \label{item3:lemma: dimension in soN spN} $\epsilon = 1$, $b = 2$ and $c \ge 2$.
            In this case, $M'_{1}(2,\,c)$ consists of all pairs $(B,\,C) \in M_{c \times 2} \times \sfr(\Omega'_{c,\,-1})^{\perp}$ such that $B$ is of the form
            \[
                \begin{pmatrix}
                    0_{c \times 1} & *_{c \times 1}
                \end{pmatrix} \quad \text{or} \quad \begin{pmatrix}
                    *_{c \times 1} & 0_{c \times 1}
                \end{pmatrix}.
            \]
            Moreover, each of the two irreducible components of $M'_{\epsilon}(2,\,c)$ is invariant under the $\CI(\Omega'_{c,\,-1}) \times \SI(\Omega'_{2,\,1})$-action.
        \end{enumerate}
\end{lemma}

\begin{remark}\label{rmk: action in lemma: dimension in soN spN}
    Of course $M_{\epsilon}(b,\,c)$ and $M'_{\epsilon}(b,\,c)$ are equipped with the $\CI(\Omega'_{c,\,-\epsilon}) \times \CI(\Omega'_{b,\,\epsilon})$-actions defined by the same formula $(g,\,h).(B,\,C) = (gBh^{tr},\, gCg^{-1})$.
    To utilize Lemma~\ref{lemma: dimension in soN spN} in the proof of Theorem~\ref{thm: square zero in soN spN}, however, one need to consider the $\CI(\Omega'_{c,\,-\epsilon}) \times \SI(\Omega'_{b,\,\epsilon})$-action, due to (\ref{eqn:stab of base pt in SON SpN}).
\end{remark}

\begin{proof}[Proof of Lemma~\ref{lemma: dimension in soN spN}]
    First we compute $\dim M_{\epsilon}(b,\,c)$.
    If $b = 0$, then $M_{\epsilon}(0,\,c) = \sfr(\Omega'_{c,\,-\epsilon})^{\perp}$ and thus it is of dimension $c(c-\epsilon)/2$.
    If $c = 1$ (and so $\epsilon = -1$), then for any $(B,\,C) \in M_{-1}(b,\,1)$, $B = 0$ and so $M_{-1}(b,\,1) \simeq \CC$.
    \begin{claim}\label{claim:lemma: dimension in soN spN}
        If $b\ge 1$, then for any $(B,\,C) \in M_{\epsilon}(b,\,c)$, every column of $B$ is $\Omega'_{c,\,-\epsilon}$-isotropic.
    \end{claim}
    \begin{proof}[Proof of Claim~\ref{claim:lemma: dimension in soN spN}]
        We may assume that $\epsilon = -1$ and $c \ge 2$.
        Suppose that a column of $B$, say the $j$th column $B_{j}$ ($1 \le j \le b$), is not $\Omega'_{c,\,1}$-isotropic.
        Up to the $\CI(\Omega'_{c,\,1}) \times \SI(\Omega'_{b,\,-1})$-action, we may assume that $(B_{j})^{tr} = (1 \ 0 \ \cdots \ 0 \ 1)$.
        Then the $(b-j+1)$th row of
        \[
            - (\Omega'_{b,\,-1})^{-1}B^{tr} \Omega'_{c,\,1} = \begin{pmatrix}
                & & & & 1 \\
                & & & -1 & \\
                & & \iddots & & \\
                & 1 & & & \\
                -1& & & & \\
            \end{pmatrix} B^{tr} \begin{pmatrix}
                &&1 \\
                &\iddots& \\
                1&&
            \end{pmatrix}
        \]
        is $((-1)^{j} \ 0 \ \cdots \ 0 \ (-1)^{j})$.
        By the definition of $M_{\epsilon}(b,\,c)$, we have
        \[
            \det \begin{pmatrix}
                1 & \\
                0 & \\
                \vdots & & & C+t \cdot id_{c} \\
                0 & \\
                1 &\\
                0 & 1 & 0 & \cdots & 0 & 1
            \end{pmatrix} =0, \quad \forall t \in \CC,
        \]
        which is a contradiction.
    \end{proof}
    Thanks to Claim~\ref{claim:lemma: dimension in soN spN}, whenever $b\ge1$, we have a surjective morphism
    \[
        \Xi_{\epsilon}:M_{\epsilon}(b,\,c) \rightarrow V(\Omega'_{c,\,-\epsilon}) \subset \CC^{c}, \quad (B,\,C) \mapsto (\text{the 1st column of $B$}).
    \]
    where
    \[
        V(\Omega'_{c,\,-\epsilon}) \coloneqq \{\text{$\Omega'_{c,\,-\epsilon}$-isotropic elements in $\CC^{c}$}\} = \left\{
        \begin{array}{cc}
            \CC^{c} & \text{if $\epsilon = 1$,} \\
            \text{affine hyperquadric} & \text{if $\epsilon = -1$.}
        \end{array}\right.
    \]
    $\Xi_{\epsilon}$ is equivariant under the actions of $\CI(\Omega'_{c,\,-\epsilon}) \times \SI(\Omega'_{b,\,\epsilon})$ and $\CI(\Omega'_{c,\,-\epsilon})$.
    Moreover, $V(\Omega'_{c,\,-\epsilon})$ has an orbital decomposition
    \[
        V(\Omega'_{c,\,-\epsilon}) =\{0\} \cup \CI(\Omega_{c,\,-\epsilon}).\begin{pmatrix}
            1 \\
            0 \\
            \vdots \\
            0
        \end{pmatrix}.
    \]
    It is not difficult to show that
    \[
        (\Xi_{\epsilon})^{-1}(0) \simeq M_{\epsilon}(b-1,\,c).
    \]
    Using Lemma~\ref{lemma: dimension in slN}(\ref{item1:lemma: dimension in slN}), one can also show that
    \[
        (\Xi_{\epsilon})^{-1}\begin{pmatrix}
            1 \\
            0 \\
            \vdots \\
            0
        \end{pmatrix} \simeq \left\{ \begin{array}{cc}
            \CC^{b+1 - (\epsilon + 1)/2} & \text{if $c=2$,} \\
            \CC^{b+c - 1 - (\epsilon + 1)/2} \times M_{\epsilon}(b,\,c-2) & \text{if $c\ge 3$.}
        \end{array} \right.
    \]
    Since
    \[
        \dim M_{\epsilon}(b,\,c) = \max\left(\dim (\Xi_{\epsilon})^{-1}(0),\, (\Xi_{\epsilon})^{-1}(V(\Omega'_{c,\,-\epsilon}) \setminus \{0\})\right)
    \]
    and $\dim V(\Omega'_{c,\,-\epsilon}) = c + \frac{\epsilon-1}{2}$, we see that
    \[
        \dim M_{\epsilon}(b,\,c) = \left\{ \begin{array}{cc}
            \max(\dim M_{\epsilon}(b-1,\,2),\, b+2) & \text{if $b\ge1$ and $c=2$,} \\
            \max\left(\dim M_{\epsilon}(b-1,\,c),\, b+2c-2 + \dim M_{\epsilon}(b,\,c-2)\right) & \text{if $b\ge1$ and $c\ge 3$.}
        \end{array} \right.
    \]
    Recalling that $\dim M_{\epsilon}(0,\,c) = c(c-\epsilon)/2$ ($\forall c \ge 1$), an easy induction on $b \ge 1$ shows that
    \[
        \dim M_{\epsilon}(b,\,2) = \left\{
        \begin{array}{cc}
            2-\epsilon & \text{if $b=0$,} \\
            b+2 & \text{if $b\ge1$.}
        \end{array}
        \right.
    \]
    Similarly, recall that $\dim M_{-1}(b,\,1) = 1$ ($\forall b \ge 0$), and then the dimension formula (\ref{eqn:dim M (b c) in soN spN}) for $M_{\epsilon}(b,\,c)$ in the statement follows from an induction on $b+c \ge 1$ (details omitted).

    Next, consider $M'_{\epsilon}(b,\,c)$.
    If $\epsilon = -1$, then by (\ref{eqn:dim M (b c) in soN spN}), we have
    \[
        \dim M'_{-1}(b,\,c) \le \dim M_{-1}(b,\,c) \le \frac{c^{2}+bc + c}{2},
    \]
    and the second equality holds exactly when $b = 0$.
    Hence we may assume that $\epsilon = 1$ (and so $c$ is even).
    Then we need to show that
    \[
        M'_{1}(b,\,c) \le \frac{c(c+b-1)}{2}
    \]
    with the equality if and only if (1) $b= 0$ (2) $b \ge 3$ and $c=2$ or (3) $b=2$ and $c \ge 2$.
    Again it is enough to consider $b \ge 1$.
    \begin{itemize}
        \item If $b=1$, then for a $c \times 1$ matrix $B$, we have
    \begin{equation}\label{eqn:lemma: dimension in soN spN}
        0=B (\Omega'_{1,\,1})^{-1}B^{tr} = \begin{pmatrix}
            B_{11} \\
            \vdots \\
            B_{c1}
        \end{pmatrix}
        \begin{pmatrix}
            B_{11} & \cdots & B_{c1}
        \end{pmatrix} \quad \Leftrightarrow \quad B=0,
    \end{equation}
    and hence for any $(B,\,C) \in M_{1}(1,\,c)$, we have $B = 0$ and so
    \[
        M_{1}(1,\,c) \simeq \sfr(\Omega'_{c,\,-1})^{\perp}, \quad \dim M_{1}(1,\,c) = \frac{c(c-1)}{2} < \frac{c^{2}}{2}.
    \]

    \item If $b = 2$, then for a $c \times 2$ matrix $B$, we have
    \begin{equation}\label{eqn2:lemma: dimension in soN spN}
        0=B (\Omega'_{2,\,1})^{-1}B^{tr} = \begin{pmatrix}
            B_{11} & B_{12} \\
            \vdots & \vdots \\
            B_{c1} & B_{c2}
        \end{pmatrix}
        \begin{pmatrix}
            0 & 1 \\
            1 & 0
        \end{pmatrix}
        \begin{pmatrix}
            B_{11} & \cdots & B_{c1} \\
            B_{12} & \cdots & B_{c2}
        \end{pmatrix} \ \Leftrightarrow \ B_{i1} = 0 \ \forall i, \ \text{or} \ B_{i2} = 0 \ \forall i.
    \end{equation}
    For such a $B$, $\{B\} \times \sfr(\Omega'_{c,\,-1})^{\perp} \subset M_{1}(2,\,c)$ for any $c \ge 2$.
    Indeed, for any $C \in \sfr(\Omega'_{c,\,-1})^{\perp}$ and $t \in \CC$, we have a square-zero matrix
    \[
        \begin{pmatrix}
           0 & B & C + t \cdot id_{c} \\
           0 & 0 & - (\Omega'_{2,\,1})^{-1}B^{tr}\Omega'_{c,\,-1} \\
           0 & 0 & 0
        \end{pmatrix} \in \sfr\begin{pmatrix}
            && \Omega'_{c,\,-1} \\
            & \Omega'_{2,\,1}& \\
            -\Omega'_{c,\,-1}&&
        \end{pmatrix} = \mathfrak{so}_{2c+2} \quad (\text{cf. (\ref{eqn:base pt in soN and spN})}),
    \]
    and this is of rank at most $c$, since the largest square-zero nilpotent orbit closure in $\mathfrak{so}_{2c+2}$ is $\overline{O}_{[2^{c},\,1^{2}]}$ (as $c$ is even).
    Therefore each form of $B$ constitutes an irreducible component of $M'_{1}(2,\,c)$ that is of dimension $c + \dim \sfr(\Omega_{c,\,-1})^{\perp} = \frac{c(c+1)}{2}$.
    Note that the two irreducible components are $\CI(\Omega'_{c,\,-1}) \times \SI(\Omega'_{2,\,1})$-invariant since the group is connected.

    \item Assume that $b\ge 3$.
    Observe that for $(B,\,C) \in M_{1}(b,\,c)$, $\rank B \le c/2$.
    In particular, when $c = 2$, $B$ is of rank $\le 1$.
    Conversely, since
    \[
        \sfr(\Omega'_{2,\,-1})^{\perp} = \left\{\begin{pmatrix}
        a & 0 \\
        0 & a
    \end{pmatrix} : a \in \CC\right\},
    \]
    it is easy to show that for any $2 \times b$ matrix $B$ of rank $\le 1$, we have $\{B\} \times \sfr(\Omega'_{2,\,-1})^{\perp} \subset M_{1}(b,\,2)$.
    Moreover, if we denote by $U_{i}$ ($i=1,\,2$) the open subset of $M'_{1}(b,\,2)$ consisting of $(B,\,C) \in M'_{1}(b,\,2)$ such that the $i$th row of $B$ is nonzero, then $M'_{1}(b,\,2) = U_{1} \cup U_{2}$ and
    \[
        U_{i} \simeq \CC \times V(\Omega'_{b,\,1}) \times \sfr(\Omega'_{2,\,-1})^{\perp} \quad (\forall i =1,\,2), \quad U_{1} \cap U_{2} \simeq \CC^{\times} \times V(\Omega'_{b,\,1}) \times \sfr(\Omega'_{2,\,-1})^{\perp},
    \]
    which are irreducible (as $b \ge 3$) and of dimension $b+1$.
    It follows that $M'_{1}(b,\,2)$ is irreducible and of dimension $b+1$.

    It remains to show that $\dim M'_{1}(b,\,c) < c(c+b-1)/2$ for $b \ge 3$ and $c \ge 4$ (with $c$ even).
    In this case, consider the morphism
    \[
        \xi = \xi_{b,\,c} : M'_{1}(b,\,c) \rightarrow M_{c \times b}, \quad (B,\,C) \mapsto B.
    \]
    The image of $\xi$ has nice representatives, as shown in the following Claim~\ref{claim2:lemma: dimension in soN spN}.
    Here, observe that $M'_{1}(b,\,c)$ is equipped with the $\CI(\Omega'_{c,\,-1}) \times \II(\Omega'_{b,\,1})$-action defined by the same formula $(g,\,h).(B,\,C) = (gBh^{tr},\, gCg^{-1})$ (cf. Remark~\ref{rmk: action in lemma: dimension in soN spN}).
    \begin{claim}\label{claim2:lemma: dimension in soN spN}
        Let $b \ge 1$ be an integer, and $c \ge 2$ an even integer.
        Then for any $(B,\,C) \in M'_{1}(b,\,c)$, there exists $(g,\,h) \in \CI(\Omega'_{c,\,-1}) \times \II(\Omega'_{b,\,1})$ such that $gBh^{tr}$ is one of
        \[
            \begin{pmatrix}
                id_{a} & 0_{a \times (b-a)} \\
                0_{(c-a) \times a} & 0_{(c-a) \times (b-a)}
            \end{pmatrix}, \quad 0 \le a \le \min(b/2, \, c/2).
        \]
    \end{claim}
    \begin{proof}[Proof of Claim~\ref{claim2:lemma: dimension in soN spN}]
        If $b = 1,\, 2$, then the statement follows from (\ref{eqn:lemma: dimension in soN spN}) and (\ref{eqn2:lemma: dimension in soN spN}), respectively.
        Assume that $b \ge 3$.
        Since the rows of $B$ are elements of $V(\Omega'_{b,\,1})$ and $V(\Omega'_{b,\,1}) \setminus \{0\}$ is $\II(\Omega_{b,\,1})$-homogeneous, we have either $B = 0$ or, up to the $\CI(\Omega'_{c,\,-1}) \times \II(\Omega'_{b,\,1})$-action,
        \[
            B = \begin{pmatrix}
                1 & 0 \\
                0 & B' \in M_{(c-1) \times (b-1)}
            \end{pmatrix}.
        \]
        Since $B \Omega'_{b,\,1}B^{tr} = B (\Omega'_{b,\,1})^{-1}B^{tr} = 0$, the rows of $B$ span an $\Omega'_{b,\,1}-$isotropic subspace, and thus the last column of $B'$ is zero.
        Moreover, by Lemma~\ref{lemma: dimension in slN}(\ref{item1:lemma: dimension in slN}), the last row of $B'$ is zero.
        That is, we have
        \[
            B = \begin{pmatrix}
                1 & 0 & 0 \\
                0 & B'' \in M_{(c-2) \times (b-2)} &0 \\
                0 & 0 &0
            \end{pmatrix}
        \]
        and $(B'',\,C'') \in M'_{1}(b-2,\,c-2)$ for a minor $C''$ of $C$.
        Now the statement follows from an inductive argument.
    \end{proof}
    Thanks to Claim~\ref{claim2:lemma: dimension in soN spN}, we have
    \[
        \dim M'_{1}(b,\,c) = \max_{0 \le a \le \min(b/2,\,c/2)} \dim \xi^{-1}\left(\CI(\Omega'_{c,\,-1}) \times \SI(\Omega'_{b,\,1}) . \begin{pmatrix}
                id_{a} & 0 \\
                0 & 0
            \end{pmatrix} \right).
    \]
    Computing the stabilizer, one can easily show that
    \begin{align*}
        &\dim \left(\CI(\Omega'_{c,\,-1}) \times \SI(\Omega'_{b,\,1}) . \begin{pmatrix}
                id_{a} & 0 \\
                0 & 0
            \end{pmatrix}\right)\\
            &= (\frac{c^{2}+c}{2}+1) + \frac{b(b-1)}{2} - (1 + a(c-a) + \frac{(c-2a)(c-2a+1)}{2} + a(b-2a) + \frac{(b-2a)(b-2a-1)}{2} + a^{2}) \\
            &= -2a^{2} + (b+c)a.
    \end{align*}
    Since
    \[
        \xi^{-1}\begin{pmatrix}
                id_{a} & 0 \\
                0 & 0
            \end{pmatrix} \simeq \CC^{a(c-a) + a(a-1)/2} \times M_{1}(a,\,c-2a),
    \]
    and
    \[
        \dim M_{1}(a,\,c-2a) = \left\{ \begin{array}{cc}
                \frac{c(c-1)}{2} & \text{if $a=0$,} \\
                \frac{(c-2a)(c-a)}{2} & \text{if $a > 0$}
            \end{array} \right.
    \]
    by (\ref{eqn:dim M (b c) in soN spN}),
    we have
    \[
        \dim \xi^{-1}\begin{pmatrix}
                id_{a} & 0 \\
                0 & 0
            \end{pmatrix} = \left\{ \begin{array}{cc}
                \frac{c(c-1)}{2} & \text{if $a=0$,} \\
                \frac{a^{2}}{2} - \frac{c+1}{2}a + \frac{c^{2}}{2} & \text{if $a > 0$,}
            \end{array} \right.
    \]
    and thus
    \begin{equation} \label{eqn3:lemma: dimension in soN spN}
        \dim \xi^{-1} \left(\CI(\Omega'_{c,\,-1}) \times \SI(\Omega'_{b,\,1}) . \begin{pmatrix}
                id_{a} & 0 \\
                0 & 0
            \end{pmatrix}\right) = \left\{ \begin{array}{cc}
                \frac{c(c-1)}{2} & \text{if $a=0$,} \\
                \frac{-3a^{2}}{2} + \frac{2b+c-1}{2}a + \frac{c^{2}}{2} & \text{if $a > 0$.}
            \end{array} \right.
    \end{equation}
    We need to show that (\ref{eqn3:lemma: dimension in soN spN}) is strictly less than $c(c+b-1)/2$ for any $b \ge 3$, $c \ge 4$ and $0 \le a \le \min(b/2,\,c/2)$.
    The case of $a=0$ is clear and we may assume that $a > 0$.
    \begin{itemize}
        \item The case when $b = c$.
        In this case, it is easy to show that the quadratic polynomial in (\ref{eqn3:lemma: dimension in soN spN}) is maximal at $a = (2b+c-1)/6 = (3b-1)/6$ and the maximum is $(7b^{2}-2b)/8$, which is strictly less than $(2b^{2}-b)/2$ as $b \ge 3$.
        \item The case when $b \not= c$. In this case, (\ref{eqn3:lemma: dimension in soN spN}) is maximal at $a = \min(b/2,\,c/2)$.
        If $b<c$, then the maximum is $(b^{2}+2bc - 2b + 4c^{2})/8$, and
        \[
            \frac{b^{2}+2bc - 2b + 4c^{2}}{8} - \frac{c(c+b-1)}{2} = \frac{(b-2)(b-2c)}{8} < 0 \quad (\because 3 \le b < c).
        \]
        If $b> c$, then the maximum is $c(3c+4b-2)/8$, and
        \[
            \frac{c(3c+4b-2)}{8} - \frac{c(c+b-1)}{2} = \frac{c(-c+2)}{8} < 0 \quad (\because c \ge 4).
        \]
    \end{itemize}
    \end{itemize}
    Now the proof of Lemma~\ref{lemma: dimension in soN spN} is completed.
\end{proof}

\noindent Minseong Kwon

\noindent Morningside Center of Mathematics, Academy of Mathematics and Systems Science, Chinese Academy of Sciences, Beijing 100190, China

\noindent minseong@amss.ac.cn
\end{document}